\documentclass[11pt,reqno]{amsart}
\usepackage[T1]{fontenc}
\usepackage[ansinew]{inputenc}
\usepackage{amssymb,amsmath,amsthm,eucal,hyperref,mathrsfs,array,graphicx,xcolor,bbm,enumerate,wasysym,comment}
\usepackage[mathlines,displaymath,pagewise]{lineno}
\usepackage[all]{xy}
\usepackage{fullpage}
\numberwithin{equation}{section}
\usepackage{comment}
\usepackage[foot]{amsaddr}

\newtheorem{lemma}{Lemma}[section]
\newtheorem{propn}[lemma]{Proposition}
\newtheorem{thm}[lemma]{Theorem}
\newtheorem{cor}[lemma]{Corollary}
\newtheorem{defn}[lemma]{Definition}
\newtheorem{remark}[lemma]{Remark}

\DeclareMathOperator{\CR}{CR}
\DeclareMathOperator{\cov}{cov}

\DeclareMathOperator{\SLE}{SLE}

\newenvironment{proofof}[1]{{\em Proof of #1.}}{\hspace*{\fill} $\square$}

 \newcommand{\R}{\mathbb{R}}
 \newcommand{\C}{\mathbb{C}}
 \newcommand{\N}{\mathbb{N}}
 \newcommand{\Z}{\mathbb{Z}}
 \newcommand{\I}{\mathbbm{1}}
 \newcommand{\E}[1]{\mathbb{E}\left(#1\right)}
 
 \newcommand{\HH}{\mathbb{H}}

 \newcommand{\D}{\mathbb D }
 \renewcommand{\I}[1]{\mathbf 1_{\{#1\}}}
  \renewcommand{\1}{\mathbf 1}
  \renewcommand{\P}{\mathbb P}
  \newcommand{\Gf}{{G_{\HH}^{\mathrm{free}}}}
  \newcommand{\var}{{\mathrm{var}}}
  \newcommand{\Leb}{{\mathcal{L}}}

 \newcommand{\eps}{\varepsilon}

 \newcommand{\avelio}[1]{{\color{red} #1}} 
 \newcommand{\ellen}[1]{{\color{blue} #1}}
\newcommand{\rev}[1]{{#1}}

\let \epsilon \varepsilon
\renewcommand{\upsilon}{{\rev{\gamma'}}}

\title{An elementary approach to quantum length of SLE} 
\author{Ellen Powell$^*$}
\address{$^*$ Durham University, Mathematical Sciences Department, Upper Mountjoy Campus, Stockton Road, Durham, DH1 3LE, UK}

\author{Avelio Sep\'{u}lveda$^\dagger$}
\address{$^\dagger$ Universidad de Chile,  Centro de Modelamiento Matem\'atico (AFB170001), UMI-CNRS 2807, Beauchef 851, Santiago, Chile.}
\date{}

\begin{document}

\begin{abstract}
We present an elementary proof establishing equality of the right and left-sided $\sqrt{\kappa}$-quantum boundary lengths of an $\SLE_\kappa$ curve, $\kappa\in (0,4]$. We achieve this by demonstrating that the $\sqrt{\kappa}$-quantum length is equal to the $(\sqrt{\kappa}/2)$-Gaussian multiplicative chaos with reference measure given by half the conformal Minkowski content of the curve, multiplied by $2(4-\kappa)^{-1}(1-\kappa/8)^{-1}$ for $\kappa\in (0,4)$ and by $2$ for $\kappa=4$. Our proof relies on a novel ``one-sided'' approximation of the conformal Minkowski content, which is compatible with the conformal change of coordinates formula.

\end{abstract}\maketitle

\section{Introduction}

Sheffield's breakthrough paper \cite{Sheffield}, on the so-called ``quantum gravity zipper'', catalysed a cascade of developments in the study of Liouville quantum gravity (LQG) surfaces, Schramm--Loewner evolutions (SLE) and the connections between the two. In particular, it led to an extensive toolbox allowing one to construct LQG surfaces decorated with independent SLE curves by conformally welding together two (or more) independent quantum surfaces according to their boundary lengths. This idea is at the core of the landmark papers \cite{GHS,TBM1, MOT, Tutte, BG,HScardy}, which proved rigorous connections between LQG surfaces and random planar maps. It has also recently been used to establish remarkable integrability results for SLE and Liouville CFT, see for example \cite{AHSintegrability,ARSintegrability}.

An important result at the heart of Sheffield's paper \cite{Sheffield} is the following. Take a Gaussian free field $\Gamma$ 
 in a simply connected domain $D$, together with an independent chordal SLE$_\kappa$ curve $\eta$ in the same domain. Then it is possible to make sense of the $\gamma$-``quantum length'' of $\eta$ measured using $\Gamma$ when $\gamma = \sqrt{\kappa}$. Roughly this length should correspond to re-weighting a suitable occupation, or content, measure on $\eta$ by the exponential of $\gamma/2$ times $\Gamma$, \`{a} la Gaussian multiplicative chaos \cite{Kahane,RV,DS,Ber}. However, the construction of the underlying content measure was not fully understood at the time of \cite{Sheffield}. Moreover, to prove the aforementioned conformal welding results, it is essential to know that if one splits $D\setminus \eta$ into two components $D^L$ and $D^R$ (the left and right), and considers $\Gamma$ restricted to each component, then one can construct this quantum length on $\eta$ by measuring quantum boundary lengths in either $D^L$ or $D^R$. 

To explain this last sentence more precisely, suppose we  take a conformal map $\phi:D^L\to \HH$ under which portions of $\eta$ are mapped to portions of the real line. Then it is possible to construct the Gaussian multiplicative chaos measure 
on $\R$ 
corresponding to $\exp(\tfrac{\gamma}{2}\Gamma^L(x)) dx$, where 
\begin{align*}
\Gamma^L=\Gamma\circ \phi^{-1}+\left(\frac2{\gamma}+\frac{\gamma}2\right)\log\left|\left(\phi^{-1}\right)'\right|\end{align*}
{is defined the so-called ``conformal change-of-coordinates'' formula\footnote{\label{coc} 
 This is natural because, due to the construction of GMC measures using regularisation,  if we have (for example) a Gaussian free field with free boundary conditions $\Phi$ on $\HH$ and a conformal map $\psi: \HH\to \HH$, then the image under $\psi$ of the GMC measure $\tilde{\nu}^{\gamma}_{\Phi}$ (roughly $\exp(\tfrac{\gamma}{2}\Phi(x)) dx$) on $\R$ is not the GMC measure corresponding to $\Phi\circ \psi^{-1}$, but rather to $\Phi\circ\psi^{-1}+(2/\gamma+\gamma/2)\log|(\psi^{-1})'|$, \cite{DS,SheffieldWang}.}, see \cite{Sheffield}}. 
 \rev{The pull-back of this chaos measure }
 through $\phi^{-1}$ defines a measure on $\eta$ which we call the \emph{quantum length on the left side of $\eta$}.  On the other hand, we could have defined quantum length on the right side of $\eta$ analogously, by starting with $D^R$ rather than $D^L$. The crucial result from \cite{Sheffield}, which relies on an extremely delicate and abstract proof, is that the left and right quantum lengths constructed this way agree almost surely. This is then taken as the definition of \emph{the quantum length for $\eta$.}

Since the work of Sheffield \rev{(which originally appeared in 2010)}, the Minkowski content measure $\hat{m}_\eta$ on an SLE curve has been rigorously constructed in the papers \cite{LawlerSheffield,LawlerZhou,LawlerRezaei}. It was later shown in an article of Benoist, \cite[Proposition 3.3]{Benoist}, that for $\gamma\in (0,2)$, Sheffield's quantum length \emph{does} correspond to an \emph{unknown} constant times the {$(\gamma/2)$}-Gaussian multiplicative chaos for $\Gamma$ with reference measure $\hat{m}_\eta$. However, {the demonstration of this result heavily relies} on the {(complex) proofs} from \cite{Sheffield}. A similar programme has been carried out in \cite{HP} \rev{(analogous to \cite{Sheffield}, obtained by taking limits)} and then \cite{MS4} \rev{(analogous to \cite{Benoist})} in the case $\gamma=2$.

The main contribution of the present article is a direct and elementary proof that for $\gamma\in (0,2]$, Sheffield's quantum length on the left or right hand side of $\eta$ can be constructed as a Gaussian multiplicative chaos measure with reference measure given by left or right \emph{one-sided conformal} Minkowski content. These reference measures $m_\eta^L$ and $m_\eta^R$ are  constructed in the present article, and we show that they are equal. Thus, they should each be thought of as half the conformal Minkowski content of $\eta$, which is itself a deterministic multiple of the Minkowski content $\hat{m}_\eta$. In particular, this implies the equality of Sheffield's quantum lengths on the left and the right hand side of the curve, without relying on the arguments in \cite{Sheffield}.

For $\gamma\in (0,2)$ we identify the unknown constant in \cite{Benoist} explicitly for the first time. In fact, it is this explicit formula that allows us to treat the $\gamma=2$ case by taking limits. \rev{The approach of taking joint limits as $\gamma\uparrow 2$, $\kappa\to 4$ in SLE/LQG results has now been successfully employed in a number of settings: see for example \cite{CLE4AG,criticalMOT}.}

\subsection{Results} Let $\eta$ be an SLE$_\kappa$ curve from $0$ to $\infty$ in $\HH$ and \rev{$\kappa\in (0,4]$}. We denote by $\hat m_{\eta}$ the ``natural occupation  measure'' associated to $\eta$, which was shown to exist in \cite{LawlerRezaei} (denoted there by $\mu$, see Section 3.2). That is to say $\hat m_\eta$ is the Borel measure on $\overline{\HH}$ defined by
\begin{equation}\label{eq:hatm}
\hat m_{\eta} (dz) = \lim_{u\to 0} u^{-(2-d)} \1_{d(z,\eta)<u} dz 
\quad ; \quad
d:=1+\frac{\kappa}8.
\end{equation}

We first show that $\hat m_\eta$ can be constructed as a multiple of the left or right-sided conformal Minkowski content of $\eta$. More precisely, for $z\in \overline{\HH}$ we define $A^L(z)$, resp.\,$A^R(z)$, to be the event that $z$ lies in the left, resp.\,right, connected component of $\HH \backslash \eta((0,\infty))$. Then we have the following.
\begin{thm}\label{T:onesidedmink}
    Let $\eta$ be an SLE$_\kappa$ from $0$ to $\infty$ in $\HH$, with \rev{$\kappa \in (0,4]$}. Then, for $\rev{q}=L$ and $\rev{q}=R$, the measure
    \begin{align*}
    m_{\eta}^{q}(dz) := \lim_{u\to 0} u^{-(2-d)}\1_{\CR(z,\HH \backslash \eta)<u}\1_{A^{q}(z)} dz
    \end{align*}
    exists as an almost sure limit, and we have $m_\eta^L=m_\eta^R=:m_\eta=K_\kappa \hat m_{\eta} (dz)$ almost surely, where $K_\kappa$ is defined in \eqref{eq:kkappa}.
\end{thm}
In this paper, we give a proof based on the results if \cite{LawlerRezaei}, however a possible proof of this statement may follow by extending Lemma 13 of \cite{KW} to to an almost sure result.

Now, take $\Gamma$ a GFF  on $\HH$ with free boundary conditions: \rev{see Definition \ref{d:freeGFF}}.  Let $\eta$ be an SLE$_\kappa$ \rev{parametrised by half-plane capacity} with \rev{$\kappa \in (0,4]$} as above, and for $T>0$ let $f_T$ be the centred Loewner map from $\HH\backslash \eta([0,T])$ to $\HH$. Finally define
\begin{align}\label{e:GT}
\Gamma^T:= \Gamma\circ f_T^{-1} +\left(\frac{ 2}{\gamma } + \frac{\gamma}{2}\right) \log\left|\left(f_T^{-1}\right)'\left(\cdot\right)\right|\end{align}
with $\gamma=\sqrt{\kappa}$. As discussed above, the $\log|(f_T^{-1})'(\cdot)|$ term is included so that Gaussian multiplicative chaos measures for $\Gamma^T$ will correspond to images of Gaussian multiplicative chaos measures for $\Gamma$ (see footnote \ref{coc}). 

When \rev{$\kappa\in (0,4]$ and $\gamma<2\sqrt{2(1+\kappa/8})$ (in fact, we will only ever consider $\gamma\in (0,2]$)} we define the $\gamma$-Liouville measure associated to $\eta$ as the Gaussian multiplicative chaos measure for $\Gamma$ with reference measure $m_\eta$. That is, we set 
\begin{align}\label{def:numeta}
\rev{\mu_{\Gamma}^{\scriptscriptstyle{\gamma/2}}[m_\eta](dz) }:= \lim_{\epsilon \to 0} \mu_{\Gamma,\eps}^{\scriptscriptstyle{\gamma/2}}[m_
\eta](dz) := \lim_{\eps\to 0} e^{\frac{\gamma }{2 } \Gamma_\epsilon(\rev{z})}\epsilon^{\frac{\gamma^2 }{8 }} m_{\eta}(dz),
\end{align}
\rev{where $\Gamma_\epsilon$ is the (reflected) circle average of $\Gamma$ at scale $\epsilon$, see Section \ref{ss:gff}}, and the limits above are in probability; this will be justified in Section \ref{ss:gmc}. \rev{Although we will always take the chaos parameter to be equal to $\gamma/2$, we use the above notation (i.e.\,carrying $\gamma/2$ in the superscript) so as to remain consistent with the literature. 

We will also use the notation $\mu_\Gamma^{\scriptscriptstyle{\gamma/2}}[\mathfrak{m}]$ for chaos with respect to a more general reference measure $\mathfrak{m}$ later on and this will be defined exactly as above with $m_\eta$ replaced by $\mathfrak{m}$.}

When $\gamma\in (0,2)$ we can further define the Liouville measure on $\R$ \rev{(with Lebesgue reference measure)} by 
\begin{align}\label{def:nuleb}
\rev{\nu^{\gamma}_{\Gamma}(dx)} := \lim_{\epsilon \to 0}  \nu^{\gamma}_{\Gamma,\eps}(dx):=
\lim_{\eps\to 0} e^{\frac{\gamma }{2} \Gamma_\epsilon(x)}\epsilon^{\frac{\gamma^2 }{4}} dx,
\end{align}
with limits again in probability. The normalisation by $\eps^{\gamma^2/4}$ rather than $\eps^{\gamma^2/8}$ is because, when restricted to the boundary, the free boundary Gaussian free field has a covariance with two times a logarithmic divergence on the diagonal. 



Finally, when $\gamma=2$ we define the critical Liouville measure on $\R$ by the limit in probability
\begin{align}\label{e.critical_measure}
 \nu_{\Gamma}^{\mathrm{crit}}:= \lim_{\gamma\uparrow 2} \frac{\nu^{\gamma}_{\Gamma}}{2(2-\gamma)}.
\end{align}
{Thanks to the results in \cite{APS0,APS2}, in particular by \cite[Section 4.1.2]{APS2},} this is equal to the perhaps more standard definition of critical chaos:
\begin{align*}
\nu_{\Gamma}^{\mathrm{crit}}
=\lim_{\eps\to 0} \left(-\frac{1}{2}\Gamma_\eps(x)+\log(1/\eps)\right)e^{\Gamma_\eps(x)-\log(1/\eps)} dx.
\end{align*}
{The existence of all the measures above is addressed in Section \ref{ss:gmc}: see Definitions \ref{d:gmc} and \ref{d:gmcboundary}.}\\

The main result of this paper is an elementary proof of the following theorem. \rev{It is a new and much simpler proof of Sheffield's result showing equality of left and right quantum length, and Benoist's proof of equality with chaos with respect to Minkowski content. The explicit constants appearing the theorem are identified for the first time.} \rev{From here on in, we set $\R^+:=[0,\infty)$ and $\R^-:=(-\infty,0]$.}
\begin{thm} Let $\kappa \in (0,4]$, $\gamma=\sqrt{\kappa}$, and $\eta$ be an SLE$_\kappa$ curve from $0$ to $\infty$ in $\HH$. For $T>0$ let $f_T$ be the centred Loewner map from $\HH\backslash \eta([0,T])$ to $\HH$. Take $\Gamma$ an independent free boundary GFF on $\HH$ as in Definition \ref{d:freeGFF}. Then $\mu_\Gamma^{\scriptscriptstyle{\gamma/2}}[m_\eta]$ is well defined, and we have that a.s.\,for any $0<r<s\le T$, 
\begin{align*} 
\mu_{\Gamma}^{\scriptscriptstyle{\gamma/2}}[{m}_\eta] \left(\eta\left([r,s]\right)\right)= \frac{2 }{(4-\kappa)(1-\tfrac{\kappa}{8}) } \nu^{\gamma}_{\Gamma^{T}}\left(f_{T}\left(\eta\left([r,s]\right)\cap \R^-\right)\right) = \frac{2 }{(4-\kappa) (1-\tfrac{\kappa}{8})}\nu^{\gamma}_{\Gamma^{T}}\left(f_{T}\left(\eta\left([r,s]\right)\cap \R^+\right)\right).
\end{align*}
Furthermore, for $\kappa=4$ and $\gamma=2$ we have that a.s.\,for any $0<r<s<T$ 
\begin{align*}
\mu_{\Gamma}^{\scriptscriptstyle{\gamma/2}} [{m}_\eta]\left(\eta\left([r,s]\right)\right)= 2\nu_{\Gamma^{T}}^{\mathrm{crit}}\left(f_{T}\left(\eta\left([r,s]\right)\cap \R^-\right)\right) =  2\nu_{\Gamma^{T}}^{\mathrm{crit}}\left(f_{T}\left(\eta\left([r,s]\right)\cap \R^+\right)\right).
\end{align*}
\label{T:main}
\end{thm}
\begin{figure}[h!]
\includegraphics[width=0.7\textwidth]{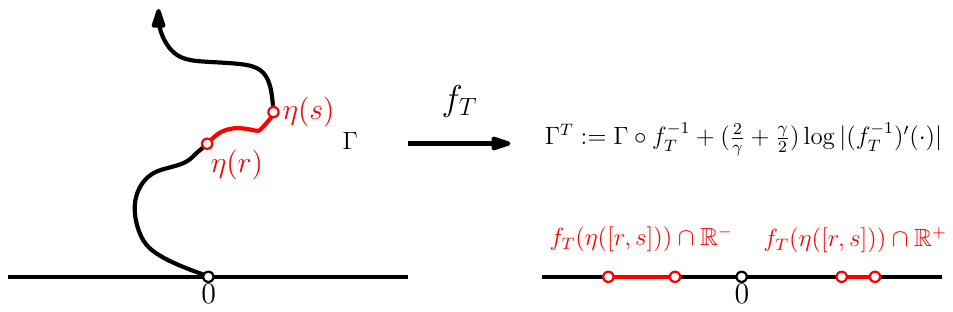}
\caption{We demonstrate that the $\gamma$-Gaussian multiplicative chaos measures for the fields $\Gamma$ of $f_{T}(\eta([r,s])\cap \R^+)$ and $f_{T}(\eta([r,s])\cap \R^-)$ coincide. Moreover, both measures are equal to a constant times the $\gamma/2$-Gaussian multiplicative chaos measure of $\eta([r,s])$, with the reference measure provided by the conformal Minkowski content of $\eta$ for the field $\Gamma^T$. }
\label{f.1}
\end{figure}

Let us relate this back to our discussion at the very beginning. In \cite{Sheffield}, Sheffield defines the left (respectively right) quantum boundary length of $\eta([r,s])$ for any $r< s\le T$, by\footnote{This is consistent in $r<s\le  T$ by footnote \ref{coc}.}
\begin{align*}\nu^{\gamma}_{\Gamma^{T}}(f_{T}(\eta([r,s])\rev{)}\cap \R^-) \quad \text{(respectively } \nu^{\gamma}_{\Gamma^{T}}(f_{T}(\eta([r,s])\rev{)}\cap \R^+)\text{)},
\end{align*}
 \rev{see Figure \ref{f.1}. When $\kappa\in (0,4)$,} the equality of these left and right quantum lengths is one of the important consequences of the remarkable but difficult argument in \cite{Sheffield}. It was extended to $\kappa=4$ in \cite{HP}, using a version of Sheffield's argument.

The equality with $\mu_{\Gamma}^{\scriptscriptstyle{\gamma/2}}[m_\eta] (\eta([r,s]))$ up to an unknown constant was shown for $\gamma<2$ in \cite{Benoist}. Here we identify the constant, and use this prove the result when $\gamma=2,\kappa=4$.  However, we believe that the elementary new proof giving equality of left and right quantum boundary lengths is the most significant contribution of our article.

\subsection{Idea of proof}\label{ss:proofidea}
The first result we prove in this paper is that the right and left conformal Minkowski contents exist and coincide (Theorem \ref{T:onesidedmink}). This follows the same path taken in \cite{LawlerRezaei}. The main difference is that we need to prove the following: conditionally on $\eta$ coming close to a given point $z\in \HH$, $\eta$ is almost as likely to go to the right or to the left of it. We achieve this thanks to the fact that, when the curve approaches the point $z$, the angle of $f_T(z)$ converges exponentially fast to a stationary measure which is symmetric.

The second result (Theorem \ref{T:main}), the core of this paper, is proved using the following strategy.
\begin{enumerate}
    \item First, we introduce a more useful approximation of the one-sided conformal Minkowski content. Namely, we show that for any $q\in\{L,R\}$,  as $\delta \to 0$,
    \begin{align*}
\sigma^{q,\delta}_\eta(dz):
= \delta\, \CR\left(z,\mathbb{H}\setminus \eta\right)^{d-2+\delta} \1_{A^{q}(z)} dz \to \left(1-\frac{\kappa}{8}\right) m_\eta(dz),
    \end{align*}
    where $\CR(z,D)$ for $D$ simply connected and $z\in D$ is the conformal radius of $z$ in $D$ (that is, $|f'(0)|$ where $f: \D\to D$ is the unique conformal isomorphism from the unit disc to $D$ with $f(0)=z$).
    The definition of this measure is inspired by \cite{ALS1,SSV}.
    \item Then, we apply the change of coordinates formula for the Liouville measure to prove that almost surely as measures in $\HH \backslash f_T(\eta[T,\infty)\cup \R)$, for every fixed $
   \delta>0$,
    \begin{align}\label{e.approximatemain}
 (f_T)_{*}\, \mu_{\Gamma}^{\scriptscriptstyle{\gamma_{\delta}/2}}\left[\sigma^{q,\delta}_\eta\right] = \frac{\delta }{\tilde \delta }\, \mu_{\Gamma^T}^{\scriptscriptstyle{\gamma_\delta/2} }\left[\1_{A^{q}_T}(2\Im(\cdot))^{\gamma_\delta^2/8}\Leb_{\R}^{\tilde \delta}\right].
    \end{align}
   Here $\tilde \delta = \delta(1-\kappa/4) + o(\delta)$ and $\gamma_\delta=\gamma+O(\delta)$  are explicit functions, $A^{q}_T$ is the set of $y\in \HH$ that belong to the right (if $q=R$) or  left (if $q=L$) of the image of $\eta$ under $f_T$,  and 
    \begin{align*}
\Leb_{\R}^{\delta}(\rev{dz})= \delta(2 \Im(\rev{z}))^{-1+\delta} d\rev{z} \to\frac{1 }{2 } \Leb_{\R} \ \ \text{ as } \delta \to 0
    \end{align*}
    where $\Leb_\R$ is Lebesgue measure on $\R$. \rev{Recall that the GMC with a general reference measure $\mathfrak{m}$ is defined exactly as in 
\eqref{def:numeta} with $\mathfrak{m}\leftrightarrow m_\eta$; see also Definition \ref{d:gmc}.}
      \item To finish the proof for $\kappa \in (0,4)$, we want to take limits as $\delta\to 0$ on both sides of \eqref{e.approximatemain}. For this we need to prove that the operation of taking Liouville measures is continuous in both the reference measure and in $\gamma$. \rev{Once this is shown, and given step (1), it is clear that the left-hand side of \eqref{e.approximatemain} converges to $(1-\kappa/8)(f_T)_*(\mu_\Gamma^{\scriptscriptstyle{\gamma/2}}[m_\eta])$. 
      
      It also holds that the right-hand side of \eqref{e.approximatemain} converges to $(2/(4-\kappa))\nu^\gamma_{\Gamma^T}$, restricted to $\R^+$ if $q=R$ or restricted to $\R^-$ if $q=L$, which immediately yields Theorem 
    \ref{T:main}. 
    
    Justifying this convergence of the right-hand side, however, is slightly more subtle. The first important observation is that $\mu_\Gamma^{\scriptscriptstyle{\gamma/2}}[\mathfrak{m}]$ is not actually continuous in $\mathfrak{m}$ if the support of $\mathfrak{m}$ is allowed to approach the real line. Instead, the object that it is truly continuous is the classical Wick exponential, which we denote by $\tilde{\mu}_\Gamma^{\scriptscriptstyle{\gamma/2}}[\mathfrak{m}]$, where the multiplicative normalisation in the approximation is $\exp(-(\gamma^2/8)\var(\Gamma_\eps(\cdot)))$, i.e.\, according to variance, rather than $\eps^{\gamma^2/8}$: see Definition \ref{d:gmc}.  
 Switching from $\mu$ to $\tilde\mu$ on the right-hand side of \eqref{e.approximatemain} exactly cancels the factor $(2\Im(\cdot))^{\gamma_\delta^2/8}$ in the reference measure, and introduces a factor $\exp(-l(\cdot)\gamma_\delta^2/8)$ where $l(z)$ is a nice continuous function arising from the variance of the free field and defined in Section \ref{ss:gff}. Combining with the convergence $$\mathbf{1}_{A_T^R} \mathcal{L}_\R^\delta\to \frac12 \mathcal{L}_\R|_{\R^+} \text{ (resp. }\mathbf{1}_{A_T^L} \mathcal{L}_\R^\delta\to \frac12 \mathcal{L}_\R|_{\R^-}\text{),}$$ the fact that $\tilde\delta/\delta\to (1-\kappa/4)$ as $\delta\to 0$, and finally the observation that $$\tilde{\mu}^{\scriptscriptstyle{\gamma/2}}_{\Gamma^T}[e^{-l(\cdot)\gamma^2/8}\mathcal{L}_\R]=\nu_{\Gamma^T}^\gamma$$  concludes the proof.}
 \item For $\kappa = 4$, it suffices to construct a nice coupling between $\Gamma$ and $\SLE_\kappa$ for a sequence of $\kappa$ converging to $4$,  so that we can take a limit as $\kappa\nearrow 4$ in the first part of Theorem \ref{T:main}. This part of the proof is the most technical part of the paper.
\end{enumerate}

The paper is organised as follows. In Section \ref{s.preliminaries}, we lay the groundwork by presenting key definitions and classical results from existing literature. Next, Section \ref{s.Minkowski_content} establishes Theorem \ref{T:confminkone} and tackles point (1). Section \ref{s.convergence_chaos} addresses point (3) by giving sufficient conditions to show convergence of chaos measures when both $\kappa$ and $\gamma$ and moving and the chaos is with respect to a free-boundary GFF. Points (1), (2) and (3) are put together in Section \ref{s.subcritical}, to show Theorem \ref{T:main} in the subcritical regime. Lastly,  Section \ref{s.critical} addresses the critical case.

\subsubsection*{Acknowledgments} \rev{We would like to thank Wendelin Werner for pointing out a different potential strategy for the proof of Theorem \ref{T:onesidedmink}. We also thank the anonymous referees for their encouraging comments and suggestions which have helped us to improve our exposition. }

The research of E.P.\,is supported by UKRI Future Leader's Fellowship MR/W008513. The research of A.S.\,is supported by grants ANID FB210005, FONDECYT regular 1240884 and ERC 101043450 Vortex, and was supported by FONDECYT iniciaci\'on de investigaci\'on N$^o$ 11200085. An important part of this work was done during a visit to CMM by E.P. and a visit to Durham by A.S.

\section{Preliminaries}\label{s.preliminaries}

\subsection{Free boundary Gaussian free field}\label{ss:gff}

We first introduce the free boundary Gaussian free field (GFF) in $\HH$ (with a specific choice of additive constant), as in \cite[Section 6]{BP}. 

\begin{defn}[Free boundary GFF]
\label{d:freeGFF}
The free boundary GFF $\Gamma$ on $\HH$ (with zero average on the upper unit semicircle) is the random centred, Gaussian field on $\HH$ whose covariance is given by 
\[ \Gf(z,w)=-\log|z-w|-\log|z-\bar{w}| +2\log|z|\I{|z|\ge 1}+2\log|w|\I{|w|\ge 1}\] for $z\ne w \in \HH$.
That is, $((\Gamma,f))_{f\in C_c^\infty(\HH)}$ is a centred Gaussian process with $$
\cov\left((\Gamma,f)(\Gamma,g)\right)=\iint_{\HH\times \HH} f(z) \Gf(z,w)g(w) dz dw,
 \quad \forall f,g \in C_c^\infty(\HH).
 $$
\end{defn}

As explained in \cite[Section 6.1]{BP}, there exists a version of $\Gamma$ that almost surely defines an element in $H^{-1}_{\mathrm{loc}}(\HH)$, \rev{the local Sobolev space of index $-1$} (and in particular, a random distribution on $\HH$). 

Moreover, let $\Gamma^\C$ be a whole plane GFF with zero average on the unit circle; that is, the centred Gaussian process indexed by $C_c^\infty(\C)$ with 
\[\cov\left((\Gamma^\C,f)(\Gamma^\C,g)\right)=-\iint_{\C\times \C} f(z) g(w) \left(\log(|z-w|)-\log(|z|)\I{|z|\ge 1}-\log(|w|)\I{|w|\ge 1}\right) dz dw \] for $f,g\in C_c^\infty(\C)$. Then, by \cite[Definition 6.24]{BP} or a direct calculation, ($\sqrt{2}$ times) the \rev{even} part of $\Gamma^\C$, defined by
$$(\Gamma,f):= (\Gamma^\C, f+f^*)/\sqrt{2} \text{ for } f\in C_c^\infty(\C)$$ 
($f^*(z):=f(\bar{z}))
$
has the law of a free boundary GFF in $\HH$ when viewed as a distribution on $\HH$.
In particular, it is known that for the whole plane GFF, one can almost surely define the circle average $\Gamma^\C_\eps(z)$ by approximation \rev{(more precisely, as $\lim_{n\to \infty}(\Gamma^\C,f_n)$ with $f_n\in C_c^\infty(\C)$ approximating \rev{$\rho_\epsilon^z$, the} uniform measure on $\partial B(z,\eps)$)} for each $z\in \C,\eps>0$. Moreover, there is a version of this process that is a.s.\;locally jointly H\"{o}lder continuous in $z,\eps$, which follows from the case of GFF with zero boundary conditions, \cite{HMP}, and the fact that in any bounded domain we can decompose the whole plane GFF as a sum of a zero boundary condition GFF and a harmonic function \cite[Corollary 6.30]{BP}.
Thus, if $\Gamma$ is a free boundary GFF coupled to a whole plane GFF $\Gamma^\C$ as above, we can define the jointly locally H\"{o}lder continuous process 
$\Gamma_\eps^z := (1/{\sqrt{2}}) (\Gamma^\C_\eps(z)+\Gamma^\C_\eps(\bar{z}))$ for $z\in \overline{\HH}, \eps>0 $
coupled to $\Gamma$. We refer to this as the \textbf{reflected circle average process of} $\Gamma$. It follows from the definition that 
\begin{equation}\label{e:varGamma} \var(\Gamma_\eps(z))=-\log\eps -\iint_{\C\times\C} \log(|w-\bar v|)\rho_\eps^z(dw)\rho_\eps^{\bar{z}}(dv)+l_\eps(z),
\end{equation}
\rev{where $l_\eps(z)\to 4\log|z|\I{|z|\ge 1}=:l(z)$} as $\eps\to 0$.
Note that by harmonicity we have 
\begin{equation}\label{e:limvarGamma} 
\var(\Gamma_\eps(z))= \begin{cases}
    -\log(\eps)-\log(2\Im(z))+l_{\eps}(z) & \text{ if }\Im(z)>\eps, \\
    -2\log(\eps)+l_\eps(z) & \text{ if } z\in \R.
\end{cases}
\end{equation} 
Furthermore for fixed $z\in \overline{\HH},\eps>0$, $\Gamma_\eps(z)$ can be obtained as the almost sure (and $L^2(\P)$) limit of $(\Gamma,f_n)$ where $(f_n)_n$ are a sequence of functions in $C_c^\infty(\HH)$ approximating uniform measure on the piecewise smooth curve $C_z^\eps$ made up of $\partial B(z,\eps)\cap \overline{\HH}$ and the reflection of $\partial B(z,\eps)\setminus\HH$ into $\HH$. In particular, $(\Gamma_\eps(z))_{\eps,z}$ is measurable with respect to $\Gamma$.

\begin{remark} We see from the above that $\var(\Gamma_1(\rev{0}))=0$. In other words, we have chosen the additive constant for the free boundary GFF so that it has \emph{average zero on the upper unit semicircle}.
\end{remark}

\subsection{Gaussian multiplicative chaos (GMC)}\label{ss:gmc}
This article concerns various Gaussian multiplicative chaos measures associated to a free boundary GFF $\Gamma$ as in Definition \ref{d:freeGFF}. Namely, measures of the form \[ \rev{\mu_\Gamma^{\scriptscriptstyle{\gamma/2}}[\mathfrak{m}](dz)  = ``  e^{\tfrac{\gamma}{2}
\Gamma(z)} \mathfrak{m}(dz) ",}
\]
where $\mathfrak{m}$ is a reference \rev{measure} and $\gamma>0$. 

\rev{From here on in, {\bf we will only consider Radon measures}. In particular, whenever we introduce a measure, it should be assumed that it is a Radon measure.} 

\rev{The above GMC measures} can be constructed by approximation \rev{(we will use the reflected circle average approximation to the free boundary GFF discussed in Section \ref{ss:gff})} due to the theory developed in \cite{Kahane,DS,RV,Ber}. In particular, we have the following result/definition. For a \rev{(random)} reference measure $\mathfrak{m}$ on $\HH$ or a collection of \rev{(random)} measures $(\mathfrak{m}_\lambda)_{\lambda\in \Lambda}$, we define the condition ($\mathcal{E}_\beta$) as follows: 
\begin{description}
\item[($\mathcal E_\beta$)]There exists $\tilde \beta > \beta$ such that for any compact set $K\subseteq \overline{\HH}$ 
\begin{align}\label{ebeta}
\sup_{\lambda \in \Lambda} \E{\iint_{K\times K} e^{\tilde \beta \Gf(z,w)} \mathfrak{m}_\lambda(dz) \mathfrak{m}_\lambda(dw)} <\infty,
\end{align}
where for a single measure $\mathfrak{m}$ we set $(\mathfrak{m}_\lambda)_{\lambda\in \Lambda}=\{\mathfrak{m}\}$.
\end{description}

\begin{defn}\label{d:gmc}
Suppose that $\Gamma$ is a free boundary GFF as in Definition \ref{d:freeGFF} with law $\mathbb{P}$. Let $\mathfrak{m}$ be a  measure on $\HH$ 
satisfying $(\mathcal{E}_\beta)$ for some $\beta<2$. Then for $\gamma<2\sqrt{2\beta}$ the limits
\begin{align}\rev{\mu_\Gamma^{\scriptscriptstyle{\gamma/2}}[\mathfrak{m}](dz):= \lim_{\eps\to 0} \mu_{\Gamma,\eps}^{\scriptscriptstyle{\gamma/2}}[\mathfrak{m}](dz):=e^{\frac{\gamma}{2} \Gamma_\eps(z)}\eps^{\gamma^2/8} \mathfrak{m}(dz)}\label{e:nu} \\ 
\rev{\tilde\mu_\Gamma^{\scriptscriptstyle{\gamma/2}}[\mathfrak{m}](dz):= \lim_{\eps\to 0} \tilde\mu_{\Gamma,\eps}^{\scriptscriptstyle{\gamma/2}}[\mathfrak{m}](dz):=e^{\frac{\gamma}{2} \Gamma_\eps(z)-\frac{\gamma^2}{8}\var(\Gamma_\eps(z))} \mathfrak{m}(dz) }\label{e:tildenu}
\end{align}
exist in probability. The topology of this convergence for $\mu_\Gamma^{\scriptscriptstyle{\gamma/2}}[\mathfrak{m}]$ is the weak topology for measures on $K$, for any compact $K\subset\HH$, and for $\tilde\mu_\Gamma^{\scriptscriptstyle{\gamma/2}}[\mathfrak{m}]$ it is the vague topology for measures on $\overline{\HH}$.
\end{defn}

Let us briefly explain why the convergence in Definition \ref{d:gmc} follows from results in the aforementioned literature. 
For any compact $K\subset \HH$, $\Gamma_\eps(z)$ is equal to the $\eps$-circle average of $\Gamma$ around $z$ for all $z\in K$ if $\eps$ is small enough. Hence \cite[Theorem 1.1]{Ber}\footnote{Here we use that the condition $(\mathcal{E}_\beta)$ on $K$ implies the dimensionality condition on $\mathfrak{m}$ assumed in \cite{Ber}.} gives the convergence in probability of $\mu_{\Gamma,\eps}^{\scriptscriptstyle{\gamma/2}}[\mathfrak{m}]$ to a limit $\mu_{\Gamma}^{{\scriptscriptstyle{\gamma/2}}}[\mathfrak{m}]$ (respectively $\tilde\mu_{\Gamma,\eps}^{{ \scriptscriptstyle{\gamma/2}}}[\mathfrak{m}]$ to a limit $\tilde\mu_{\Gamma}^{\scriptscriptstyle{ \gamma/2}}[\mathfrak{m}]$)
with respect to the weak topology of measures on any such $K$; and the convergence of the total mass $\mu_{\Gamma,\eps}^{\scriptscriptstyle{ \gamma/2}}[\mathfrak{m}](K)$ to $\mu_{\Gamma}^{\scriptscriptstyle{ \gamma/2}}[\mathfrak{m}](K)$ (respectively $\tilde\mu_{\Gamma,\eps}^{\scriptscriptstyle{ \gamma/2}}[\mathfrak{m}](K)$ to $\tilde\mu_{\Gamma}^{\scriptscriptstyle{ \gamma/2}}[\mathfrak{m}](K)$).  In particular, this justifies the definition of $\mu_\Gamma^{\scriptscriptstyle{\gamma/2}}[\mathfrak{m}]$.

For the convergence of $\tilde{\mu}_\Gamma^{\scriptscriptstyle{\gamma/2}}[\mathfrak{m}]$ with respect to the vague topology for measures on $\overline{\HH}$ , we fix $n\in \N$ and invoke the above with $K=D_n\setminus \{w: \Im(w)\rev{<} \delta\}$ for each $\delta>0$, where $D_n:=B(0,2^{n})\cap \overline{\HH}$. We also observe that for any $\eps,\delta>0$
\[\mathbb{E}\left(\tilde\mu_{\Gamma,\eps}^{\scriptscriptstyle{\gamma/2}}[\mathfrak{m}]\left(D_n\cap \{w:\Im(w)\leq \delta\}\right)\right)= \mathfrak{m}\left(D_n\cap \{w:\Im(w)\le \delta\}\right),\] where the right-hand side converges to $0$ as $\delta\to 0$ since $\mathfrak{m}$ is a measure on $\HH$.  This implies\footnote{{See \cite[Section 6]{Ber} for the detailed argument.}} that $\tilde\nu_{\Gamma,\eps}^{\mathfrak{m},\gamma}(D_n)\to \sup_{\delta>0} 
\tilde\mu_\Gamma^{\scriptscriptstyle{\gamma/2}}[\mathfrak{m}](D_n\setminus \{w: \Im(w)\rev{<} \delta\})<\infty$ in probability and $L^1(\mathbb{P})$ as $\eps\to 0$. This in turn implies that \rev{$\tilde \mu_{\Gamma,\eps}^{\scriptscriptstyle{\gamma/2}}[\mathfrak{m}]\to \tilde \mu_{\Gamma}^{\scriptscriptstyle{\gamma/2}}[\mathfrak{m}]$} in probability with respect to the weak convergence of measures on $D_n$. 
Since $n$ was aribtrary we get the convergence with respect to the vague topology for measures on $\overline{\HH}$. 
\medskip

We also have the following ``change-of-coordinates'' formula, that describes how the measures $\mu_\Gamma^{\scriptscriptstyle{\gamma/2}}[\mathfrak{m}]$ behave when applying a conformal map. \rev{Recall that $f_*\mathfrak{m}$, the pushforward of a measure $\mathfrak{m}$ by a function $f$, is defined by $f_*\mathfrak{m}(B)=\mathfrak{m}(f^{-1}(B))$ for all $B$.}

\begin{lemma}\label{lem:coc}
Suppose that $H\subset \HH$ is a simply connected domain, $f:H\to \HH$ is a conformal isomorphism, and $\mathfrak{m}$ is a measure on $\HH$ satisfying $(\mathcal{E}_\beta)$, $\gamma<2\sqrt{2\beta}$. Then for any $K\subset \HH$ compact, we have
\[ f_* (\mu_\Gamma^{\scriptscriptstyle{\gamma/2}}[\mathfrak{m}]) = \mu_{\Gamma\circ f^{-1}}^{\scriptscriptstyle{\gamma/2}}[\mathfrak{m}'] \text{ as measures on $f(K\cap H)$} \text{ almost surely,} \]
where $\mathfrak{m}'(dz)=|(f^{-1})'(z)|^{\gamma^2/8}f_*\mathfrak{m}(dz)$, and $\mu_{\Gamma\circ f^{-1}}^{\scriptscriptstyle{\gamma/2}}[\mathfrak{m}']$ is well defined as the limit in \eqref{e:nu} on any such set ${f(K\cap H)}$.
\end{lemma}

\begin{remark}
Applying the above when $\mathfrak{m}=\mathrm{Leb}$ is Lebesgue measure, we obtain that   
\[ f_* \left(\mu_\Gamma^{\scriptscriptstyle{\gamma/2}}[\mathrm{Leb}]\right) = \mu_{\Gamma\circ f^{-1}}^{\scriptscriptstyle{\gamma/2}}[|(f^{-1})'(\cdot)|^{\gamma^2/8}\mathrm{Leb}]=\mu^{\scriptscriptstyle{\gamma/2}}_{\Gamma'}[\mathrm{Leb}]  \]
 almost surely, where $\Gamma'=\Gamma\circ f^{-1}+(2/(\gamma/2)+(\gamma/2)/2)\log|(f^{-1})'|$. This is how the ``change-of-coordinates'' formula for Liouville quntum gravity is usually stated in the literature, when $\Gamma$ is a zero boundary condition GFF on $H\subset \HH$, and we use the parameter $\gamma/2$ rather than $\gamma$. We, however, will need the statement for a free boundary GFF restricted to $\HH$ and a general reference measure $\mathfrak{m}$.
\end{remark}

\begin{proof}
We first prove the lemma when $\Gamma$ is replaced by a Gaussian free field $\Gamma_0$ with zero boundary conditions in $H$. 
In this case, we have, \cite[Section 5]{Ber}, that $\mu_{\Gamma_0}^{\scriptscriptstyle{\gamma/2}}[\mathfrak{m}]$ can also be defined as 	the limit in probability of \[
\mu_{{\Gamma_0},N}^{\scriptscriptstyle{\gamma/2}}[\mathfrak{m}]:=
e^{\tfrac{\gamma}2{\Gamma_0^N}(z) - \tfrac{\gamma^2}{8} \var ({\Gamma_0^N}(z))} \CR(z,H)^{\tfrac{\gamma^2}{8}} \mathfrak{m}(dz),
\] where
${\Gamma_0^N}(z) := \sum_{i=1}^N ({\Gamma_0},f_i) f_i$  for $f_n$ an orthonormal basis of $H_0^1(\rev{H})$, \rev{the closure of compactly supported functions in $H$ with respect to the Sobolev norm of index one}. We then have that 
\[ f_*(\mu_{{\Gamma_0},N}^{\scriptscriptstyle{\gamma/2}}[\mathfrak{m}])=
e^{\tfrac{\gamma}2({\Gamma_0^N}\circ f^{-1})(z) - \tfrac{\gamma^2}{8} \var ({\Gamma_0^N}\circ f^{-1}(z))} \CR(f^{-1}(z),H)^{\tfrac{\gamma^2}{8}} f_*\mathfrak{m}(dz),
\]
and since $\CR(f^{-1}(z),H)=\CR(z,\HH)|(f^{-1})'(z)|$, while $({\Gamma_0^N}\circ f^{-1})$ has the law of $\sum_{i=1}^N ({\Gamma'_0},e_i)e_i$ for ${\Gamma'_0}$ a zero boundary GFF in $\HH$, we can apply the result of \cite[Section 5]{Ber} again and let $N\to \infty$ to see that $f_*(\mu_{\Gamma_0}^{\scriptscriptstyle{\gamma/2}}[\mathfrak{m}])=\mu_{\Gamma_0\circ f^{-1}}^{\scriptscriptstyle{\gamma/2}}[\mathfrak{m}']$ as measures on $K\cap H$.

Now we move onto the general statement of the lemma. Since the free boundary GFF $\Gamma$ is absolutely continuous with respect to $\Gamma_0$ on 
\begin{align*}
\rev{(K\cap H)_\delta:=K\cap H\setminus \{z:d(z,\HH\setminus K \cup \HH\setminus H)<\delta\}}
\end{align*}
for any $\delta>0$, the result follows on $(K\cap H)_\delta$ for every $\delta>0$. On the other hand, $\mu_\Gamma^{\scriptscriptstyle{\gamma/2}}[\mathfrak{m}]$ is well-defined as a measure on $K$ and $\mu_\Gamma^{\scriptscriptstyle{\gamma/2}}[\mathfrak{m}](K\cap H)=\sup_{\delta>0} \mu_\Gamma^{\scriptscriptstyle{\gamma/2}}[\mathfrak{m}]((K\cap H)_\delta)$. Hence, $\mu_{\Gamma\circ f^{-1}}^{\scriptscriptstyle{\gamma/2}}[\mathfrak{m}']$ extends to a measure on $K\cap H$ which agrees with $f_*\mu_{\Gamma}^{\scriptscriptstyle{\gamma/2}}[\mathfrak{m}]$ almost surely. 
\end{proof} 

When $x\in \R$, $\Gamma_\eps(x)$ is simply the upper $\eps$ semicircle average of $\Gamma$ about $x$ and in particular we have $\var(\Gamma_\eps(x))=2\log(1/\eps)+l_\eps(x)$ as explained before (where $l_\eps(x)\to (4\log|x|)\I{|x|\ge 1}=l(x)$ as $\eps\to 0$). This brings us to the following definition. 
\begin{defn}\label{d:gmcboundary}
Suppose that $\Gamma$ is a free boundary GFF as in Definition \ref{d:freeGFF} with law $\mathbb{P}$. Then for $\gamma\in (0,2)$ we can define the Liouville measure (with Lebesgue reference measure) on $\R$ by 
\begin{align*}
\nu^{\gamma}_{\Gamma}(dx) := \lim_{\epsilon \to 0} \nu^{\gamma}_{\Gamma,\eps}(dx):=
\lim_{\eps\to 0} e^{\frac{\gamma }{2} \Gamma_\epsilon(x)}\epsilon^{\frac{\gamma^2 }{4}}dx.
\end{align*}
where the limit exists in probability with respect to the vague topology for measures on $\R$. 
When $\gamma=2$ we define the critical Liouville measure on $\R$ by 
\begin{align*}
{\nu_{\Gamma}^{\mathrm{crit}}(dx)}:= \lim_{\gamma\uparrow 2} \frac{\nu^{\gamma}_{\Gamma}}{2(2-\gamma)}
\end{align*}
\noindent where the limit exist in probability with respect to the vague topology on measures on $\R$.
\end{defn}
\rev{Notice that 
\begin{equation}\nu^\gamma_\Gamma=\tilde\mu_\Gamma^{\scriptscriptstyle{\gamma/2}}[e^{{(\gamma^2/8)}l(\cdot)}\mathcal{L}_\R],
\label{eq.relmunu}
\end{equation}
}
\rev{for $\mathcal{L}_\R$ the Lebesuge measure on $\R$,} and so the convergence when $\gamma\in [0,2)$ follows from the more general setting. See also\cite[Theorem 6.36]{BP} for a direct argument. For the existence of the limit as $\gamma\uparrow 2$, see \cite[Theorem 1.1 and Section 4]{APS2}. 
\medskip

\subsection{SLE} \label{SS:SLE}

In this section we give a very brief introduction to Schramm--Loewner evolutions (SLE). For much more detailed expositions see  \cite{LawC,SLEnotes,Kem17}. 

In this article, we will consider \textit{chordal} $\SLE_\kappa$ with $\kappa\in (0,4]$. {This is a random curve $\eta$ from $\R^+$ to $\overline{\HH}$ }with $\eta(0)=0$ and $\lim_{t\to \infty} \eta(t) = \infty$ which is almost surely simple and does not hit the real line except at time $0$. It can be described through the chordal (half-planar) Loewner evolution. That is, we start with $(W_t)_{t\ge 0}=(\sqrt{\kappa}B_t)_{t\ge 0}$ where $B_t$ is a standard linear Brownian motion, and define the associated Loewner flow or collection of Loewner maps $(g_t)_{t\ge 0}$ by solving the ordinary differential equations
\begin{equation}\label{E:LE}
\partial_t g_t(z) = \frac{2}{g_t(z)-\rev{W}_t}; \quad g_0(z)=z; \quad t\le T(z):=\sup\{t: g_t(z)\rev{\ne W_t}\}
\end{equation}
for every $z\in \HH$. Then for each $t\ge 0$, $g_t$ is a conformal isomorphism from $H_t:=\HH\setminus \{z: T(z)\le t\}$ to $\HH$, and one can prove that when $\kappa\in (0,4]$, \cite{RohdeSchramm}, we have $(\HH\setminus H_t)_{t\ge 0}=(\eta([0,t]))_{t\ge 0}$ where $\eta$ is an almost surely simple curve as described above.

This curve $\eta$ defines our $\SLE_\kappa$ for $\kappa\in (0,4]$. We call $(W_t)_{t\ge 0}$ the driving function of $\eta$, and also sometimes use the notation $f_t:=g_t-W_t$ for the centred Loewner map. For every $t$, the half-plane capacity of $\eta([0,t])$ is equal to $2t$, which means that the unique conformal map $G:\HH\setminus \eta([0,t])\to \HH$ satisfying \footnote{In the setting of Loewner evolution, we always have that this unique conformal map $G$ is equal to $g_t$, and it always holds that $g_t(z)=z+2t/z+O(|z|^{-2})$ as $|z|\to \infty$.} $G(z)= z+o(1)$ as $z\to \infty$  in fact satisfies that $G(z)=z+2t/z + O(|z|^{-2})$ as $|z|\to \infty$. This parametrisation of the curve is often referred to as the half-plane capacity parametrisation.

A curve $\eta$ between boundary points $a$ and $b$ in a domain $D$ is said to be an $\SLE_\kappa$ from $a$ to $b$ if it is the image of an $\SLE_\kappa$ in $\HH$ from $0$ to $\infty$, under a conformal map from $ \HH$ to $ D$ mapping $0$ to $a$ and $\infty$ to $b$. 

An important quantity in this article, in particular when it comes to defining conformal Minkowski content for $\eta$, is $\SLE_\kappa$ Green's function. This function captures the asymptotic probability that an $\SLE_\kappa$ curve comes near to a point $z$. 

\begin{defn}[SLE$_\kappa$ Green's function]
\label{d:GreenSLE}  
\rev{Recall that $d=1+(\kappa/8)$ is the Hausdorff dimension of SLE$_\kappa$.} We define 
\[G(z):=c_\kappa^*\left(2\Im(z)\right)^{d-2}\left(\sin\left(\arg z\right)\right)^{8/\kappa-1}\]
where $c_\kappa^*=2(\int_0^\pi \sin(x)^{8/\kappa} \, dx)^{-1}$.
\end{defn}
As shown in, e.g. \cite{LawlerWerness}, we have
\[
G(z)=\lim_{u\to 0} u^{-(2-d)}\mathbb{P}(\CR(z,\HH\setminus \eta)<u)
\]
where under $\mathbb{P}$, $\eta$ has the law of SLE$_\kappa$ from $0$ to $\infty$ in $\HH$.

\subsection{SLE-GFF coupling} For our proof, the following remarkable result which describes how the law of a free boundary Gaussian free field transforms after applying the centred Loewner map for an independent $\SLE_\kappa$, will be crucial. This was first observed by Sheffield in \cite{Sheffield}, and although the result is surprising and powerful, its proof relies only on elementary stochastic calculus. (The reader can also consult \cite[Chapter 8]{BP} for an introduction for neophytes).

\begin{thm}[Theorem 1.2 of \cite{Sheffield}]\label{T:gffinvariance}
Let $\Gamma$ be a free boundary Gaussian free field as in Definition \ref{d:freeGFF}, and for $\gamma\in (0,2)$, set 
$$ \rev{\hat \Gamma}^\gamma=\Gamma+\frac{2}{\gamma}\log|(\cdot)|.
$$
Let $\kappa=\gamma^2$ and let $\eta$ be an independent $\SLE_\kappa$ curve from $0$ to $\infty$ in $\HH$. 

Then 
$$ \rev{\hat\Gamma}^\gamma\circ f_T^{-1} + \left(\frac{2}{\gamma}+\frac{\gamma}{2}\right)\log \left|\left(f_T^{-1}\right)'\left(\cdot\right)\right|+C\overset{(d)}{=} \rev{\hat \Gamma}^\gamma$$
where $C$ is a random constant, chosen such that the upper unit semicircle average of the random generalised function $\rev{\hat \Gamma}^\gamma\circ f_T^{-1} + (\frac{2}{\gamma}+\frac{\gamma}{2})\log |(f_T^{-1})'(\cdot)|+C-\frac{2}{\gamma}\log|(\cdot)|$ is equal to $0$.
\end{thm}

This will be important for our proof, because we will need to use GMC results (more precisely Proposition \ref{p.convergence_liouville}) for the field $\Gamma^T$ defined in \eqref{e:GT} and 
appearing in Theorem \ref{T:main}. Theorem \ref{T:gffinvariance} above implies that this field is absolutely continuous with respect to a free boundary GFF \rev{(outside any neighbourhood of $0$)}, which allows us to apply the aforementioned GMC results.

\section{Conformal Minkowski content on $\eta$}\label{s.Minkowski_content}

Let $\kappa=\gamma^2\le 4$.
and let $\eta$ be an SLE$_\kappa$ in $\HH$ from $0$ to $\infty$. Recall the definitions of $g_t$, $f_t$, $W_t, \CR(\cdot, \cdot)$ from Section \ref{SS:SLE}.

\subsection{Existence and properties of one-sided conformal Minkowski content: proof of Theorem \ref{T:onesidedmink}}\label{S:onesidedmink}
 Let $\HH^L$ (resp. $\HH^R$) be \rev{the connected component of $\HH\backslash \eta$ to left (resp. right) side of $\eta$}.
Write $A^L(z)$ for the event that $z$ lies on the left-hand side of $\eta([0,\infty))$ and $A^R(z)$ for the event that it lies on the right hand side.  Let $d=d(\kappa)\in (1,3/2]$ be as in \eqref{eq:hatm}. 

\begin{thm} \label{T:confminkone}
For $i=L,R$ and $u>0$ let \[m_\eta^{q,u}(dz):=u^{-(2-d)}\I{\CR(z,\HH\setminus \eta)<u}\I{A^q(z)} \, dz.\]
There exists a limiting measure 
$m^q_\eta$, which we call the one-sided conformal Minkowski content of $\eta$, such that 
\[
m_\eta^{q,u} \to m_\eta^q \text{ a.s.}
\]
(in the sense of vague convergence of measures on $\overline{\HH}$) as $u\to 0$. 
Moreover, for $q=L,R$, $t>0$ and
bounded $D\subset \HH$ with $d(D,\R)>0$, it holds almost surely that:
\begin{align}
& m_\eta^q(\eta([0,t])) \text{ is } \rev{\sigma( \eta(s); s\le t)} \text{-measurable}, \text{\& defines an acontinuous, increasing process in } t; \label{eq:minkchar1}\\ 
& \mathbb{E}\left(m_\eta^q(D)\, | \, \rev{\sigma(\eta(s); s\le t)}\right)=m_\eta^q\left(D\cap\eta([0,t])\right)+\frac{1}{2} \int_{D\setminus \eta([0,t])} |g_t'(z)|^{2-d} G(g_t(z)-W_t) \, dz   \label{eq:minkchar2}
\end{align}
\end{thm}

\begin{cor}\label{C:confminkone}
With probability one, 
\[ m_\eta^L=m_\eta^R=:m_\eta\]
as measures on $\overline{\HH}$. Moreover, $m_\eta=K_\kappa\hat{m}_\eta$ where $\hat{m}_\eta$ is as defined in \eqref{eq:hatm} and 
\begin{equation}\label{eq:kkappa}
    2 K_\kappa:=\lim_{u\to 0} u^{-(2-d)}\mathbb{P}(d(\eta,0)<u)
\end{equation} where under $\mathbb{P}$, $\eta$ has the law of SLE$_\kappa$ in $\D$ from $1$ to $e^{2i\theta}$ for $\theta\in [0,\pi)$ random with density $\frac{1}{2}\sin(\theta) \, d\theta$. This limit exists (see for example \cite[Proposition 4.1]{LawlerRezaei}) but to our knowledge there is no explicit numerical value.
\end{cor}

\begin{proof}[Proof of Theorem \ref{T:onesidedmink}]
This follows from Theorem \ref{T:confminkone} and Corollary \ref{C:confminkone}.
\end{proof}

\begin{proof}[\rev{Proof of Corollary \ref{C:confminkone} given Theorem \ref{T:confminkone}}]
It is shown in \cite[Section 3: \rev{immediately after (3.2)}]{LawlerSheffield} that \eqref{eq:minkchar1},\eqref{eq:minkchar2} uniquely characterise $m_\eta^q$, which gives that $m_\eta^L=m_\eta^R=:m_\eta$ almost surely. The measure $\hat{m}_\eta$ \rev{defined in \eqref{eq:hatm} and} constructed in \cite{LawlerRezaei} also satisfies \eqref{eq:minkchar1},\eqref{eq:minkchar2}, but with the factor $1/2$ on the right-hand side of \eqref{eq:minkchar2} replaced by $2K_\kappa$ (and $m_\eta^q$ replaced by $\hat{m}_\eta$ everywhere); see \cite[Theorems 1.1 and 4.2]{LawlerRezaei}. This implies that $K_\kappa \hat{m}_\eta=(1/2)2K_\kappa \hat{m}_\eta$ satisfies \eqref{eq:minkchar1},\eqref{eq:minkchar2} and so must be equal to $m_\eta$ as desired.
\end{proof}

The crucial ingredient for the existence of one-sided conformal Minkowski content is the following sharp one-point estimate. For $z\in \HH$, let $T_r(z)$ be the first time that the conformal radius of $\HH\setminus \eta$, seen from $z$, reaches $e^{-r}$.

\begin{lemma}\label{L:onepointmain} For any $\kappa
\in (0,4]$ there exist $C,\alpha,r_0\in (0,\infty)$ such that
\begin{equation}\label{E:greensmain}
\left|1-\frac{2e^{r(2-d)}}{G(z)}\mathbb{P}\left(T_{r}(z)<\infty,A^q(z)\right)\right|\le C\left(\frac{e^{-r}}{2\Im(z)}\right)^\alpha
\end{equation}
for all $z\in \HH$, $q=L,R$ and $r\ge r_0-\log(2\Im(z))$.
\end{lemma}

\begin{proof}
Applying the conformal map $w\mapsto e^{-2i\arg(\rev{z})}\frac{w-z}{w-\bar{z}}$ from $\HH$ to $\D$ that sends $z\mapsto 0$, $0\mapsto 1$ and $\infty\mapsto e^{2i\theta_0}$, \rev{where $\theta_0=\pi-\arg(z)$}, we can rewrite 
\[ \mathbb{P}\left(T_r(z)<\infty, A^q(z)\right) = \mathbb{P}_{\theta_0}\left(\mathcal{T}>r+\log(2\Im(z)), \,  \mathcal{A}^q\right) \]
where: $\mathbb{P}_{\theta_0}$ denotes the law of an SLE$_\kappa$ $\rev{\eta}$ in $\D$ from $1$ to $e^{2i\theta_0}$, parameterised so that $\CR(0,D_t):=\CR(0,\D\setminus \rev{\eta}([0,t])\rev{)}=e^{-t}$ for all $t$; $\mathcal{T}$ is its total lifetime; and $\mathcal{A}^q$ is the event that $0$ lies on the left (for $q=L$) or right (for $q=R)$ of $\rev{\eta}([0,\mathcal{T}])$. For $t>0$ let $g_t:D_t\to \mathbb{D}$ be the unique conformal map sending $0\mapsto 0$ and \rev{$\eta(t)\mapsto 1$}. Set $g_t(e^{2i\theta_0})=:e^{2i\theta_t}$. Then applying the Markov property of SLE at time $r':=r+\log(2\Im(z))$ yields that \[
\mathbb{P}_{\theta_0}\left(\mathcal{T}>r+\log(2\Im(z)), \,  \mathcal{A}^q\right) = \mathbb{E}_{\theta_0}\left(\1_{\mathcal{T}>{r'}} \mathbb{P}_{\theta_{r'}}(\mathcal{A}^q)\right).\]
Next, we appeal to the fact that  $M_t:=\1_{\mathcal{T}>t}e^{t(2-d)}\sin(\theta_t)^{8/\kappa-1}$ is a martingale under $\mathbb{P}_{\theta_0}$. 
(This is easily verified using stochastic calculus, or see, for example, \rev{\cite[Section 2.3: third and fourth displayed equations]{LawlerWerness}}.  Writing $\mathbb{P}_{\theta_0}^*$ for the law $\mathbb{P}_{\theta_0}$ weighted by $M_t/M_0$ (this is the law of \emph{two-sided radial SLE$_\kappa$ going through the origin and stopped when reaching the origin}) we have that $\mathcal{T}>r'$ $\mathbb{P}^*_{\theta_0}$-a.s., and can thus write the right-hand side above as 
\[
e^{(d-2)r'}\sin(\theta_0)^{8/\kappa-1} \mathbb{E}_{\theta_0}^*\left(\sin(\theta_{\rev{r'}})^{1-8/\kappa}\mathbb{P}_{\theta_{r'}}(\mathcal{A}^q)\right)=e^{(d-2)r'}\sin(\theta_0)^{8/\kappa-1} \int_0^{\pi} \phi_{r'}(\theta;\theta_0) \sin(\theta)^{1-8/\kappa}\mathbb{P}_\theta(\mathcal{A}^q) d\theta\]
where $\phi_t(\theta;\theta_0)$ is the density at time $t$ of $\theta_t$ under $\mathbb{P}^*_{\theta_0}$. In conclusion (combining the three displayed equations above) we have that 
\begin{equation}\label{eq:onepointrewrite}
\left|1-\frac{2e^{r(2-d)}}{G(z)}\mathbb{P}\left(T_{r}(z)<\infty,A^q(z)\right)\right|=\left| 1- \frac{2}{\rev{c^*_\kappa}}\int_0^{\pi} \phi_{r'}(\theta,\theta_0)\sin(\theta)^{1-8/\kappa}\mathbb{P}_\theta(\mathcal{A}^q) d\theta  \right|.
\end{equation} 

Now, by e.g.\,\cite[(40)]{LawlerRezaei}, we have that $\phi_t$ converges to an equilibrium distribution exponentially fast and uniformly in $\theta_0$. More precisely, for some $C',\alpha,t_0\in (0,\infty)$ depending only on $\kappa$
\[ \left |1-\frac{\phi_t(\theta;\theta_0)}{(c_*/2) \sin(\theta)^{8/\kappa}}\right| \le C'e^{-\alpha t} \quad \forall t\ge t_0, \theta\in [0,\pi].\]
Applying this with $t=r'$ in \eqref{eq:onepointrewrite} and using that \begin{align*}
\int_0^\pi \sin(\theta)\mathbb{P}_\theta(\mathcal{A}^L)=\int_0^\pi \sin(\theta)\mathbb{P}_\theta(\mathcal{A}^R)=\frac12 \int_0^\pi \sin(\theta) d\theta = 1\end{align*} by symmetry of SLE$_\kappa$, we obtain the statement of the lemma. 
\end{proof}

The proof of Theorem \ref{T:confminkone} given Lemma \ref{L:onepointmain} is contained in the Appendix, since it essentially follows the same reasoning as \cite{LawlerRezaei}.

\subsection{More useful approximations for conformal Minkowski content and Lebesgue measure} The objective of this section is to find a new approximation for the one-sided conformal Minkowski content of an SLE curve, as discussed in Section \ref{ss:proofidea}. In order to obtain this approximation, we fix $0<\delta<\gamma/2$, $q=L$ or $q=R$, and define the approximate measure on $\HH\setminus \eta$:
\[ {\sigma^{q,\delta}_\eta}(dz):
= \delta\, \CR(z,\mathbb{H}\setminus \eta)^{\rev{d(\kappa)}-2+\delta} \1_{A^q(z)} dz=\delta\, \CR(z,\mathbb{H}\setminus \eta)^{-1+(\kappa/8)+\delta} \1_{A^q(z)} dz.\]

The main result of this subsection is the following, which closely follows Lemma 6.3 of \cite{SSV}.
\begin{propn}\label{prop::convergence_better_approximation}
    For any $\kappa\leq 4$, let $\eta$ be an SLE$_\kappa$ from $0$ to $\infty$ in $\HH$. Then we have that a.s.\,for $q\in \{L,R\}$,
    \begin{align*}
    \sigma^{q,\delta}_\eta (dz) \to \left(1-\frac{\kappa}{8}\right)\, m_{\eta} (dz), \text{ as }\delta \to 0
    \end{align*}
    with respect to the vague topology for measures on $\overline{\HH}$. Recall that $m_\eta=m_\eta^L=m_\eta^R$, as described in Corollary \ref{C:confminkone}.
\end{propn}
\rev{Before proving the proposition, let us note that the only assumption we rely on is the existence of the one-sided conformal Minkowski content. The proof itself is entirely deterministic. Moreover, this approach to Minkowski content can be extended, using the same argument, to other gauge functions in the definition of conformal Minkowski content, as in Theorem 5.1 of \cite{ALS1}, or even to variants where the conformal radius is replaced by the distance to the set itself.}

\begin{proof} We start by noting that when $\CR(z,\HH\setminus \eta)<1$, we can write
\[ \CR(z,\HH\setminus \eta)^{-1+\kappa/8+\delta} = 1- \left(-1+\kappa/8+\delta\right)\int_0^1 \1_{\CR(z,\HH\setminus \eta)<t} t^{-1+\kappa/8+\delta-1} \, dt. \]
Let $f$ be a positive, bounded, compactly supported function on $\overline{\HH}$ (which can be random); without loss of generality we assume that all $z$ in the support of $f$ have $\CR(z,\HH
{\setminus \eta})<1$, as it is clear that the limiting measure has support in $\eta$. We compute, taking $q=L$ without loss of generality,
\begin{align} \nonumber\sigma_\eta^{L,\delta}(f) & = \delta \int_{\HH^L} f(z) \, dz + \delta (1-\kappa/8-\delta) \int_{\HH^L}\int_0^1 f(z)\1_{\CR(z,\HH\setminus \eta)<\rev{u}} {\rev{u}}^{-1+\kappa/8+\delta-1} \, \rev{du dz} \\
\label{e.basic_min}& = \delta \int_{\HH^L} f(z) \, dz + \delta (1-\kappa/8-\delta) \int_0^1 m_\eta^{L,t}(f) {\rev{u}}^{-1+\delta} \, d
\rev{u},
\end{align}
where we applied Fubini and used the definition of $m_\eta^{L,\rev{u}}$ in the second line.
Making the change of variables $\rev{u}=e^{-s/\delta}$ we get 
\begin{align}\label{e.decomposition_min}
    \delta  \int_0^1 m_\eta^{L,\rev{u}}(f) \rev{u}^{-1+\delta} \, d\rev{u} =  \int_0^\infty e^{-s} m_\eta^{L,e^{-s/\delta}}(f)\, ds.
\end{align}
Taking $\delta\to 0$ and applying Theorem \ref{T:confminkone} along with dominated convergence, we have that 
\[ \lim_{\delta\to 0} \int_0^\infty e^{-s} m_\eta^{L,e^{-s/\delta}}(f)\, \rev{ds} = m_\eta(f).\]
\end{proof}

We will also use the following approximation to Lebesgue measure on the real line. For $\delta>0$ we define the measure on $\HH$ 
\begin{equation}\label{Lebdelta}
\Leb_\R^\delta(dz):= \delta (2\Im(z))^{-1+\delta} dz
\end{equation}
(where $dz$ is Lebesgue measure on $\HH$).

\begin{propn} \label{prop.convergenceLebesgue} Let $\Leb_\R$ denote Lebesgue measure on $\R$. Then
    \begin{align*}
    \Leb^\delta_\R  \to \frac{1}{2}\Leb_\R , \text{ as }\delta \to 0
    \end{align*}
    with respect to the vague topology for measures on $\overline{\HH}$.
\end{propn}
This proposition can be obtained by interpreting $\R\subseteq \C$ as an $\SLE_0$ and then using Proposition \ref{prop::convergence_better_approximation}. However, for clarity, we provide a self contained proof.

\begin{proof}
Take $f\in C_0^\infty(\bar \HH)$, without loss of generality we assume that the support of \rev{$f$} has imaginary part smaller than $1$. Let us compute
\begin{align*}
\int f \Leb^{\delta}_\R(dz) &= 2^{-1+\delta}\delta \int_\R \int_0^1 f(x+iy)  y^{\delta -1} dy dx\\ 
&= 2^{-1+\delta} \int_\R \int_0^1  f( x+ i u^{1/\delta}) du dx,
\end{align*}
where, in the last equation, we make the change of variables $u=y^{\delta}$. By bounded convergence theorem the above converges, as $\delta \to 0$ to
\begin{align*}
2^{-1} \int_\R f(x + 0i)dx,
\end{align*}
which finishes the proof.
\end{proof}
\subsection{Bounds on the energy}  

Recall the definition $(\mathcal{E}_\beta)$, \eqref{ebeta}, for a collection of measures (or a single measure). The following remark will be useful in this section.

\begin{remark}\label{r:ebeta}
If a sequence of measures $(\mathfrak{m}_n)_{n\in \N}$ satisfies ($\mathcal E_{\beta}$) and a.s. $\mathfrak{m}_n\to \mathfrak{m}$ \rev{for the vague topology}, then 
\[
\E{\int e^{\tilde \beta \Gf(z,w)} d\mathfrak{m}(dz) d\mathfrak{m}(dw)} <\infty,
\] 
where $\tilde{\beta}>\beta$ is as in the definition of $(\mathcal{E}_\beta)$. This follows from the convergence of the integral of $e^{\alpha \Gf(z,w)}\wedge M$, plus Fatou's Lemma when taking $M\to \infty$. \end{remark}

We first prove that the approximation to Lebesgue measure satisfies the energy condition for any \rev{$\beta<1/2$.}

\begin{propn}\label{prop.EbLebesgue}
    We have that $(\Leb_\R^\delta)_{0<\delta<1}$ \rev{from \eqref{Lebdelta}} satisfies the energy condition $(\mathcal E_\beta)$ for any $\beta < 1/2$.
\end{propn}

\begin{proof}
Take $\beta<\tilde \beta<1/2$ and $K\subseteq \HH$; without loss of generality we can assume that $K=[-R,R]\times[0,M]$. We start by noting that for $e^{\tilde \beta \Gf(\cdot, \cdot)}$ is a positive definite operator\footnote{This can be seen, as this is the infinite sum of $\beta^n \Gf ^n/n!$, where each one of them is also positive definite. This last part follows from the fact that they are a constant times the correlation function of the Wick product $:\Gamma^n:$.}. This implies that for any finite measure $\Psi$ (not necessarily positive)
\begin{align*}
\iint e^{\tilde \beta \Gf(x,y)} \Psi(dx)\Psi(dy)\geq 0,
\end{align*}
which implies that for any two positive measures $\Psi_1$, $\Psi_2$
\begin{align}\label{e.a2+b2}
2\iint e^{\tilde \beta \Gf(z,w)} \Psi_1(dz)\Psi_2(dw) \leq \iint e^{\tilde \beta \Gf(z,w)} \Psi_1(dz)\Psi_1(dw) + \iint e^{\tilde \beta \Gf(z,w)} \Psi_2(dz)\Psi_2(dw).
\end{align}

Next, note the fact that, when restricted to $K$
\begin{align*}
    \Leb^{\delta}_\R (\cdot)  = \rev{2^\delta}\int_0 ^M c_y \Leb_{\R + iy} (\cdot) dy,
\end{align*}
where $c_y=\delta y^{-1+\delta}$ \rev{has $\int_0^M c_y dy = M^{\delta}$}. We use this fact together with  \eqref{e.a2+b2} to bound
\begin{align*}
&2\iint_{K\times K} e^{\tilde \beta \Gf(z_1,z_2)} \Leb^{\delta}_{\rev\R}(dz_1) \Leb^{\delta}_{\rev\R}(dz_2)\\ 
&= 2\iint_{[0,M]^2} dy_1 dy_2 c_{y_1}c_{y_2} \int_{K\times K}e^{\tilde \beta \Gf(z_1,z_2)} \Leb_{\R + iy_1}(dz_1)\Leb_{\R + iy_2}(dz_2)\\ 
&\leq \iint_{[0,M]^2} dy_1 dy_2 c_{y_1}c_{y_2} \int_{K\times K}e^{\tilde \beta \Gf(z_1,z_2)}(\Leb_{\R + iy_1}(dz_1)\Leb_{\R+iy_1}(dz_2) +\Leb_{\R+ iy_2}(dz_1)\Leb_{\R + iy_2}(dz_2))\\ 
&\leq 2M^{2\delta}
\rev{\int_{-R}^R}\int_{-R}^{R} \frac{1 }{|x_1-x_2|^{2\tilde \beta} } dx_1 dx_2 \\
&\leq \frac{8}{1-2\tilde{\beta}}M^{2\delta}R^{\rev{2-2\tilde \beta}}<\infty.
\end{align*}
\vspace{-.2cm}
\end{proof}

Next, we prove that both approximations of the one-sided conformal Minkowski content satisfy the energy condition (away from the origin) for any $\beta<d(\kappa)$.
\begin{propn}\label{p:energymeta}
    We have that for any $\epsilon>0$ both $(m_\eta^{q,u}\1_{\Im(\cdot)>\epsilon})_{0<u<1}$ and $(\sigma_\eta^{q,\delta}\1_{\Im(\cdot)>\epsilon})_{0<\delta<1}$ satisfy the energy condition $(\mathcal E_\beta)$ for any $\beta < d=d(\kappa)=1+\tfrac\kappa8$ and $q\in \{L,R\}$. 
\end{propn}
\begin{proof}  Take $\beta<\tilde \beta<d$ and $K\subseteq \HH$; without loss of generality we assume that the diameter of $K$ is smaller than or equal to 1. We start with $(m^{q,u}_\eta \1_{\Im(\cdot)>\epsilon})_{0<u<1}$. Using that 
$|z-\bar{w}|^{\tilde \beta}\rev{>}2\eps^{\tilde{\beta}}$ for all $z,w$ with $\Im(z),\Im(w)>\eps$, we bound
    \begin{align*}
    &\E{\iint_{K\times K} e^{\tilde \beta \Gf(z,w)}\rev{\1_{\Im(z)>\epsilon} \1_{\Im(w)>\epsilon}} m^{q,e^{-r}}_\eta(dz) m^{q,e^{-r}}_\eta(dw)}\\ 
    &\leq ce^{2r(2-d)} \sum_{k=1}^\infty  e^{\hat \beta k} \iint_{K\times K} \1_{e^{-k}<|z-w|<e^{-k+1}} \P(\CR(z,\HH\backslash\eta),\CR(w,\HH\backslash\eta)< e^{-r} ) dz dw,\end{align*}
    \rev{where $c$ is a constant depending only on $\eps$ and $K$.} By Lemma \ref{L:upperbounds} \eqref{E:tau3} together with the comments at the beginning of the Appendix, we have that when $|x-y|>e^{-(k+1)},$
    \begin{align*}
    \P(\CR(z,\HH\backslash\eta),\CR(w,\HH\backslash\eta)< e^{-r} ) \leq ce^{-2r(2-d)}|z-w|^{-(2-d)}\leq ce^{-2r(2-d)} e^{(k+1)(2-d)}.
    \end{align*}
    We conclude by noting that $\iint_{K\times K} \1_{e^{-(k+1)}<|z-w|<e^{-k}} dz dw < ce^{-2k}$.
    
    For the case of $(\sigma_\eta^{\delta,q}\1_{\Im(\cdot)>\epsilon})_{0<\delta<1}$,  if we decompose \[\sigma_\eta^{q,\delta}=\delta \Leb_{\HH^q} + (1-\kappa/8-\delta)\int_0^\infty ds e^{-s}m_\eta^{q,e^{-s/\delta}}\] as in \eqref{e.basic_min} and follow the same idea as for Proposition \ref{prop.EbLebesgue}, it suffices to show\footnote{This is sufficient since the energy of the cross term is controlled by $1/2$ the sum of the energy of the first and the second term by positive definiteness.} that each term individually satisfies the energy condition. This is trivial for the first term.  For the second term, we use \eqref{e.decomposition_min} and see that
    \begin{align*}
    &2\E{\int_0^\infty\int_0^\infty e^{-s} e^{-\hat s} \iint_{K\times K}\rev{\1_{\Im(z)>\epsilon} \1_{\Im(w)>\epsilon}} e^{\tilde \beta \Gf(z,w)}m_\eta^{q,s/\delta}(dz)m_\eta^{q,\hat s/\delta}(dw)ds d\hat s}\\ 
    & \leq\E{ \int_0^\infty\int_0^\infty e^{-s-\hat s} (\iint_{K\times K}\rev{\1_{\Im(z)>\epsilon} \1_{\Im(w)>\epsilon}} e^{\tilde \beta \Gf(z,w)} (m_\eta^{q,s/\delta}(dz)m_\eta^{q,s/\delta}(dw)+ m_\eta^{q,\hat s/\delta}(dz)m_\eta^{q,\hat s/\delta}(dw)))ds d\hat s}\\ 
    &\leq 2 \sup_{r\geq 0} \E{\iint_{K\times K} \rev{\1_{\Im(z)>\epsilon} \1_{\Im(w)>\epsilon}} e^{\tilde \beta \Gf(z,w)} m^{q,r}_\eta(dz) m^{q,r}_\eta(dw)}.
    \end{align*}
    This allows us to conclude using the first part of the proof.
\end{proof}

\section{Convergence of chaos on fractal base measures}\label{s.convergence_chaos}
Assume that we have some (possibly random) sequence of  measures $\mathfrak{m}_n$ on $\overline{\HH}$, independent of a GFF with free boundary condition $\Gamma$ as in Definition \ref{d:freeGFF}, such that 
{$\mathfrak{m}_n\to \mathfrak{m}=\mathfrak{m}_\infty$ almost surely for the vague topology on measures on $\overline\HH$.} Assume further that $\gamma_n$ is a sequence converging to $\gamma=\gamma_\infty \in (0,2)$ as $n\to \infty$.
The question we solve in this section is, when does the $\gamma_n$-Liouville measure for $\Gamma$ with base measure $\mathfrak{m}_n$ converge to the $\gamma$-Liouville measure for $\Gamma$ with base measure $\mathfrak{m}$?


We have the following result.
\begin{propn}\label{p.convergence_liouville}
Assume that $(\mathfrak{m}_n,\gamma_n)_n$ are as above and $(\mathfrak{m}_n)_n$ satisfy $(\mathcal E_\beta)$ with $\beta$ such that $\gamma<2\sqrt{2\beta }$. Then, 
for any deterministic continuous function $f$ with compact support in $\overline{\HH}$,
\begin{align*}
    \int f(z) \tilde \mu_{\Gamma,\eps}^{\scriptscriptstyle{\gamma_n/2}}[\mathfrak{m}_n] (dz)\stackrel{
    \eps \to 0}{\longrightarrow} \int f(z) \tilde \mu^{\scriptscriptstyle{\gamma_n/2}}_\Gamma[\mathfrak{m}_n](dz) \text{ for } n>0 \text{ and } 
    \int f(z) \tilde \mu_{\Gamma,\eps}^{\scriptscriptstyle{\gamma/2}}[\mathfrak{m}](dz)\stackrel{
    \eps \to 0}{\longrightarrow} \int f(z) \tilde \mu_\Gamma^{\scriptscriptstyle{\gamma}/2}[\mathfrak{m}](dz)
\end{align*}
in probability \rev{and in $L^1$}. \rev{When the reference measures $\mathfrak{m}_n,\mathfrak{m}$ put no mass on $\partial\mathbb{H}$, the chaos measures $\tilde{\mu}^{\scriptscriptstyle{\gamma_n/2}}_{\Gamma}[\mathfrak{m}_n],\tilde{\mu}^{\scriptscriptstyle{\gamma/2}}_{\Gamma}[\mathfrak{m}]$ above are those defined by Definition \ref{d:gmc}; when the reference measures give mass to $\partial\mathbb{H}$, they are defined by this proposition.} 

Moreover it holds that
\begin{align*}
   \int f(z) \tilde \mu^{\scriptscriptstyle{\gamma_n/2}}_\Gamma[\mathfrak{m}_n](dz) \stackrel{n \to \infty}{\longrightarrow} \int f(z) \tilde \mu^{\scriptscriptstyle{\gamma/2}}_\Gamma[\mathfrak{m}] (dz),
\end{align*}
where the limit is again in probability \rev{and in $L^1$}.

\end{propn}

Before proving the proposition let us make a remark regarding the non-existence of atoms for a Liouville measure.
\begin{remark}
    Note that if a measure $\mathfrak{m}$ satisfies condition $(\mathcal E_\beta)$, then it puts $0$ mass on singleton. A classic argument \cite[Section 3.13]{BP}, implies that the same is true for its Gaussian multiplicative chaos measure.
\end{remark}
    The proof of this proposition closely follows the construction of the Liouville measure in \cite{Ber} and the proof of convergence of Liouville measures in Appendix A of \cite{GHSS}. The main difficulty we have here is that the free boundary GFF is not a log correlated field in the sense that there is no constant $c$ such that simultaneously for all $z\in \overline{\HH}$,
    \begin{align*}
        \Gf(z,w)\sim c\log(1/|z-w|) \ \ \text{as } w\to z.
    \end{align*}
    However, the same techniques used in the aforementioned papers work in our context. The only thing we need is that 
    \begin{align*}
         \log(1/|z-w|) \lesssim \Gf(z,w)\leq 2\log(1/|z-w|)+C(z,w),
    \end{align*}
    where $C(z,w)$ is uniformly bounded over compacts of $\overline{\HH}$.
    We will try to keep the notation as close as possible to \cite{Ber,GHSS}, {however our definitions are slightly different as we are using $\gamma/2$ instead of $\gamma$.} The main lemma one needs is the following.
    \begin{lemma}\label{l.uniform_convergence_liouville}
Assume $\mathfrak{m}_n$ satisfies $(\mathcal E_\beta)$ and $\mathfrak{m}_n\to \mathfrak{m}=\mathfrak{m}_\infty$ almost surely for the {vague topology of measures on $\overline{\HH}$}. Fix $\hat \gamma<2\sqrt{2\beta}$. Then we have that for any deterministic continuous   function $f$ with compact support in $\C$, $n\in \mathbb{N}\cup \{\infty\}$, and $0<\upsilon<\hat\gamma$, $\int f(x) \tilde{\mu}_{\Gamma,\eps}^{\scriptscriptstyle{\upsilon/2}}[\mathfrak{m}_n]$ has a limit $\int f(x) \tilde{\mu}_\Gamma^{\scriptscriptstyle{\upsilon/2}}[\mathfrak{m}_n]$ in probability as $\eps\to 0$, and
\begin{align*}
\lim_{\eps \to 0}\sup_{n\in \N{\cup \{\infty\}},{\upsilon} <\hat \gamma}\mathbb{E}\left(\left|\int f(x) \tilde \mu_{\Gamma,\eps}^{\scriptscriptstyle{\upsilon/2}}[\mathfrak{m}_n] -\int f(x) \tilde \mu_{\Gamma}^{\scriptscriptstyle{\upsilon/2}}[\mathfrak{m}_n]\right|\right)=0.
\end{align*} \end{lemma}

    We first prove Proposition \ref{p.convergence_liouville} using this lemma.

\begin{proof}[Proof of Proposition \ref{p.convergence_liouville} assuming Lemma \ref{l.uniform_convergence_liouville}] The first limits of the proposition follow directly from Lemma \ref{l.uniform_convergence_liouville} 
For the second, we just need to use the triangular inequality:
\begin{align*}
   & \E{\left|\int f(z) \tilde \mu^{\scriptscriptstyle{\gamma/2}}_\Gamma[\mathfrak{m}](dz) -\int f(z) \tilde \mu^{\scriptscriptstyle{\gamma_n/2}}_\Gamma[\mathfrak{m}_n](dz)\right|}
    \nonumber\leq   \E{\left|\int f(z) \tilde \mu_{\Gamma,\eps}^{\scriptscriptstyle{\gamma/2}}[\mathfrak{m}](dz) -\int f(z) \tilde \mu^{\scriptscriptstyle{\gamma/2}}_\Gamma[\mathfrak{m}](dz) \right|} \\
    & +  \E{\left|\int f(z) \tilde \mu^{\scriptscriptstyle{\gamma_n/2}}_{\Gamma,\eps}[\mathfrak{m}_n](dz) -\int f(z) \tilde \mu_{\Gamma,\eps}^{\scriptscriptstyle{\gamma/2}}[\mathfrak{m}](dz)\right|}+
    \E{\left|\int f(z) \tilde \mu^{\scriptscriptstyle{\gamma_n/2}}_{\Gamma}[\mathfrak{m}_n](dz) -\int f(z) \tilde \mu^{\scriptscriptstyle{\gamma_n/2}}_{\Gamma,\eps}[\mathfrak{m}_n](dz)\right|}.
\end{align*}
Thanks to Lemma \ref{l.uniform_convergence_liouville}, the first and the last terms on the right-hand side go to $0$ as $\eps\to 0$, uniformly in $n$. Thus we only need the second term to go to $0$ as $n\to \infty$ for fixed $\eps>0$. This can be done because $\mathfrak{m}_n\to \mathfrak{m}$ almost surely and $\gamma_n\to \gamma$.   
\end{proof}

We now prove the lemma.

\begin{proof}[Proof of Lemma \ref{l.uniform_convergence_liouville}]
    Fix $\eps_0>0$ and $\alpha>\frac{ \hat \gamma}{2 }$ with $\beta >\frac{\alpha^2}{\rev{2} }$. Without loss of generality we assume that $f\geq 0$. We first show that the set of points
\begin{align*}
    G_\eps=G^{\alpha,\eps_0}_{\eps}:=\left \{z \in \overline \HH: \Gamma_{\eps'}(z) \leq \alpha \E{\Gamma^2_{\eps'}(z)}  \text{ for all } \eps'\in[ \eps,\eps _0]\right \}
\end{align*} 
{defined for $\eps<\eps_0$}, uniformly carries almost all the mass for $\tilde \mu_{\Gamma,\eps}^{\scriptscriptstyle{\gamma'/2}}[\mathfrak{m}_n]$.
To do this, define
\begin{align}
    &\rev{I_\eps^{n}}=I_\eps^{n,\upsilon}:= \int f(z) \1_{(G_\eps)^c}\tilde \mu_{\Gamma,\eps}^{\scriptscriptstyle{\gamma'/2}}[\mathfrak{m}_n](dz), &\rev{J_\eps^{n}}&=J_\eps^{n,\upsilon}:= \int f(z) \1_{G_\eps}\tilde\mu_{\Gamma,\eps}^{\scriptscriptstyle{\gamma'/2}}[\mathfrak{m}_n](dz).
\end{align}
The same proof as \cite[Proof of Lemma 3.4]{Ber} implies the existence of $p(\alpha,\eps_0)$, converging to $0$ as $\eps_0 \to 0$, such that
\begin{align*}
   {\sup_{n\in \N\cup \{\infty\}, \eps<\eps_0, \upsilon<\hat{\gamma}}} \E{I_\eps^n} \leq \int f(z) p(\alpha, \eps_0) dz. 
\end{align*}
Thus, it suffices to show that for any $\eps_0>0$
\begin{align}\label{e.J2}
    \lim_{\eps,\eps'\to 0} {\sup_{n\in \N\cup \{\infty\}, \upsilon<\hat{\gamma}}} \E{(J^{n,\upsilon}_\eps-J^{n,\upsilon}_{\eps'})^2 } =0.
\end{align}
This implies the convergence in probability of $\int f \tilde \mu_{\Gamma,\eps}^{\scriptscriptstyle{\gamma'/2}}[\mathfrak{m}_n]$ for fixed $n,\upsilon$ as in \cite{Ber}, and also yields the desired uniform $L^1$ convergence.
We prove \eqref{e.J2} in the same way as in the Appendix A of \cite{GHSS}, by showing that there exists a function $F:\HH\to \rev{[0,1]}$ such that uniformly in $n$ and $\upsilon\leq\hat \gamma$
\begin{align*}
    \E{(J_\eps^n)^2},\E{J_\eps^nJ_{\eps'}^n } \to \iint_{\HH\times \HH} F(z,w)f(z)f(w) \mathfrak{m}_n(dz)\mathfrak{m}_n(dw).
\end{align*}  
We just prove the case $\eps'=\eps$ as the case for $\eps'\ne\eps$  is analogous (and will give rise to the same $F$). Similarly to \cite{Ber} (except we use $\frac{\upsilon}{2 }$ instead of $\upsilon $):
  \begin{align*}
    \E{(J_\eps^n)^2}= \iint_{\HH\times \HH}e^{\frac{\upsilon^2 }{4 }\E{\Gamma_\eps(z)\Gamma_\eps(w)}}\rev{f(z)f(w)}\widetilde \P\left(G_\eps(z)\cap G_\eps(w) \right)\mathfrak{m}_n(dz)\mathfrak{m}_n(dw),
\end{align*}
where $\tilde \P$ is the probability measure defined by equation (3.8) in \cite{Ber}, that is to say
\begin{align*}
\frac{d \tilde \P }{d \P }= \exp\left(\frac{\upsilon }{2 } (\Gamma_\epsilon(z)+\Gamma_\epsilon (w))- \frac{\upsilon^2 }{8 }\E{(\Gamma_\epsilon(z)+\Gamma_\epsilon(w))^2}\right) 
\end{align*}

One can check that for any given $\delta$, for all $|z-w|\geq \delta $, $ \tilde \P\left(G_\eps(z)\cap G_\eps(w) \right)$ is converging uniformly \rev{(see below equation (4.4) in \cite{Ber})} to some $F(z,w)$ as $\eps\to 0$  {(which coincides with the limit of $\tilde \P\left(G_\eps(z)\cap G_{\eps'}(w) \right)$ as $(\epsilon,\epsilon')\to 0$)}. Thus, it suffices to show that \begin{align*}
    {\sup_{n\in \N\cup \{\infty\}, \eps<\eps_0, \upsilon<\hat{\gamma}}}\E{\iint_{|z-w|\leq \delta}e^{\frac{\upsilon^2 }{4 }\E{\Gamma_\eps(z)\Gamma_\eps(w)}}\widetilde \P\left(G_\eps(z)\cap G_\eps(w) \right)\mathfrak{m}_n(dz)\mathfrak{m}_n(dw)}\to 0
\end{align*}
as $\delta \to 0$.
This is the only point where our proof differs from that of \cite{GHSS}. To do this, we follow \cite{Ber} from equation (3.8) until the end of that proof.  

Taking $r= \eps \vee |z-w|\rev{/2}$, we have that 
\begin{align*}
\widetilde \P\left(G_\eps(z)\cap G_\eps(w) \right)&\leq \widetilde \P\left( \Gamma_r(z)\leq \alpha  \E{\Gamma_{r}^2(z)}\right).
\end{align*}

Now note that by Girsanov's theorem, under the probability $\widetilde \P$, the law of $\Gamma_r(z)$ is that of a Gaussian random variable with variance $\E{\Gamma^2_r(z)}= \E{\Gamma_\epsilon(z)\Gamma_\epsilon(w)} +O(1)$, \rev{where the error term is at most $\log(16)$}, and mean given by
\begin{align*}
    \frac{\upsilon}{2 }\E{ \Gamma_r(z)( \Gamma_\eps(z) + \Gamma_\eps(w))}= \frac{\rev{\upsilon}}{2 }(\Gf(z,w))_{r,\eps} + \frac{\rev{\upsilon}}{2 }(\Gf(z,z))_{r,\eps}= \upsilon(\Gf(z,w))_{\eps,\eps} + O(1),\end{align*}
    where
    \begin{align*}
    (\Gf(z,w))_{r,\eps}:=\E{\Gamma_r(z) \Gamma_\eps(w)}.    \end{align*}
    \rev{and the $O(1)$ term is smaller than $\log(32)$.}

Thus, 
\begin{align*}
\widetilde \P\left( \Gamma_r(z)\leq \alpha  \E{\Gamma_{r}^2(z)}\right)&= \P\left(\mathcal N\left ( \upsilon (\Gf(z,w))_{\epsilon,\epsilon}+O(1), \frac{\gamma'^2 }{2 }(\Gf(z,\rev{z}))_{r,r}\right ) \leq \alpha (\Gf(z,\rev{z}) )_{r, r} \right) \\ 
&\leq \exp\left( -\frac{1 }{2 }(\upsilon-\alpha)^2\left(\sqrt{(\Gf(z,w))_{\epsilon,\epsilon}}+ O(1)\right)^2 \right).
\end{align*} 

This implies, that for any $\upsilon <2\alpha$ we have that $\upsilon^2/4- \frac{1 }{2 }(\upsilon-\alpha)^2 <\beta$ and thus
\begin{align*}
    &\E{\iint_{|z-w|\leq \delta}e^{\frac{\upsilon^2 }{4 }\E{\Gamma_\eps(z)\Gamma_\eps(w)}}\widetilde \P\left(G_\eps(z)\cap G_\eps(w) \right)\mathfrak{m}_n(dz)\mathfrak{m}_n(dw)}\\ 
    &\leq \E{\iint_{|z-w|\leq \delta}e^{\frac{\upsilon^2 }{4 }(\Gf(z,w))_{\epsilon, \epsilon}}\exp\left( -\frac{1 }{2 }(\upsilon-\alpha)^2\left(\sqrt{(\Gf(z,w))_{\epsilon,\epsilon}} + O(1)\right)^2 \right)\mathfrak{m}_n(dz)\mathfrak{m}_n(dw)}\\ 
        &\leq \E{\iint_{|z-w|\leq \delta} e^{\beta \Gf(z,w)} d\mathfrak{m}_n d\mathfrak{m}_n} \\ 
        &\leq \E{\iint_{|z-w|\leq \delta} e^{\tilde \beta \Gf(z,w)} \mathfrak{m}_n (dz)\mathfrak{m}_n}(dw)\times \sup_{z,w: |z-w|\leq \delta} e^{-(\tilde \beta - \beta) \Gf(z,w) } ,
\end{align*}
\rev{where $\tilde\beta$ is defined from $(\mathcal{E}_\beta)$.}
We conclude the proof by taking supremum over $n\in \N$, $0\leq \upsilon <\hat \gamma$ and  $\epsilon<\epsilon_0$ and then taking $\delta \to 0$.
\end{proof}

\section{Chaos on subcritical Minkowski content}\label{s.subcritical}
Let $\Gamma$ be a free boundary condition GFF on $\HH$, as in Definition \ref{d:freeGFF}, and $\eta$ be an independent SLE$_\kappa$ from $0$ to $\infty$ in $\HH$ with $\kappa\in \rev{(}0,4)$.
Suppose that $\eta=(\eta(t))_{t\ge 0}$ is parameterised by half-plane capacity, and for $t>0$ let $f_t$ be the centered Loewner map from $\HH\setminus \eta([0,t])\to \HH$. 
Fix $0<T<\infty$ and define $\Gamma^T:=\Gamma\circ f_T^{-1}+(\tfrac{2}{\sqrt{\kappa}}+\tfrac{\sqrt{\kappa}}{2})\log|(f_T^{-1})'(\cdot)|$. By Theorem \ref{T:gffinvariance}
\[ \Gamma^T \overset{(d)}{=} \tilde{\Gamma} +C+\frac{2}{\gamma}\log|\frac{(\cdot)}{f_T^{-1}(\cdot)}| \]
 where $\tilde{\Gamma}$ is a free boundary GFF with the same covariance as $\Gamma$, and $C$ is a random constant (that is measurable with respect to the pair $(f_T,\tilde{\Gamma})$). Note that $f_T^{-1}$ is not independent from $\tilde \Gamma$.

\begin{propn}\label{p:mainsub}
Let \rev{$\kappa\in (0,4)$} and set $\gamma=\sqrt{\kappa}$. Let $m_\eta$ be as  in Corollary \ref{C:confminkone} and $T>0$. Then, on an event of probability one, the measure $\nu_\Gamma^{m_\eta,\gamma}$ is well defined, and we have for all $0< r < s \le T$,
\begin{equation}\label{e.main_result_subcritical} \frac{2 }{(4-\kappa)(1-\tfrac{\kappa}{8})}\nu_{\Gamma^T}^{\gamma}(f_T(\eta([r,s]))\cap\R^-)=\mu_\Gamma^{\scriptscriptstyle{\gamma/2}}[m_\eta](\eta([r,s]))  = \frac{2 }{(4-\kappa)(1-\tfrac{\kappa}{8})}\nu_{\Gamma^T}^{\gamma}(f_T(\eta([r,s]))\cap \R^+).
 \end{equation} \end{propn}
\rev{In plain words: we can measure the length of $\eta([r,s])$ using the Gaussian multiplicative chaos for $\Gamma$ with parameter $\gamma/2$ and reference measure $m_\eta$ (the conformal Minkowski content of $\eta$). Moreover, this is the same - up to an explicit multiplicative constant - as measuring the length of the image of the left hand side of $\eta$ under the conformal map $f_T$,  using the Gaussian multiplicative chaos for $\Gamma^T$ with parameter $\gamma/2$ and reference measure given by Lebesgue measure on $\R$. The same holds if left is replaced with right in the previous sentence.}

{We recommend that before delving into the proof, the reader revisits points (1), (2), and (3) of Section \ref{ss:proofidea}. These points offer a concise summary of the proof that follows.}

\begin{proof}
Define $N_u$ for $u>0$ by   \begin{align*}
  N_u^T=N_u:=\HH\setminus (\{z\in \HH: d(z,\eta([T,\infty))\cup \R)< u\}\cup\{z: |z|>1/u\}).
  \end{align*} 
\rev{That is, we delete from $\HH$ all points that are too big (absolute value $>1/u$) or too close to $\eta([T,\infty))$ or $\mathbb{R}$ (within distance $u$). We will prove the proposition when we restrict the measures to $N_u^T$ for arbitrary $u$, and then conclude by letting $u\to 0$. See Figure \ref{fig:proof}.} 

We start by showing that $\mu_{\Gamma}^{\scriptscriptstyle{\gamma/2}}[m_\eta]$ is well defined. We note that \rev{by Proposition \ref{p:energymeta} and Remark \ref{r:ebeta}}, as a measure on $N_u^\infty$, $m_\eta$ satisfies the energy condition $(\mathcal E_\beta)$ for $\beta<d= 1+\kappa/8$. Therefore, the measure $\mu_{\Gamma}^{\scriptscriptstyle{\gamma/2}}[m_\eta]$ is well defined as a measure in $N_u^\infty$. Taking $u\to 0$, and using monotone convergence together with the fact that $m_\eta$ puts no mass in $\R$, we conclude that $\mu_{\Gamma}^{\scriptscriptstyle{\gamma/2}}[m_\eta]$ is well defined as a measure on $\HH$.

Next, we prove \eqref{e.main_result_subcritical}, which is the real content of the proposition. Without loss of generality we show the second equality (the first follows in an identical manner). For this, it suffices to prove that for every $u>0$
\begin{equation}\label{e.mainthingtoprove}
(f_T)_*(\mu_{\Gamma}^{\scriptscriptstyle{\gamma/2}}[m_\eta])= \frac{2}{(4-\kappa)(1-\tfrac{\kappa}{8})}\nu_{\Gamma^T}^{\gamma}|_{\R^+}\text{ almost surely as measures on } \overline{f_T(N_u)}, \end{equation}
where $f_*\mathfrak{m}$ represents the pushforward of $\mathfrak{m}$ by $f$.
 This suffices because thanks to Lemmas \ref{lem:weakconv} and \ref{lem:cts}, the occupation measure $m_\eta$ is continuous at $0$ and $\eta(T)$, and $m_\eta$  puts no mass on $\R$. Thus one can take $u\to 0$ and conclude using the monotone convergence theorem.

\begin{figure}[h]
	\includegraphics[width=\textwidth]{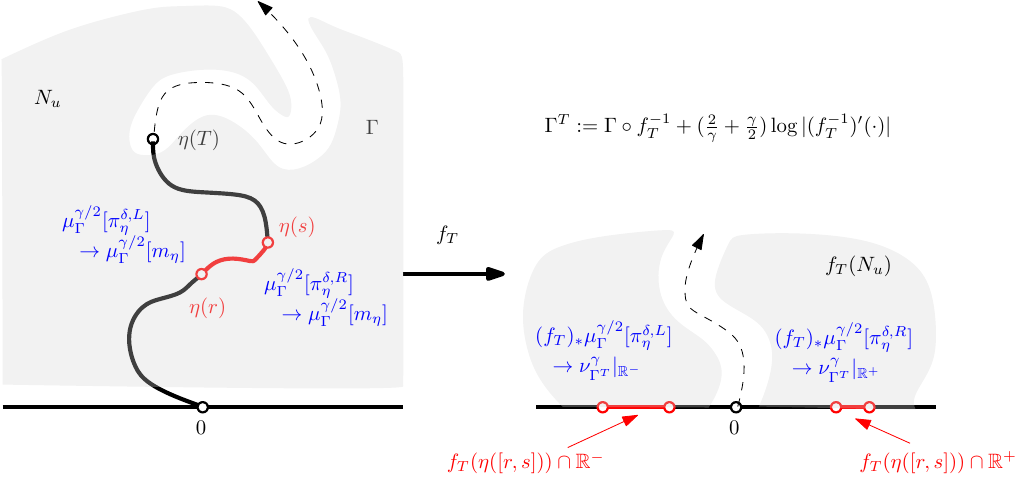}
	\rev{\caption{The regions $N_u$ and $f_T(N_u)$ are shaded in grey. The strategy for the proof of Proposition \ref{p:mainsub} is to define appropriate measure $\pi_\eta^{\delta,L}$ and $\pi_\eta^{\delta,R}$ supported on the left and right of $\eta$, that both converge to $m_\eta$. By continuity of Gaussian multiplicative chaos with respect to the base measure we have the convergence of $\mu_\Gamma^{\scriptscriptstyle{\gamma/2}}[\pi_\eta^{\delta,q}]\to \mu_\Gamma^{\scriptscriptstyle{\gamma/2}}[m_\eta]$ for $q=\{L,R\}$. This is illustrated in blue on the left portion of the figure. We then show that the images of the measures $\mu_\Gamma^{\scriptscriptstyle{\gamma/2}}[\pi_\eta^{\delta,q}]$ under $(f_T)*$ converge to chaos with respect to Lebesgue measure on the real line (restricted to $\R^-$ when $q=L$ and $\R^+$ when $q=R$).	\label{fig:proof}}}
\end{figure}

\rev{The proof of \eqref{e.mainthingtoprove} consists of verifying the following string of identities.} Recall that $l(z)=4\log|z|\I{|z|\ge 1}$ and set $\gamma_\delta:= \sqrt{\kappa}+ 4\kappa^{-1/2}(1-\sqrt{1- \delta \kappa/2})$. \rev{Then for any $u>0$ we have that:
\begin{align}
   (1-\frac{\kappa}{8})(f_T)_*(\mu_\Gamma^{\scriptscriptstyle{\gamma/2}}[m_\eta]) & {=} 
   (f_T)_*(\lim_{\delta\to 0}\mu_\Gamma^{\scriptscriptstyle{\gamma_\delta/2}}[\delta \CR(\cdot,\HH\setminus \eta([0,T]))^{-1+\tfrac{\kappa}{8}+\delta}\mathbf{1}_{A^R(\cdot)}])
   \label{1}\\ 
    & {=} \lim_{\delta\to 0} \mu_{\Gamma^T}^{\scriptscriptstyle{\gamma_\delta/2}}[\delta \CR(\cdot, \HH)^{-1+\tfrac{\kappa}{8}+\delta}\mathbf{1}_{A^R(f_T^{-1}(\cdot))}]
    \label{2} \\ & {=}  \lim_{\delta\to 0} \tilde{\mu}_{\Gamma^T}^{\scriptscriptstyle{\gamma_\delta/2}}[e^{(\gamma_\delta^2/8)l(\cdot)}\delta \CR(\cdot, \HH)^{-1+\tfrac{\kappa-\gamma_\delta^2}{8}+\delta}\mathbf{1}_{A^R(f_T^{-1}(\cdot))}]
       \label{3} \\ & {=} 
       \frac{2}{4-\kappa} \tilde{\mu}_{\Gamma^T}^{\scriptscriptstyle{\gamma/2}}[e^{(\gamma^2/8)l(\cdot)}\mathcal{L}_\R]\big|_{\R^+} 
       \label{4} \\ & {=} 
       \frac{2}{4-\kappa} {\nu}_{\Gamma^T}^{\gamma}\big|_{\R^+} 
       \label{5}
\end{align}
as measures on $\overline{f_T(N_u)}$. 
In brief, \eqref{1} is due to continuity with respect to the base measure, \eqref{2} holds by changing  coordinates, \eqref{3} is switching from $\mu$ to $\tilde{\mu}$, \eqref{4} is due continuity with respect to the base measure again, and \eqref{5} is going from $\tilde{\mu}$ to $\nu$, using \eqref{eq.relmunu}, to conclude.}

 \rev{Let us now fix $u>0$ and prove in full detail that \eqref{1}-\eqref{5} hold almost surely as measures on $\overline{f_T(N_u)}$. This yields \eqref{e.mainthingtoprove} for arbitrary $u$ and thus completes the proof of the proposition.} 
 \vspace{1em}
 
\noindent \rev{{\bf Proof of \eqref{1}}.} \rev{We need to show that 
$$(1-\frac{\kappa}{8})\mu_\Gamma^{\scriptscriptstyle{\gamma/2}}[m_\eta]  {=} 
 \lim_{\delta\to 0}\mu_\Gamma^{\scriptscriptstyle{\gamma_\delta/2}}[\delta \CR(\cdot,\HH\setminus \eta([0,T]))^{-1+\tfrac{\kappa}{8}+\delta}\mathbf{1}_{A^R(\cdot)}]
 $$
   as measures on $\overline{N_u}$, for which we will use Propositions \ref{p.convergence_liouville} and \ref{prop::convergence_better_approximation}.}
\rev{For conciseness, we denote 
\begin{align*}
\pi^{\delta,R}_\eta:= \delta\, \CR(z,\mathbb{H}\setminus \eta([0,T]))^{-1+\tfrac{\kappa}{8}+\delta} \1_{A^R(z)} dz= \left(\frac{\CR(z,\HH\setminus \eta([0,T]))}{\CR(z,\HH\setminus \eta)}\right)^{-1+\tfrac{\kappa}{8}+\delta}\sigma_\eta^{\delta,R}
\end{align*}
on $N_u$. Note that measures $\pi^{\delta,R}_\eta$ are almost exactly the measures $\sigma_\eta^{\delta,R}$, which we know from Proposition \ref{prop::convergence_better_approximation} converge to $m_\eta$, but there is a small modification (the ratio of conformal radii) which means that they behave better under the conformal transformation $f_T$.  In the end, this makes no difference, because the ratio of conformal radii is uniformly bounded over $z\in {N_u}\cap \HH^R$ and as $z\in N_u\cap \HH^R \to \bar z \in \eta([0,T])$, it converges to $1$. This last convergence statement can be verified by observing that if $p_z$ is the probability for a Brownian motion started at $z$ to exit $\HH\setminus \eta([0,T])$ through $\eta([T,\infty))$, then $p_z\to 0$ as $z\to \bar{z}$. Therefore if $g_z:\HH\setminus \eta([0,T])\to \D$ is conformal with $g(z)=0$, we must have $(1-r_z)\D\subset g_z(\HH^R)$ with $r_z\to 1$ as $z\to \bar{z}$, by Beurling's projection theorem. Now, the ratio of conformal radii in question is equal to $\CR(0,g_z(\HH^R))$, and the previous sentence implies that this converges to $1$ as $z\to \bar{z}$.} 

The considerations above imply that $(\pi^{\delta,R}_\eta)_{\delta>0}$ satisfy the energy condition $(\mathcal E_\beta)$ for any $\beta <\rev{2}$ and that $\pi^{\delta,R}_\eta \to (1-\tfrac{\kappa}{8})m_\eta$ almost surely as a measure on ${N_u}$ as $\delta\to 0$. Thus, by Proposition \ref{p.convergence_liouville} \rev{(since Proposition \ref{p.convergence_liouville} is for $\tilde{\mu}$, we have to switch from $\mu$ to $\tilde{\mu}$ and back again, which is straightforward using the definition of $\tilde{\mu}$ compared to $\mu$ and \eqref{e:limvarGamma})}:
\begin{align}\label{e.limitpi}
   (1-\tfrac{\kappa}{8}) \mu_\Gamma^{\scriptscriptstyle{\gamma/2}}[m_\eta] &= (1-\tfrac{\kappa}{8})\tilde \mu^{\scriptscriptstyle{\gamma/2}}_\Gamma[e^{({\gamma^2}/8)l(\cdot)}(2\Im(\cdot))^{-\gamma^2/8}m_\eta] \nonumber \\
   &=\lim_{\delta \to 0} \tilde \mu_\Gamma^{ \scriptscriptstyle{\gamma_\delta/2}}[e^{({\gamma_\delta^2}/8)l(\cdot)}(2\Im(\cdot))^{-\gamma_\delta^2/8} \pi^{\delta,R}_\eta]= \lim_{\delta \to 0} \mu_\Gamma^{\scriptscriptstyle{ \gamma_\delta/2}}[\pi_\eta^{\delta,R}],
\end{align}
where the limits are in probability. 
\vspace{1em}

\noindent\rev{{\bf Proof of \eqref{2}}}
Since $f_T$ is continuous, it suffices to show that for fixed $\delta>0$, 
\begin{equation}\label{e.equalityapproximatemeasures}(f_T)_*\left(\mu_\Gamma^{\scriptscriptstyle{\gamma_\delta/2}}[\delta \CR(\cdot,\HH\setminus \eta([0,T]))^{-1+\tfrac{\kappa}{8}+\delta}\mathbf{1}_{A^R(\cdot)}]\right)
    {=} \mu_{\Gamma^T}^{\scriptscriptstyle{\gamma_\delta/2}}[\delta \CR(\cdot, \HH)^{-1+\tfrac{\kappa}{8}+\delta}\mathbf{1}_{A^R(f_T^{-1}(\cdot))}]
\end{equation}
as measures on $\overline{f_T(N_u)}$.

To do this, we start from Lemma \ref{lem:coc} setting $H=\HH\setminus \eta([T,\infty))$, $f=f_T$ and $\mathfrak{m}=\pi^{\delta,R}_\eta$, which gives us that \[ (f_T)_* \mu_\Gamma^{\scriptscriptstyle{\gamma_\delta/2}}[\delta \CR(\cdot,\HH\setminus \eta([0,T]))^{-1+\tfrac{\kappa}{8}+\delta}\mathbf{1}_{A^R(\cdot)}]=\mu_{\Gamma\circ f_T^{-1}}^{\scriptscriptstyle{\gamma_\delta/2}}[\mathfrak{m}'] \text{ as measures on } f_T(N_u\setminus \eta([0,T]))\]
with \[ \mathfrak{m}'(dz)=|(f_T^{-1})'(z)|^{\tfrac{\gamma_\delta^2}8}(f_T)_*\pi^{\delta,R}(dz)=\delta \CR(z,\HH)^{-1+\tfrac{\kappa}8+\delta}|(f_T^{-1})'(z)|^{1+\tfrac{\gamma_\delta^2}8+\tfrac{\kappa}8+\delta}\1_{A^R_T(z)} dz. \]
Recall that $\Gamma\circ f_T^{-1}$ is equal to $\Gamma^T-(\frac{2 }{\sqrt{\kappa} }+ \frac{ \sqrt{\kappa}}{2}) \log|(f_T^{-1})'(\cdot)|$, and observe that if we are in a compact set inside $f_T(N_u\backslash \eta([0,T])\rev{)}$ the term $\log|(f_T^{-1})'(\cdot)|$ does not blow up. Thus,
\[\mu_{\Gamma\circ f_T^{-1}}^{\scriptscriptstyle{\gamma_\delta/2}}[\mathfrak{m}']  = \mu_{\Gamma^T}^{\scriptscriptstyle{\gamma_\delta/2}}[\mathfrak{m}'']  \text{ as measures on any compact $K\subseteq f_T(N_u\backslash \eta([0,T])$},\]
where
\begin{align*} \mathfrak{m}''(dz)&=|(f_T^{-1})(y)|^{-\tfrac{\gamma_\delta}{\sqrt{\kappa}}-\tfrac{\gamma_\delta \sqrt{\kappa}}{4}}\mathfrak{m}'(dz)\\ 
&=\delta \CR(z,\HH)^{-1+\tfrac{\kappa}8+\delta}|(f_T^{-1})'(z)|^{1-\frac{\gamma_\delta}{\sqrt{\kappa} }+\frac{ 1}{8 }(\gamma_\delta^2+\kappa-2\gamma_\delta \sqrt{\kappa})+\delta} \1_{A_T^R(z)} dy\\ 
&=\delta \CR(z,\HH)^{-1+\tfrac{\kappa}8+\delta} \1_{A_T^R(z)} dz,
\end{align*}
and the last equality is thanks to our choice \footnote{Another way of saying this is that we chose $\gamma_\delta$ so that the power of $|(f_T^{-1})'(z)|$ is equal to 0.} of $\gamma_\delta$.
With this we have proved \eqref{e.equalityapproximatemeasures} for any compact $K\subseteq f_T(N_u\backslash \eta([0,T])\rev{)}$. Since the measures appearing on either side of \eqref{e.equalityapproximatemeasures} do not put mass on $\partial f_T(N_u\backslash \eta([0,T])\rev{)}$, for $u$ small enough, we obtain \eqref{e.equalityapproximatemeasures}.
\vspace{1em}

\noindent\rev{{\bf Proof of \eqref{3}}}
This follows immediately from the definitions of $\mu$ and $\tilde{\mu}$, plus the explicit expressions for the variance of the free field: \eqref{e:varGamma},\eqref{e:limvarGamma}.
\vspace{1em}

\noindent\rev{{\bf Proof of \eqref{4}}}
We consider the measure on $\overline{f_T(N_u)}$ defined by $$\1_{A^T_R(z)}\Leb_\R^{\delta}=\delta \1_{A_T^{R}}(z) (2\Im(z))^{-1+\delta},$$
where $A^R_T(z)=A^R(f_T^{-1}(z))$ is the event that $z$ lies on the right-hand side of $f_T(\eta([T,\infty)))$. Proposition \ref{prop.convergenceLebesgue} implies that this measure converges to $1/2$ \rev{times} the Lebesgue measure on $\R^+\cap \overline{f_T(N_u)}$ as $\delta\to 0$. Proposition \ref{prop.EbLebesgue} shows that the family of measures (in $\delta$) satisfies the energy condition $(\mathcal E_\beta)$ for any $\beta <1/2$, see Remark \ref{r:ebeta}. Thus, using Proposition \ref{p.convergence_liouville} and the fact that $(\kappa-\gamma_\delta^2)/8 = -\delta \kappa/4 + O(\delta^2)$ as $\delta\to 0$, we have that in probability as measures on $\overline{f_T(N_u)}:$
     \begin{align*}\label{e.limitL}
\lim_{\delta\to 0} \tilde\mu_{\Gamma^T}^{\scriptscriptstyle{\gamma/2}}[e^{(\gamma_\delta^2/8 )l(\cdot)}\1_{A_T^R(\cdot)}\delta (2\Im(\cdot))^{-1+(\kappa - \gamma_\delta^2)/8+\delta}] & =\lim_{\delta\to 0} \frac{\delta}{\delta+(\kappa - \gamma_\delta^2)/8}\tilde\mu_{\Gamma^T}^{\scriptscriptstyle{\gamma/2}}[e^{(\gamma_\delta^2/8 )l(\cdot)}\1_{A_T^R(\cdot)}\Leb^{(\kappa - \gamma_\delta^2)/8+\delta}_{\R}] \\
& =\frac{2 }{4-\kappa}\tilde{\mu}_{\Gamma^T}^{\scriptscriptstyle{\gamma/2}}[\mathcal{L}_\R]\big|_{\R^+}.
     \end{align*}
     Note that we can use Proposition \ref{p.convergence_liouville} because by Theorem \ref{T:gffinvariance}, $\Gamma^T$ is absolutely continuous with respect to a GFF \rev{plus a random continuous function} on $f_T(N_u)$.\footnote{\rev{We define the Liouville measures $\mu_{\Gamma^T},\tilde{\mu}_{\Gamma^T}$ associated to $\Gamma^T$ in the same way as for $\Gamma$ (the approximations still converge by absolute continuity). In particular, in the approximation for $\tilde{\mu}_{\Gamma^T}$ we use the multiplicative normalisation $e^{-(\gamma^2/8) \var(\Gamma_\eps(z))}$.}}
\vspace{1em}

\noindent\rev{{\bf Proof of \eqref{5}}} This is exactly \eqref{eq.relmunu}.



\end{proof}

We finish this subsection with a remark regarding the use of both $\Gamma^T$ and $\gamma_\delta$.
\begin{remark}
    Let us note that using the exact change of coordinates given by $\Gamma^T$ is key in step (2) of the proof. One can use Proposition \ref{p.convergence_liouville} for $\Gamma^T$ thanks to the fact that it is absolutely continuous with respect to a free boundary GFF \rev{plus a random continuous function}. We could not do this for any other constant in front of $\log|(f^{-1}_T)'(\cdot)|$.

    A surprising step of the proof is the appearance of $\gamma_\delta \neq \sqrt{\kappa}$. This term appears due to the fact that at an approximate level one does not have the equality \eqref{e.mainthingtoprove} exactly for $\gamma$, instead one has it for $\gamma_\delta$ \eqref{e.equalityapproximatemeasures} satisfying a quadratic equation with respect to $\delta$. This last fact is the one that makes the term $2/(4-\kappa)$ appear in Theorem \ref{T:main}.
\end{remark}

\section{Chaos on critical Minkowski content}\label{s.critical}
In this section, we want to take a limit as $\gamma\nearrow 2,\kappa \nearrow 4$ in \eqref{e.main_result_subcritical} to obtain the result in the critical case. Namely, we prove the following proposition.
\begin{propn}\label{p:maincrit}
Let $\kappa=4$ and $\gamma=2$. Let $m_\eta$ be as  in Corollary \ref{C:confminkone} and $T>0$. Then, on an event of probability one, the measure $\mu_\Gamma^{2}[m_\eta]$ is well defined, and we have for all $0< r < s \le T$,
    \begin{equation}
\label{e:maincrit}\mu_\Gamma^{\rev{1}}[m_\eta]\left(\eta([r,s])\right)  =2\nu_{\Gamma^T}^{\mathrm{crit}}\left(f_T(\eta([r,s]))\cap \R^+\right)  = 2\nu_{\Gamma^T}^{\mathrm{crit}}\left(f_T(\eta([r,s]))\cap \R^-\right).
    \end{equation}
 \end{propn} 
This proposition is the final step to prove Theorem \ref{T:main}.
 \begin{proof}[Proof of Theorem \ref{T:main}]
{This follows directly from Propositions \ref{p:mainsub} and \ref{p:maincrit}.}
\end{proof}

 The proof of Proposition \ref{p:maincrit} can be grasped intuitively by considering the limits as $\kappa\nearrow 4$ and $\gamma\nearrow 2$ in Proposition \ref{p:mainsub}. However, three crucial challenges arise, making the proof more technical than anticipated.
 \begin{enumerate}[(a)]

      \item It is necessary to find a coupling of $(\eta_\kappa)_{\kappa\in (0,4]}$ with $\eta_\kappa$ having the law of SLE$_{\kappa}$ for each $\kappa$, such that almost surely $\lim_{\kappa\nearrow 4}\eta_\kappa = \eta$ and $\lim_{\kappa\nearrow 4}m_{\eta_\kappa}=m_\eta$.
       \item One needs to show that as $\gamma\nearrow 2$, $2((4-\gamma^2 )(1-
      \gamma^2/8))^{-1}\nu_{\Gamma^T}^{\gamma}|_{\R^+}$ is converging towards $2\nu_{\Gamma^T}^{\mathrm{crit}}|_{\R^+}$. This is not straightforward from \eqref{e.critical_measure}, since $\Gamma^T$ depends on $\gamma$ and $\eta$.
    \item In the proof one has measures $\mathfrak{m}_n$ converging to $\mathfrak{m}$ for the vague topology and sets $A_n$ converging to $A$ for the Hausdorff topology. This does not guarantee that $\mathfrak{m}_n(A_n)$ converges towards $\mathfrak{m}(A)$.    
 \end{enumerate}
To address (a) and (b), we opt to weaken our convergences by considering the convergence in law of a comprehensive tuple and identifying accumulation points. For (b), a simpler coupling is chosen, maintaining $\Gamma_{T}$ nearly constant. Addressing (a) proves more intricate. One approach would be to show that the limit (in probability) defining $m_{\eta}(D)$ is uniform in $\kappa$, but this necessitates a detailed examination of constants in the proof of \cite{LawlerRezaei}, which we avoid. Instead, in Lemma \ref{l.caracterization_eta,m}, we consider any accumulation point $(\eta, m_{\eta})$ of pairs $(\eta_n,m_{\eta_n})$ (made up of an $\SLE_{\kappa_n}$ curve together with its Minkowski content measure). While we cannot directly prove that $m_\eta$ is the Minkowski content measure of $\eta$, we establish that the measure $m_\eta$, conditional on $\eta$, corresponds to the Minkowski content measure. This is sufficient to address (a). Finally, (c) is resolved by employing the technical Lemma \ref{l:convergingmeasuresandsets}. To utilise this result, we demonstrate that the occupation measure and its Gaussian multiplicative chaos don't assign mass to points close to $\eta[r,s]$, as shown in Lemma \ref{c.uniformsepSLE}. 

With these steps outlined, we proceed to present the proof of Proposition \ref{p:maincrit}.
\begin{proof}[Proof of Proposition \ref{p:maincrit}] In this proof it is simpler to work with $\hat \Gamma^\gamma(\cdot)=\Gamma(\cdot) + (2/\gamma)\log|(\cdot)|$ instead of $\Gamma(\cdot)$, in order to use Theorem \ref{T:gffinvariance} more directly. Note that the law of $\hat \Gamma^\gamma$ restricted to \rev{any compact of} $\overline{\HH}\backslash B(0,\eps)$ \rev{for any $\eps>0$} is absolutely continuous with respect to the law of a free boundary GFF restricted \rev{to the same set.} This implies that Proposition \ref{p:mainsub} (with $\kappa=\gamma^2$) holds for $\hat \Gamma^\gamma$ instead of $\Gamma$. It also means that proving Proposition \ref{p:maincrit} with $\hat \Gamma^{\gamma=2}$ is enough to deduce the result for $\Gamma$.

Also notice that it is sufficient to prove the first equality of \eqref{e:maincrit} with $r=S\in (0,T)$ arbitrary, and $s=T$. Indeed, the second equality (for these fixed values of $r,s$) follows identically, and since $S\in (0,T)$ was arbitrary, we may then deduce the result restricted to all rational $0< r<s\le T$. The full result finally follows from this since none of the measures in question have atoms.

So, we fix $0<S<T$. We take a sequence $\kappa_n\nearrow 4$, $\gamma_n:=\sqrt{\kappa_n}$, and study  convergence of the tuple
\begin{equation}\label{e.massivetuple}
(\hat{\Gamma}_n, \eta_n, \hat{\Gamma}_n^T, m_n, I_n, \rev{\mu}_n, \rev{\mu}_n(A_n),{\nu}_n, {\nu}_n(I_n))
\end{equation}
where $\hat \Gamma_n$ is \rev{equal in distribution to a free boundary GFF $\Gamma$ plus the function $(2/\gamma_n)\log(|\cdot|)$} , $\eta_n$ is an independent $\SLE_{\kappa_n}$ parametrised by half-plane capacity with associated centred Loewner map $f_{T,n}$, $I_n=f_{T,n}(\eta_n([S,T]))\cap \R^+$, $A_n=\eta_n([S,T])$, 
$\hat\Gamma^T_n:=\hat \Gamma\rev{_n}\circ f_{T,n}^{-1}+(2/\gamma_n+\gamma_n/2)\log|(f_{T,n}^{-1})'|$, and 
\[ m_n=m_{\eta_n}, \; \rev{\mu_n=\mu_{\hat \Gamma_n}^{\scriptscriptstyle{\gamma_n/2}}[m_n]} \text{ and }\rev{\nu_n}=\rev{\frac{1}{(4-\kappa_n)(1-\tfrac{\kappa_n}{8})}}\rev{\nu}_{\hat\Gamma^T_n}^{\gamma_n}|_{\R^+}.\]

We will always consider convergence in the following topologies: the topology of uniform convergence on compact \rev{subsets of $[0,\infty)$} for the curves $\eta_n$; $H^{-1}_{\mathrm{loc}}(\HH)$ convergence for the fields $\hat{\Gamma}_n$ and $\hat\Gamma^T_n$, the topology of vague convergence for measures on $\overline{\HH}$ for $m_n$ and $\mu_n$, the topology of vague convergence for measures on $\R$ for ${\nu}_n$, Hausdorff convergence for the compacts $I_n, A_n$, and usual Euclidean convergence on $\R$ for $\mu_n(A_n)$ and $\nu_n(I_n)$.

The result is implied by the following marginal convergence statements. 
\begin{enumerate}
    \item\label{i:ggeta} As $n\to \infty$, 
    \[
    (\hat{\Gamma}_n, \eta_n, \hat{\Gamma}_n^T)\to (\hat{\Gamma}, \eta, \hat{\Gamma}^T)
    \] in law, where $\hat{\Gamma}$ has the law of a free boundary GFF plus $\log(|\cdot|)$, $\eta$ is an independent $\SLE_4$ with centred Loewner map $f_T$ at time $T$, and $\hat\Gamma^T=\hat\Gamma\circ f_T^{-1}+2\log|(f_T^{-1})'|$ almost surely.
    \item\label{i:etaI} As $n\to \infty$, 
    \[
    (\eta_n, I_n) \to (\eta, I)
    \]
 in law, where $\eta$ is as in (\ref{i:ggeta}), and $I=f_T(\eta([S,T])\rev{)}\cap \R^+$ almost surely.
    \item\label{i:ghat} As $n\to \infty$, 
    \[ 
    (\hat{\Gamma}_n^T,{\nu}_n)\to (\hat{\Gamma}^T, {\nu})
    \]
 in law, where the law of $\hat{\Gamma}^T$ is as in (\ref{i:ggeta}) and $\rev{\nu}=\nu^{\mathrm{crit}}_{\hat{\Gamma}^T}$ almost surely.
    \item\label{i:nu} The law of $(\hat\Gamma_n,\eta_n, m_n, \mu_n, \mu_n(A_n))$ is tight, and for any subsequential limit $(\hat\Gamma,\eta,m,\mu,X)$, $(\hat\Gamma,\eta)$ has same the joint law as in (\ref{i:ggeta}) and  $\mathbb{E}(X\mid \eta,\hat\Gamma)=\mu^{1}_{\hat\Gamma}[m_\eta]\left(\eta([S,T])\right)$
    almost surely. Note that we will not show that $X$ belongs to $L^1$; however its conditional expectation is well-defined, since $X$ is almost surely positive.
\end{enumerate}
Indeed, if (\ref{i:ggeta})-(\ref{i:nu}) hold, then the tuple \eqref{e.massivetuple} is tight, and by Proposition \ref{p:mainsub} applied with $\hat{\Gamma}_n$ in place of $\Gamma$, we have that 
\[ 
\mu_n(A_n)=\rev{2}{\nu}_n(I_n)
\]
almost surely for every $n$, so for any subsequential limit
\[
(\hat{\Gamma},\eta,\hat{\Gamma}^T, m,I,\mu,X,\nu,Y)
\]
of \eqref{e.massivetuple}, we have $X=\rev{2Y}$ almost surely. (\ref{i:ggeta})-(\ref{i:nu}) imply that for this limiting tuple, $\hat{\Gamma}$ has the law of a free boundary GFF plus $\log(|\cdot|)$, $\eta$ is an independent $\SLE_4$ with centred Loewner map $f_T$ at time $T$, 
\[
\hat\Gamma^T=\hat\Gamma\circ f_T^{-1}+2\log|(f_T^{-1})'|, \quad I=f_T(\eta([S,T]))\cap \R^+, \text{ and } \nu=\nu_{\hat\Gamma^T}^{\mathrm{crit}}|_{\R^+}
\] 
almost surely, and 
\begin{equation}\label{e.condX}
    \mathbb{E}(X\mid \eta,\hat\Gamma)=\mu^{1}_{\hat\Gamma}[m_\eta]\left(\eta([S,T])\right) \text{ almost surely.}
\end{equation} 
Since $I$ is an interval, the convergence $I_n\to I$ and $\nu_n\to \nu$ implies that $Y={\nu}(I)$ almost surely (this can be seen by working on a probability space where the convergence holds almost surely, \rev{and using that $\nu_n,\nu$ are almost surely atomless}). Hence, 
\[
X=\rev{2}Y=\rev{2}\nu_{\hat\Gamma^T}^{\mathrm{crit}}\left(f_T(\eta([S,T]))\cap \R^+\right) \text{ almost surely.}
\]
Taking conditional expectations of both sides given $\hat{\Gamma},\eta$ and using \eqref{e.condX} plus the fact that the right-hand side is $(\hat{\Gamma},\eta)$-measurable, we obtain that 
\[
\mu^{1}_{\hat\Gamma}[m_\eta]\left(\eta([S,T])\right)=\nu_{\hat\Gamma^T}^{\mathrm{crit}}\left(f_T(\eta([S,T]))\cap \R^+\right) \text{ almost surely.}
\]
This is exactly what we needed, so it remains to prove (\ref{i:ggeta}),(\ref{i:etaI}),(\ref{i:ghat}),(\ref{i:nu}). We do this in turn below.
\end{proof}

\begin{proof}[Proof of (\ref{i:ggeta})]
We define a coupling of $(\hat\Gamma_n,(\eta_n),\eta)$ where: $\Gamma$ is a free boundary GFF independent of all the curves and $\hat\Gamma_n=\Gamma+(2/\gamma_n)\log|\cdot|$ for each $n$; $\eta_n\to \eta$ almost surely for the topology of uniform convergence on compact \rev{subsets of $[0,\infty)$} as $n
\to \infty$. This is possible by \cite[Theorem 1.10] {KS}. The convergence of the curves implies that
\[
f_{T,n}^{-1}, (f_{T,n}^{-1})' \to f_T^{-1}, (f_T^{-1})'
\]  
uniformly almost surely on any compact set contained in $\HH$. \rev{It is clear that $\hat\Gamma_n\to \hat{\Gamma}=\Gamma+\log|\cdot|$ as $n\to \infty$.} Together with the convergence in the displayed equation above, we see that 
\[
\hat\Gamma_n^T=\hat\Gamma_n\circ f_{T,n}^{-1}+(2/\gamma_n+\gamma_n/2)\log|(f_{T,n}^{-1})'|\to \hat\Gamma^T=\hat\Gamma \circ f_T^{-1}+2\log|(f_T^{-1})'|\] almost surely in $H^{-1}_{\mathrm{loc}}(\HH)$. 
\end{proof}

\begin{proof}[Proof of (\ref{i:etaI})]
We define a coupling of $(\eta_n)_n$ by using $\sqrt{\kappa_n}W$ for the driving function of each $\eta_n$, with $W$ a single Brownian motion. Then $\eta_n\to \eta$ almost surely in the Carath\'{e}odory sense, where $\eta$ is an $\SLE_4$ curve driven by $2W$. Moreover, $(f_{T,n}(\eta_n(S)))_n$ are a collection of Bessel processes evaluated at time $T-S$, all driven by the same Brownian motion $W$ and with converging dimension. This implies that $f_{T,n}(\eta_n(S))\to f_T(\eta(S))$, where $f_T$ is the \rev{centred Loewner map associated to} $\eta$ \rev{at time $T$}. Since $f_{\rev{T,n}}(\eta_n(T))=0$ for all $n$ by definition, this implies that $I_n\to f_{T}(\eta([S,T]))$ almost surely. In particular, $I_n$ is tight in law, and $\eta_n$ is tight in law due to the convergence in (\ref{i:ggeta}). As uniform convergence on compacts of time implies convergence in the Carath\'{e}odory sense, the above discussion yields (\ref{i:etaI}).
\end{proof}

\begin{proof}[Proof of (\ref{i:ghat})] 
For each $n$, let $C_n$ be the average of $\hat{\Gamma}_n^T-(2/\gamma_n)\log(|\cdot|)$ on the upper unit semicircle. The convergence of $\hat{\Gamma}_n^T$ from \eqref{i:ggeta} implies that 
\begin{equation}\label{e.hgtg}(\hat{\Gamma}_n^T,C_n,\hat{\Gamma}_n^T-C_n-(2/\gamma_n)\log(|\cdot|))
\to 
(\hat\Gamma^T,C,\hat\Gamma^T-C-\log(|\cdot|))
\end{equation} in law as $n\to \infty$, where $C$ is the average of $\hat{\Gamma}^T-\log(|\cdot|)$ on the upper unit semicircle. Recall that $\hat{\Gamma}_n^T-C_n-(2/\gamma_n)\log(|\cdot|)$ has the law of a free boundary GFF as in Definition \ref{d:freeGFF} for each $n$, due to Theorem \ref{T:gffinvariance}. Therefore, by defining a coupling where they are all identically equal, we see that 
\[
(\hat{\Gamma}_n^T-C_n-(2/\gamma_n)\log(|\cdot|), \frac{1}{(4-\gamma_n^2)(1-\gamma_n^2/8)}\nu_{\hat{\Gamma}_n^T-C_n-(2/\gamma_n)\log(|\cdot|)}^{\gamma_n}|_{\R^+})\to (\tilde{\Gamma},\nu^{\mathrm{crit}}_{\tilde{\Gamma}}|_{\R^+})
\]
in law as $n\to \infty$, where $\tilde{\Gamma}$ has the law of a free boundary GFF. For this last point, we are also using the definition of the critical measure and the fact that $(2-\gamma_n)/\left((4-\gamma_n^2)(1-\gamma_n/8)\right)
\to 1/2$ as $n\to \infty$. Combining with 
\eqref{e.hgtg}, we obtain that 
\begin{multline*}
\left(\hat{\Gamma}_n^T,C_n, \hat{\Gamma}_n^T-C_n-(2/\gamma_n)\log(|\cdot|), \frac{1}{(4-\gamma_n^2)(1-\gamma_n^2/8)}\nu_{\hat{\Gamma}_n^T-C_n-(2/\gamma_n)\log(|\cdot|)}^{\gamma_n}|_{\R^+}\right) \\
\to \left(\hat\Gamma,C,\hat\Gamma^T-C-\log(|\cdot|), \nu^{\mathrm{crit}}_{\hat\Gamma^T-C-\log(|\cdot|)}|_{\R^+}\right)
\end{multline*}
in law as $n\to \infty$, where $C$ is the average of $\hat{\Gamma}^T-\log(|\cdot|)$ on the upper unit semicircle.  Finally, since 
\[
\nu_{\hat\Gamma^T_n}^{\gamma_n}(dx)=xe^{\tfrac{\gamma_nC_n}2}\,
\nu_{\hat{\Gamma}_n^T-C_n-(2/\gamma_n)\log(|\cdot|)}^{\gamma_n}(dx)
\]
for each $n$, we deduce from the convergence above that also 
\[
(\hat{\Gamma}_n^T, \frac{1}{(4-\gamma_n^2)(1-\gamma_n^2/8)}\nu_{\hat{\Gamma}_n^T}^{\gamma_n}|_{\R^+})\to (\hat\Gamma^T, \nu^{\mathrm{crit}}_{\hat\Gamma^T}|_{\R^+})
\]
in law, as desired.
\end{proof}

The proof of (\ref{i:nu}) is the most involved, and we will need a couple of preparatory lemmas. For each $n$, we let $G_n$ be the $\SLE_{\kappa_n}$ Green's function, as in Definition \ref{d:GreenSLE}. We let $G$ be the $\SLE_4$ Green's function.

\begin{lemma}
Let $0\le s\le t< \infty$, $D\subset \HH$ be a bounded simply connected domain,   and $a>0$. Then denoting $m_n:=m_{\eta_n}$ and $A_n:=\eta_n([s,t])\cap D$, we have 
\[\sup_{n>0}\P(m_n((A_n)_\eps)-m_n(A_n)\ge a)\to 0\]
as $\eps\to 0$, \rev{where for a closed set $K$, $K_\eps$ denotes the set $\{z: d(z,K)<\eps\}$.}
\label{c.uniformsepSLE}
\end{lemma}

\begin{proof}
We fix $a>0$. For any $K,\delta>0$, we write $W_{t,n,\eps,K}$ for the event that $f_{t,n}((A_n)_\eps\setminus A_n)\,\rev{\not\subset} \,[-K,K]\times i\R$ and $H_{t,n,\eps,\delta}$ for the event that $f_{t,n}((A_n)_\eps\setminus A_n)\,\rev{\not\subset}\,\R \times i[0,\delta]$. Then we have 
\begin{align*}
& \sup_{n>0}\P\left(m_n\left((A_n)_\eps\cap \eta_n([t,\infty))\right)\ge a\right) \\
& \le \sup_{n>0} \P(W_{t,n,\eps,K}\cup H_{t,n,\eps,\delta}) + \sup_{n>0} \frac{\E{m_n\left((A_n)_\eps\cap \eta_n([t,\infty))\right)\1_{(W_{t,n,\eps,K}\cup H_{t,n,\eps,\delta})^c}}}{a}\\
& \le \sup_{n>0,\eps<1} \P(W_{t,n,\eps,K})+\sup_{n>0} \P(H_{t,n,\eps,\delta})+\sup_{n>0}\frac{\int_{[-K,K]\times i[0,\delta]} G_n(z)}{a}
\end{align*}
The first term on the right converges to $0$ as $K\to \infty$. For any fixed $K$, the final term on the right converges to $0$ as $\delta\to 0$. Finally, for any fixed $\delta$ the second term converges to $0$ as $\eps\to 0$.  Thus  $\sup_{n>0}\P(m_n((A_n)_\eps\cap \eta_n([t,\infty)))-m_n(A_n)\ge a)\to 0$ as $\eps \to 0$. To see that $\sup_{n>0}\P(m_n((A_n)_\eps\cap \eta_n([0,s]))-m_n(A_n)\ge a)$ also converges to $0$, we use reversibility of SLE, \cite{reversibility}.
\end{proof}

\begin{lemma}\label{l.caracterization_eta,m}
Suppose that $(\eta_n,m_n)$ has a limit in law $(\eta,m)$ with respect to the topology of uniform convergence on compacts of time in the first coordinate, and the vague topology for measures on $\overline{\HH}$ in the second coordinate.
Then the measure $D\mapsto \mathbb{E}(m(D)|\eta)$ is almost surely equal to $m_\eta$.
\label{l.etam}
\end{lemma}

\begin{proof}
    By applying Skorokhod embedding we may assume that $\eta_n\to \eta$ and $m_n\to m$ almost surely as $n\to \infty$. Set $\hat{m}(D)=\mathbb{E}(m(D)|\eta)$. By \cite[Section 3; \rev{immediately after (3.2)}]{LawlerSheffield}, it suffices to show that for every bounded $D$: 
    \begin{enumerate}[(i)]
        \item $\hat{m}(D\cap \eta([0,t])\rev{)}$ is $\rev{\sigma( \eta(s); s\le t)}$  measurable for every $t$;
        \item $\hat{m}(D\cap \eta([0,t]))$ is increasing and continuous in $t$;
        \item $\mathbb{E}(\hat{m}(D)|\rev{\sigma( \eta(s); s\le t)} )\rev{)}=\hat{m}(D\cap \eta([0,t]))+\tfrac12 \int_{D\setminus \eta([0,t])} |g_t'(z)|^{2-d}G(g_t(z)-W_t) dz$ 
        almost surely for every $t$, where $G$ is the $\SLE_4$ Green's function from Definition \ref{d:GreenSLE} and $d=d(4)=3/2$.
    \end{enumerate}

    We prove (ii), then (iii), and finally (i).
    \begin{itemize}
        \item[(ii)]  $\hat{m}(D)\cap \eta([0,t])\rev{)}$ is increasing in $t$, since $\hat{m}$ is a positive measure and $\eta([0,t])$ is increasing with $t$. To prove continuity, we use that  $\eta$ is almost surely continuous in $t$ and does not touch the real line except at $0$. This means that for fixed $t$, with probability one, for any $N>0$ there exists $\delta>0$ such that $\eta([t,t+\delta])$ is contained in some (random) square $S=2^{-N}([j,j+2]+i[k,k+2])$ with $j,k\in \Z$ and $k>0$. Write $\mathcal{Q}_N^+$ for the set of all such squares at level $N\ge 0$. It suffices to prove that \begin{align*}
        \sup_{S\in \mathcal{Q}_N^+} \hat{m}(S\cap D)\to 0 \   \text{ as $N\to \infty$ almost surely.}
        \end{align*}
         By Borel-Cantelli, for this it is enough to show existence of some $\eps_N$ converging to $0$ as $N\to \infty$, such that 
        \begin{equation}\label{BC}
        \sum_{N=0}^\infty\sum_{S\in \mathcal{Q}_N^+} \mathbb{P}(\hat{m}(S\cap D)>\eps_N)<\infty.
        \end{equation}
        To prove this we use that for any $S\in \mathcal{Q}_N^+$,
        $$ \mathbb{E}\left(\hat{m}(S)^2\right)\le \mathbb{E}\left(m(S)^2\right)\le \liminf_n \mathbb{E}\left(m_n(S)^2\right)
        $$
        thanks to Fatou's lemma. Moreover, \rev{by} \cite[Theorem 3.1 and Section 4.2]{LawlerRezaei}, there exists $c<\infty$, not depending on $n$, such that 
\begin{equation}\label{eq:moment_not_n}\mathbb{E}\left(m_n(S)^2\right)\le c2^{-N(1+(\kappa_n/8))} \int_S G_n(z) dz\end{equation}for all $S\in \mathcal{Q}_N^+$ and $n\ge 0$. Thus, for each $N$ 
        $$ \sum_{S\in \mathcal{Q}_N^+}\mathbb{P}(\hat{m}(S\cap D)>\eps_N) \le c\liminf_n 2^{-N(1+(\kappa_n/8))} \int_{D} G_n(z) dz \le c' 2^{-aN} $$
        for some $c'<\infty$ and (any) $a<3/2$. Choosing $\eps_N$ such that $\sum_N \eps_N^{-2} 2^{-aN}<\infty$, we obtain \eqref{BC}.

        \item[(iii)] By definition of $\hat{m}$, the tower property, and a change of variables this is equivalent to showing that 
    \begin{align}\label{e.3}
& \E{m(D)|\rev{\sigma( \eta(s); s\le t)} }= \nonumber
\\
& {\E{m(D\cap\eta([0,t])\rev{)} \mid \rev{\sigma( \eta(s); s\le \infty)} }} +\int_{f_t(D)} |(f_t^{-1})'(z)|^d G(z) \, dz.
\end{align}
By definition of $m_n$, we have that \[
\mathbb{E}\left(m_n(D)\mid \rev{\sigma( \eta_n(s); s\le t)} \right)=m_n(D\cap \eta_n([0,t]))+\int_{f_{t,n}(D)} |(f_{t,n}^{-1})'(z)|^{d_n} G_n(z) dz
\] for each $n$, where $G_n$ is the $\SLE_{\kappa_n}$ Green's function and $d_n=1+\kappa_n/8$.  By Lemma \ref{l:convergingmeasuresandsets} and Lemma \ref{c.uniformsepSLE} we have that $m_n(D\cap \eta_n([0,t]))\to m(D\cap\eta([0,t])\rev{)}$ almost surely. Moreover, we have that $f_{t,n}\to f_t$, $G_n\to G$ uniformly on compacts of $\HH$, and for any fixed $K>0$
\begin{align*}
\sup_n\int_{[-K,K]\times i[0,\eps]} |(f_{t,n}^{-1})'(z)|^{d_n} G_n(z) \, dz\to 0 \text{ as $\eps\to 0$,} \end{align*}
 (indeed, we have the deterministic identity $
\sup_n\sup_{z\in[-K,K]\times i[0,\eps]}|(f_{t,n}^{-1})'(z)|\le C\sqrt{\rev{1+}t}$ for some $C$ which follows from elementary stochastic calculus, see e.g. \cite[(57)]{LawlerSheffield}). Hence, \[
        \int_{f_{t,n}(D)} |(f_{t,n}^{-1})'(z)|^{d_n} G_n(z) dz \to \int_{f_t(D)}|(f_t^{-1})'(z)|^d G(z) dz \text{ almost surely.}\]    
        Combining the above we see that 
        \[ 
        \mathbb{E}\left(m_n(D)\mid\rev{\sigma( \eta_n(s); s\le t)} \right)\to m(D\cap \eta([0,t])\rev{)}+ \int_{f_t(D)} |(f_t^{-1})'(z)|^d G(z) dz
        \] 
        almost surely as $n\to \infty$.
        This convergence also holds in $L^1$ since $m_n(D)$ is uniformly integrable \rev{(see \eqref{eq:moment_not_n})}. Applying Lemma \ref{l.convergence_conditional_expectation} for the variables $Z=m(D\cap \eta([0,t])\rev{)}+\int_{f_t(D)}|(f_t^{-1})'(z)|^d G(z) dz$, $X_n=m_n(D)$ and $Y_n=\eta_n([0,t])$, we conclude that 
        \begin{equation}\label{e.}\E{m(D)|\rev{\sigma( \eta(s); s\le t)} }=\E{m(D\cap \eta([0,t]))\, | \, \rev{\sigma( \eta(s); s\le t)} }+\int_{f_t(D)} |(f_t^{-1})'(z)|^d G(z) dz.
        \end{equation}

        We now apply Lemma \ref{l.conditional_independence_limit}, with $Z_n =\eta_n([0,t])\to \eta([0,t])$, $Y_n=\eta_n([t,\infty)]\rev{)}\to \eta([t,\infty))$, $Z_n=m_n\left(D\cap \eta_n([0,t])\right)\to m\left(D\cap \eta([0,t])\right)$. We are allowed to apply this lemma since for any continuous and bounded $f$, writing $\mathbb{E}^{\kappa}_{D,a,b}$ for the law of an SLE$_\kappa$ from $a$ to $b$ in $D$, we have that $\E{f(\eta_n([t,\infty)))\mid \eta_n([0,t])}=\mathbb{E}^{\kappa_n}_{\rev {\HH\setminus}\eta_n([0,t]),\eta(t),\infty}(f(\eta))$, and this converges in probability to the $\rev{\sigma( \eta(s); s\le t)} $-measurable random variable $\mathbb{E}^4_{\rev {\HH\setminus}\eta([0,t]),\eta(t),\infty}(f(\eta))$. Thus, we see that conditionally on $\eta([0,t])$, $m(D\cap \eta([0,t]))$ is independent of $\eta([t,\infty))$. This gives that
        \begin{align}\label{e.slecondind}
        \E{m(D\cap \eta([0,t]))\mid\rev{\sigma( \eta(s); s\le \infty)} } &= \E{m(D\cap \eta([0,t]))\mid \rev{\sigma( \eta(s); s\le t)} , \rev{\sigma( \eta(s); s\ge t)}  }\nonumber\\
        &=\E{m(D\cap \eta([0,t]))\mid \rev{\sigma( \eta(s); s\le t)}  } .
        \end{align}
        Substituting this into \eqref{e.} yields \eqref{e.3}.

        \item[(i)] This follows from \eqref{e.slecondind} since $\hat{m}(D\cap\eta([0,t]))=\E{m(D\cap\eta([0,t]))|\rev{\sigma( \eta(s); s\le \infty)} }$ and the right-hand side of \eqref{e.slecondind} is manifestly $\rev{\sigma( \eta(s); s\le t)}$-  measurable.
    \end{itemize}
\end{proof}

We are now ready to prove (\ref{i:nu}). 

\begin{proof}[Proof of (\ref{i:nu})]
We first address tightness, by showing tightness of each term for its respective topology. The tightness of $(\hat \Gamma_n)_{n\in \N}$, $(\eta_n)_{n\in \N}$, $(\hat\Gamma^T_n)_{n\in \N}$, $(I_n)_{n\in \N}$, $({\nu}_n)_{n\in \N}$ follows from the convergence proved in  (\ref{i:ggeta}), (\ref{i:etaI}) and (\ref{i:ghat}). The measures $m_n$ and $\rev{\mu}_n$ satisfy that for any compact set $K\subseteq \bar \HH$, $\rev{\sup_n}\E{m_n(D)}=\rev{\sup_n}\E{\mu_n(D)}<\infty $ \rev{(by the convergence in $L^2(\mathbb{P})$; Proposition \ref{P:confminkone})}. Thus they are tight for the vague convergence of measures on $\overline{\HH}$. Finally, $\rev{\mu}_n(A_n)$ is tight, since for any $\epsilon>0$ there exists a compact $K\subseteq \overline{\HH}$ such that $\P(A_n\subseteq K)>1-\epsilon$ for all $n\in \N$.

For the second statement of (\ref{i:nu}), we suppose that 
\[
(\hat{\Gamma}_n,\eta_n,m_n,\mu_n,\mu_n(A_n))\to (\hat\Gamma,\eta,m,\mu,X)
\]
in law along a subsequence. We can already observe that by Lemma \ref{l.etam}, $D\mapsto \mathbb{E}(m(D)|\eta)$ is almost surely equal to $m_\eta$, and by (\ref{i:ggeta}), the joint law of $(\hat\Gamma,\eta)$ is characterised. In particular, $\eta,m$ are independent of $\hat\Gamma$. 

Now, we apply Skorohod representation theorem, to find a probability space where the convergence is almost sure. On this probability space, we note  that since $m$ is a limit of measures satisfying the condition $(\mathcal E_\beta)$ for any $\beta<1+3/2$,  $\rev{\mu_{\hat\Gamma}^{1}[m]}$ is well-defined. Moreover
\begin{equation}\label{e:nuX}
    \mu=\mu_{\hat\Gamma}^{1}[m] \text{ and } X=\mu(\eta([S,T]))
\end{equation}
almost surely. Indeed, the first equality holds due to Proposition \ref{p.convergence_liouville}, since $(m_n)_{n\in \N}$ satisfies $(\mathcal E_\beta)$ for $\beta <1+3/2$. For the second, we apply Lemma \ref{l:convergingmeasuresandsets}. Since $\mu_n \to \mu$ and $\eta_n([S,T])\to \eta([S,T])$  we only need to check that for any $a>0$
\begin{align*}
\lim_{\epsilon\to 0}\sup_{n\in \N} \P(\mu_n((\eta_n([S,T])_\epsilon) -\mu_n(\eta_n([S,T])) \geq a)=0.
\end{align*}
This follows from the same proof as in Lemma \ref{c.uniformsepSLE}, using that the separation into events carried out there becomes a separation of events that are measurable with respect to $\eta$, and using that $\E{\mu_n(D)}= \E{m_n(D)}$ for all $n$.

Next, we claim that 
\begin{equation}\label{e:cond_nu}
\text{the measure }
D\mapsto  \E{\nu(D)\mid \hat\Gamma,\eta}  \text{ is equal to } \mu_{\hat\Gamma}^{1}[m].
\end{equation}
This is because, for any deterministic compact $D\subseteq \HH$, we have that $\mu_{\hat\Gamma}^{1}[m]$ is equal to the $L^1$ limit as $\epsilon \to 0$ of  $\int_{D}e^{\hat\Gamma_\epsilon(x)+\frac{1 }{2 }\log(\epsilon)}m(dx)$, and thus
\rev{\begin{align*}
\E{ \mu_{\hat\Gamma}^{1}[m] \mid \hat\Gamma,\eta} (D) &= \lim_{\epsilon\to 0}\E{\int_D e^{\hat\Gamma_\epsilon(x) +\frac{1 }{2 } \log(\epsilon)} m(dx)\mid \hat\Gamma, \eta } \\ 
&= \lim_{\epsilon \to 0 }\int_D e^{\hat\Gamma_\epsilon(x) +\frac{1}{2} \log(\epsilon)} m_\eta(dx) =\mu_{\hat\Gamma}^{1}[m](D).
\end{align*}}
Here in the second equality we used that $\hat\Gamma$ is independent of $(\eta,m)$.

Finally,  we combine \eqref{e:nuX} and \eqref{e:cond_nu}, to obtain that for any compact $K\subseteq \overline \HH$ \begin{align*}
\mathbb{E}(X\1_{\eta([0,T])\subseteq K} \mid\eta,\hat\Gamma)=\mu_{\hat\Gamma}^{1}[m](\eta([S,T]))\1_{\eta([0,T])\subseteq K}. \end{align*}
This is valid because on the event $ \{\eta([0,T])\subseteq K\}$, one can approximate $\eta([S,T])$ arbitrarily well by a union of dyadic squares, and this approximation at any fixed level can take only finitely many possible values. 
Letting $K\nearrow \HH$ we obtain that $\E{X\mid \eta,\hat\Gamma}=\mu_{\hat\Gamma}^{1}[m](\eta([S,T]))$ as desired. Note that this holds even though $X$ may not belong to $L^1$, because $X\geq 0$ and $X\1_{\eta([0,T])\subseteq K}\nearrow X$ almost surely.
\end{proof}

\appendix

\section{One-sided conformal Minkowski content}\label{s.appendix_Minkowski}

As usual, let $\kappa=\gamma^2\le 4$,
and let $\eta$ be an SLE$_\kappa$ in $\mathbb{H}$ from $0$ to $\infty$. For $z\in \HH$, let $\tau_r(z)$ be the first time that $\eta$ gets within distance $e^{-r}$ of $z$, and $T_r(z)$ be the first time that the conformal radius of the complement of the curve (as seen from $z$) reaches $e^{-r}$, so that $\tau_r(z)\le T_r(z)\le \tau_{r+\log 4}(z)$
by the Koebe quarter theorem. Recall that $A^L(z)$ denotes the event that $z$ lies on the left-hand side of $\eta([0,\infty))$, $A^R(z)$ denotes the event that it lies on the right hand side, and $d=1+\kappa/8 \in (1,3/2]$.
\rev{We also use the notation $\eta_s=\eta[0,s]$ for compactness in this section.}

The key to the proof of Theorem \ref{T:confminkone} is the following proposition, which says that masses of dyadic squares converge. Let $\mathcal{Q}_+$ denote the set of dyadic squares of the form $[j2^{-n},(j+1)2^{-n}]\times i[k2^{-n},(k+1)2^{-n}]$ with $n\in \mathbb{Z}_{\ge 0}$, $j\in \mathbb{Z}$ and $ k\in\mathbb{Z}, k>0$. \rev{For $r>0$ and $q=L,R$ we also set $$J_r^q(z)=e^{r(2-d)}\mathbf{1}_{\CR(z,\HH \setminus \eta)<e^{-r}}\I{A^q(z)}.$$} 
\begin{propn}\label{P:confminkone}
For any fixed dyadic square $\Gamma\in \mathcal{Q}_+$,  the limit
\begin{equation*}
\lim_{r\to \infty} e^{r(2-d)} \int_{\Gamma} \mathbf{1}_{\CR(z,\HH \setminus \eta)<e^{-r}}\I{A^q(z)} \, dz = \lim_{r\to \infty} \int_\Gamma J_r^q(z) \, dz =: \lim_{r\to \infty} J_r^q(\Gamma)=J^q(\Gamma)
\end{equation*}
exists almost surely and in $L^2(\mathbb{P})$, for $q=L,R$, and satisfies $\mathbb{E}(J^q(\Gamma))=\frac{1}{2}\int_\Gamma G(z) dz$. 
\end{propn}

\subsection{Proof of Proposition \ref{P:confminkone}}
As mentioned and proved in Section \ref{S:onesidedmink}, we will make use of the following sharp one-point estimate:  there are $ C,\alpha,r_0\in (0,\infty)$, depending only on $\kappa$ such that
\begin{equation}\label{E:greensLR}
\bigg|\frac{2e^{r(2-d)}}{G(z)}\mathbb{P}(T_{r}(z)<\infty,\I{A^q(z)})- 1 \bigg| \le C\big(\frac{e^{-r}}{\CR(z,\HH)}\big)^\alpha
\end{equation}
for all $q=L,R$, $z\in \HH$ and $r\ge r_0-\log(2\Im(z))$. This also implies that

\begin{equation}\label{E:greens}
\big|\frac{e^{r(2-d)}}{G(z)}\mathbb{P}(T_{r}(z)<\infty)- 1 \big| \le C\big(\frac{e^{-r}}{\CR(z,\HH)}\big)^\alpha
\end{equation}
for all  $z\in \HH$ and $r\ge r_0-\log(2\Im(z))$. 
\medskip 
We also make use of the following upper bounds. 

\begin{lemma}\label{L:upperbounds} There exists $c<\infty$, such that the following hold.
\begin{enumerate} 
\item \label{E:tau1} For all $u>0$ and $|z|,|w|\ge e^{-u}$ with $|z-w|\ge e^{-u}$, we have 
\[
    \mathbb{P}(T_{r+u}(z)<\infty, T_{s+u}(w)<\infty)  \le \mathbb{P}(\tau_{r+u}(z)<\infty, \tau_{s+u}(w)<\infty) \le ce^{(s+r)(d-2)}\] for $0<s<r$.
  \item\label{E:tau2} For all $u>0$ and $|z|,|w|\ge e^{-u}$ with $|z-w|\ge e^{-u}$ we have  
    \[
    \mathbb{P}(\tau_{s+u}(w)<\tau_{r+u}(z)<\tau_{r+u}(w)<\infty)  \le ce^{2r(d-2)}e^{-a s}\]
    for $0<s<r$, with $a=\frac{1}{2}(\frac{8}{\kappa}-1)$.
\item \label{E:tau3} For all $z,w$ with $\Im(z),\Im(w)\ge 1$ and $|z-w|\le 1$, 
\[
   \mathbb{P}(T_{r}(z)<\infty, T_{s}(w)<\infty) \le \mathbb{P}(\tau_{r}(z)<\infty, \tau_{s}(w)<\infty)\le ce^{(s+r)(d-2)}|z-w|^{d-2}
\]
\end{enumerate}
\end{lemma}

\begin{proof} This follows from \cite[Lemma 2.9]{LawlerRezaei} and \cite[Equation (19)]{LawlerRezaei}.
\end{proof}
\medskip

For the proof of Proposition \ref{P:confminkone}, most of the work will be in establishing the following lemma. Fix $q=L$ or $R$ and set 
$Q_{r,\delta}^q(z):=J_{r+\delta}^q(z)-J_r^q(z).$
\begin{lemma}\label{L:secondmomentred} 
There are real numbers $ u,s,\beta>0$ and $c<
\infty$ such that
\begin{equation}\label{E:secondmomentred} \mathbb{E}\left(Q_{r,\delta}^q(z)Q_{r,\delta}^q(w)\I{\tau_{r}(z)\le \tau_{sr}(w)< \tau_{r}(w)<\infty}\right) \le ce^{-\beta r}\end{equation}
for all $z,w\in \HH$ with $\Im(z),\Im(w)\ge 1$ and $|z-w|\ge e^{-ur}$.
\end{lemma}

Note that given Lemma \ref{L:secondmomentred}, choosing $\hat{u}$ smaller than $u$ if necessary so that $\hat{u}(2(2-d)+a)\le a s/2$, ($a$ as in Lemma \ref{L:upperbounds}) we have that for all $|z-w|\ge e^{-\hat{u}r}\ge e^{-ur}$, 
\[ \mathbb{E}\left(Q_{r,\delta}^q(z)Q_{r,\delta}^q(w)\I{\tau_{r}(z)\le \tau_{r}(w)<\infty}\right) \le ce^{-\beta r} +\mathbb{E}\left(Q_{r,\delta}^q(z)Q_{r,\delta}^q(w)\I{ \tau_{sr}(w) \le \tau_{r}(z)\le \tau_{r}(w)<\infty}\right)  \]
where by Lemma \ref{L:upperbounds} \eqref{E:tau2},
\begin{align*} \mathbb{E}\left(Q_{r,\delta}^q(z)Q_{r,\delta}^q(w)\I{\tau_{sr}(w)\le \tau_{r}(z)\le  \tau_{r}(w)<\infty}\right) & \le \hat{c} e^{2r(2-d)} \mathbb{P}\left(\tau_{sr}(w)\le \tau_{r}(z)\le \tau_{r}(w)\right) \\
& \le \hat{c} e^{2r(2-d)}e^{2(r-\hat{u}r)(d-2)}e^{-a(sr-\hat{u}r)} \\
& \le \hat{c} e^{\hat{u}r(2(2-d)+a)} e^{-a sr}\le \hat{c} e^{-a sr/2}.
\end{align*}
Above the constant $\hat c$ may vary from line to line but does not depend on $z,w$ with $\Im(z),\Im(w)\ge 1$ such that $|z-w|\ge e^{-\hat u r}$. By symmetry and since everything is zero unless $\tau_{r}(z), \tau_{r}(w)<\infty$, we obtain the following.

\begin{cor}\label{C:secondmoment}
$\exists$ $\beta,u>0$ and $c<\infty$ such that for all $r>0$, and all $z,w\in \HH$ with $\Im(z),\Im(w)\ge 1$, $|z-w|\ge e^{-ur}$:
\begin{equation}\label{E:secondmoment}
\mathbb{E}\left(Q_{r,\delta}^q(z)Q_{r,\delta}^q(w)\right)\le c e^{-r\beta} 
\end{equation}
\end{cor}
\rev{Using the bound \eqref{E:tau3} from Lemma \ref{L:upperbounds} we also have that for some $c<\infty$, it holds for all $\beta'>0$ and all $z,w\in \HH$ with $\Im(z),\Im(w)>1$ and $|z-w|\le e^{-ur}$, that $$\mathbb{E}\left(Q_{r,\delta}^q(z)Q_{r,\delta}^q(w)\right) \le c|z-w|^{d-2}\le ce^{-\beta' r}|z-w|^{d-({\beta'}/{u})-2}.$$ Thus by choosing $\beta'<ud\wedge \beta$ we get the existence of $\beta'>0, a>0$ and $c<\infty$, such that for all $z,w\in \HH$ with $\Im(z),\Im(w)>1$ and $z,w\in \Gamma$ for some $\Gamma$ (so $|z-w|\le 1$) 
\begin{equation}\label{E:secondmomentextended}
\mathbb{E}\left(Q_{r,\delta}^q(z)Q_{r,\delta}^q(w)\right)\le c e^{-r\beta'}|z-w|^{a-2}.
\end{equation}
}

\begin{proof}[Proof of Proposition \ref{P:confminkone} given Corollary \ref{C:secondmoment}]
Without loss of generality we fix $q=L$. 
By scaling, we may assume that $\inf\{\Im(z): z\in \Gamma\}\ge 1$.  
We then argue in exactly the same way as \cite{LawlerRezaei}. 
In summary:
\begin{itemize}
    \item  Fix $\delta\in (0,1/10)$. \rev{Integrating \eqref{E:secondmomentextended} over $z$ and $w$ implies that 
    $$0\le \mathbb{E}\left(|J^L_r-J^L_{r+\delta}|^2\right)=\mathbb{E}\left(\iint\nolimits_{\Gamma\times \Gamma} Q^L_{r,\delta}(z)Q^L_{r,\delta}(w) \, dz dw\right)\le ce^{-\beta' r}. $$}
  \rev{(Note that we automatically have a lower bound by $0$ on the left-hand side so we only need the upper bound for $\mathbb{E}(Q_{r,\delta}^q(z)Q_{r,\delta}^q(w))$ and do not need to take absolute values.)}
    In particular, 
    \[
    \sum_{n\ge 1} \mathbb{P}(|J^L_{n\delta}-J^L_{n\delta+\delta}|>e^{-\beta' n \delta/4})<\infty
    \]
    and therefore $J^L_{n\delta}$ is a Cauchy sequence with probability one. This implies that on an event of probability one, the sequence $(J^L_{n\delta})_{n\ge 0}$ converges for all $\delta=2^{-m}$ with $m\ge 4$.
    \item Let $J^L_\infty$ be the almost sure limit; note that this has to coincide for different $m$. Then
    \[ 
    e^{-\delta(\rev{2-d})}J^L_\infty \le \liminf_{r\to \infty} J^L_r \le \limsup_{r\to \infty} J^L_r \le e^{\delta(2-d)}J^L_\infty.
    \] 
    Since this holds for all $\delta=2^{-m}$ we see that $J^L_r\to J^L_\infty$ almost surely as $r\to \infty$.
\end{itemize}
The convergence holds in $L^2(\mathbb{P})$ since fixing some arbitrary $\delta\in (0,1/10)$, we have $\mathbb{E}(|J^L_{r}-J^L_\infty|^2)\leq (\sum_{m\ge 0} e^{-\beta' (r+m\delta)/2})^2$ which converges to $0$ as $n\to \infty$. The claim about the expectation follows from \eqref{E:greensLR}.
\end{proof}

\vspace{.5cm}
To show Proposition \ref{P:confminkone}, it thus remains to prove Lemma \ref{L:secondmomentred}, since this implies Corollary \ref{C:secondmoment}.

\begin{proof}[Proof of Lemma \ref{L:secondmomentred}]
\begin{figure}
\includegraphics[width=\textwidth]{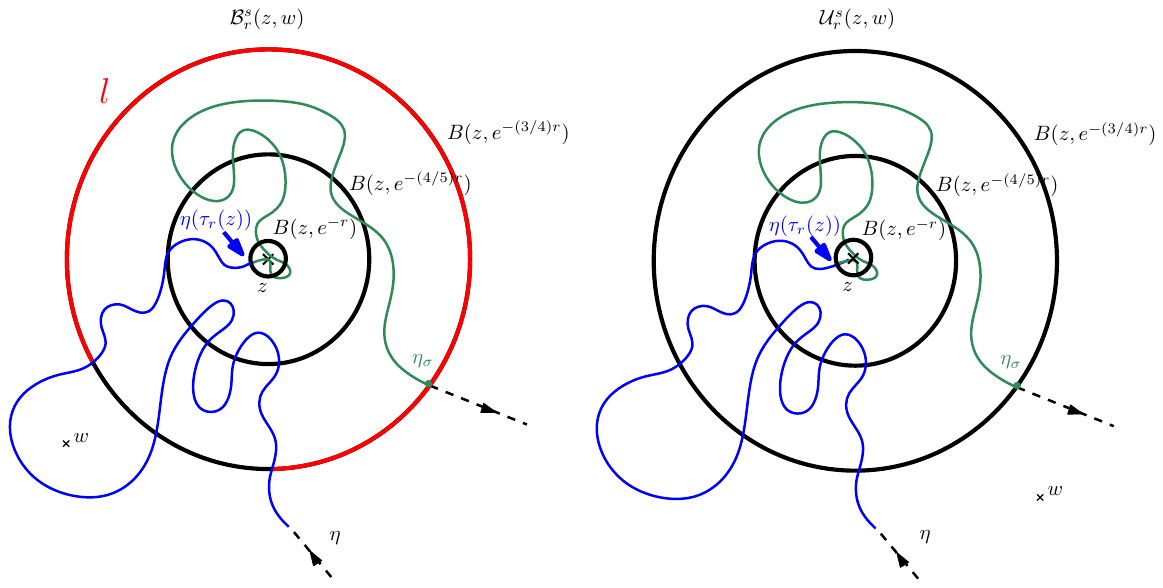}
\caption{\rev{Representation of the notation used in the proof. In blue is the curve $\eta$ run up to the first time it hits the ball of center $z$ and radius $r$, and in green is the curve run from this time until the first time that it hits $l$. The left figure illustrates the event $\mathcal{B}_r^s(z,w)$ while the right illustrates the event $\mathcal{U}_r^s(z,w)$}.}
\label{f.Appendix2}
\end{figure}
As in \cite{LawlerRezaei}, we will prove \eqref{L:secondmomentred} with $p=(\tfrac{1}{\kappa}-\tfrac{1}{8})$,  $u:=p/(3(2-d))\in (0,\tfrac14)$. 
In particular,  $|z-w|\ge e^{-r/4}\ge e^{-r/2}$ if $|z-w|\ge e^{-ur}$. In everything that follows we write $a \lesssim b$ to indicate that for some $c$ not depending on $r$, $\delta\in (0,1/10)$ or $z,w\in \HH$ with $\Im(z),\Im(w)\ge 1$ and $|z-w|\ge e^{-ur}$, we have $a\le c b$.

Following the notation of \cite{LawlerRezaei} \rev{(see Figure \ref{f.Appendix2})}, on the event $\tau_r(z)<\infty$, we let $V$ be the connected component of $(\HH\setminus \eta_{\tau_r(z)})\cap B(z,e^{-(3/4)r})$ containing $z$, and $l$ be the unique connected arc of $\partial B(z,e^{-(3/4)r})$ in 
$\partial V\cap (\HH\setminus \eta_{\tau_r(z)})$ such that $z$ is in the bounded connected component of $(\HH\setminus \eta_{\tau_r(z)})\setminus l$. 

For $s\in (0,1/8)$, let $\mathcal{U}^s_r(z,w)$ be the indicator function of the event that $\tau_r(z)<\infty$, $\tau_r(z)<\tau_{sr}(w)$ and $w$ is in the unbounded component of $(\HH\setminus  \eta_{\tau_r(z)})\setminus l$. Let $\mathcal{B}^s_r(z,w)$ be the indicator function of the event that $\tau_r(z)<\infty$, $\tau_r(z)<\tau_{sr}(w)$ and $w$ is in the bounded component \rev{of $(\HH\setminus  \eta_{\tau_r(z)})\setminus l$}. It is shown in \cite[p1118]{LawlerRezaei}  that for $u$ as defined at the start of the proof,
\begin{equation*}\mathbb{E}\left(\I{\tau_{r}(z)< \tau_{r}(w)<\infty}\mathcal{B}^s_r(z,w)\right) \lesssim e^{- pr}e^{-2(2-d)r},\end{equation*}
and therefore 
\[\mathbb{E}\left(Q_{r,\delta}^q(z)Q_{r,\delta}^q(w)\I{\tau_{r}(z)< \tau_{r}(w)<\infty}\mathcal{B}^s_r(z,w)\right)\lesssim e^{- pr}\]
for $z,w\in \HH$ with $\Im(z),\Im(w)\ge 1$ and $|z-w|\ge e^{-ur}$. Note that this works for any $s\in (0,1/8)$.

Thus we are left to prove that for some $\beta>0$, $s\in (0,1/8)$,
\begin{equation}\label{final}\mathbb{E}\left(Q_{r,\delta}^q(z)Q_{r,\delta}^q(w)\I{\tau_{r}(z)< \tau_{sr}(w)<\tau_{r}(w)<\infty}\mathcal{U}^s_r(z,w)\right) \lesssim e^{- \beta r}.\end{equation} Let \rev{$\sigma$ be the first time after $\tau_r(z)$ that $\eta$ hits $l$, and let} $\tilde{Q}^q_{r,\delta}(z)$ be the same as $Q_{r,\delta}^q(z)$ but stopping the curve $\eta$ at time $\sigma$. \rev{That is,} we replace the event $\CR(z,\HH\setminus \eta)<e^{\rev{-}r}$ with the event $\CR(z,\HH\setminus \eta_\sigma)<e^{-r}$, and the event ${A^q(z)}$ with the event $\tilde{A}^q(z)$ that the connected component of $B(z,e^{-(4/5)r})\cap (\HH\setminus \eta_\sigma))$ containing $z$ is bounded by the left-hand side of $\eta$ for $q=L$ or the right-hand side for $q=R$.\footnote{The centred Loewner map at time $\sigma$ will send this component to an open set in $\HH$ which will contain part of the negative or the positive real line on its boundary (but not both). If it is the negative real line then $\tilde{A}^L(z)$ will be one, otherwise $\tilde{A}^R(z)$ will be one.}
Then we have that the expression on the left of \eqref{final} can be rewritten as the sum of \begin{equation} \label{final1} \mathbb{E}\left((Q_{r,\delta}^q(z)-\tilde{Q}^q_{r,\delta}(z))Q_{r,\delta}^q(w)\I{\tau_r(z)<\tau_{r}(w)<\infty} \mathcal{U}^s_r(z,w)\right) \end{equation} and \begin{equation}\label{final2} \mathbb{E}\left(\tilde{Q}^q_{r,\delta}(z) \mathcal{U}^s_{r}(z,w) \mathbb{E}(Q_{r,\delta}^q(w)\I{\tau_r(w)<\infty}\mid \eta_{\sigma})\right),\end{equation}
(since $\tilde{Q}_{r,\delta}^q(z)$  and $\mathcal{U}_r^s(z,w)$ are $\eta_\sigma$-measurable).

To bound \eqref{final2}, we \rev{use \eqref{E:greensLR} and conformal invariance, which gives that for chordal SLE in a general domain $D$, $\mathbb{P}(\CR(z,D\setminus \eta)<e^{-r})\lesssim (e^{-r}/\CR(z,D))^{(2-d+\alpha)}$}. This together with the observation that on the event $\{\tau_r(z)<\tau_{sr}(w)\}$ we have $\CR(w,\HH\setminus \eta_\sigma)\ge e^{-sr}$, yields
\[
\left|\mathbb{E}\left(Q_{r,\delta}^q(w)\I{\tau_r(w)<\infty}\mid \eta_{\sigma}\right)\right|\le C e^{r(2-d)}e^{(s-1)r(2-d+\alpha)}.
\]
Therefore, combining with the fact that $\mathbb{E}(\mathcal{U}^s_r(z,w))\lesssim  e^{-r(2-d)}$  by \eqref{E:greens}, and choosing $s>0$ small enough, we have 
\begin{equation}\label{final2conc} \mathbb{E}\left(\tilde{Q}_{r,\delta}^q(z) \mathcal{U}^s_{r}(z,w) \mathbb{E}(Q_{r,\delta}^q(w)\I{\tau_r(w)<\infty}\mid \eta_{\sigma})\right)\lesssim e^{-r\alpha/2}. \end{equation}

Let us finally turn to \eqref{final1}. 
\rev{The idea is that in order for $Q_{r,\delta}^q(z)-\tilde{Q}_{r,\delta}^q(z)$ to be non-zero, we either need the curve $\eta$ to enter $B(z,e^{\scriptscriptstyle{-(4/5)r}})$ again after time $\sigma$ (which is shown to be very unlikely in \cite{LawlerRezaei}), or $\CR(z,\HH\setminus \eta_\sigma)$ has to be in a very small interval (which is also extremely unlikely). 

More precisely, for $Q_{r,\delta}(z)\ne \tilde{Q}_{r,\delta}(z)$, we either need that: (i) $\sigma<T_r(z)<\infty$; (ii) $T_r(z)\le \sigma<T_{r+\delta}(z)<\infty$, or (iii) $\tilde{A}^q(z)\ne A^q(z)$. (iii) requires that $\eta$ enters $B(z,e^{\scriptscriptstyle{-(4/5)r}})$ again after time $\sigma$. Furthermore, if $\CR(z,\HH\setminus \eta)\le y(r)\CR(z,\HH\setminus \eta_\sigma)$ then by Koebe's distortion estimates it must be that $\eta$ hits $B(z,(1-y(r))^{-2}\CR(z,\HH\setminus \eta_\sigma))$ again after time $\sigma$. Since $\CR(z,\HH\setminus \eta_\sigma)\le 4e^{-r}$, choosing $y(r)=\exp(-4e^{-r/10})$, we have $(1-y(r))^{-2}\CR(z,\HH\setminus \eta_\sigma)\le e^{-(4/5)r}$. This means that (fixing this definition of $y(r)$ from now on) if $\CR(z,\HH\setminus \eta)\le y(r)\CR(z,\HH\setminus \eta_\sigma)$, then $\eta$ must enter $B(z,e^{\scriptscriptstyle{-(4/5)r}})$ again after time $\sigma$. On the other hand, if $\CR(z,\HH\setminus \eta)>y(r)\CR(z,\HH\setminus \eta_\sigma)$, then in order for (i) or (ii) to occur, we must have $\CR(z,\HH\setminus \eta_\sigma)\in (e^{-r},e^{-r}/y(r)]$ or $\CR(z,\HH\setminus \eta_\sigma)\in (e^{-(r+\delta)},e^{-(r+\delta)}/y(r)]$. 

The upshot of the above argument is the following. Let $\rho$ be the first time after $\sigma$ that $\eta\in B(z,e^{\scriptscriptstyle{-(4/5)r}})$. Then
\begin{equation}\label{QneQ}
     \{Q_{r,\delta}(z)\ne \tilde{Q}_{r,\delta}(z)\}\subset 
 \left\{\rho<\infty\right\} \cup E \cup E' 
 \end{equation}
 where 
 \[ E:=\{\CR(z,\HH\setminus \eta_\sigma)\in (e^{-r},\frac{e^{-r}}{y(r)}]\} \text{ and } E'=\{\CR(z,\HH\setminus \eta_\sigma)\in (e^{-(r+\delta)},\frac{e^{-(r+\delta)}}{y(r)}]\}.\]

We deal with the event $\{\rho<\infty\}$ first. The same argument as in \cite[see eq. (49)]{LawlerRezaei} implies that 
\[ \mathbb{P}(\rho<\infty, \tau_r(w)<\infty | \eta_\sigma)\lesssim e^{r(2-d)}e^{-r\beta} \]
for $\beta>0$ small enough. Hence, using again that $\mathbb{P}(\mathcal{U}_r^s(z,w))\lesssim e^{-r(2-d)}$, we have
\begin{align}\label{final1conc2}
    \left|\mathbb{E}\left((Q_{r,\delta}^q(z)-\tilde{Q}^q_{r,\delta}(z))Q_{r,\delta}^q(w)\I{\rho<\infty,\tau_{r}(w)<\infty} \mathcal{U}^s_r(z,w)\right)\right| & \lesssim e^{2r(2-d)} \mathbb{E}\left(\mathcal{U}_r^s(z,w) \mathbb{P}(\rho<\infty, \tau_r(w)<\infty \mid \eta_\sigma)\right) \nonumber \\
    & \lesssim e^{-\beta r}. 
\end{align}

For the event $E$, we start with the trivial upper bound
\[\left|\mathbb{E}\left((Q_{r,\delta}^q(z)-\tilde{Q}_{r,\delta}^q(z))Q_{r,\delta}^q(w)\I{E,\tau_{r}(w)<\infty} \mathcal{U}^s_r(z,w)\right)\right| \lesssim e^{2r(2-d)} \mathbb{E}\left(\mathbb{P}(\tau_r(w)<\infty \mid \eta_\sigma)\mathbf{1}_{E}\mathcal{U}^s_r(z,w)\right).\]
Arguing as in the proof of \eqref{final2}, we almost surely have $\mathbb{P}(\tau_r(w)<\infty | \eta_\sigma)\lesssim e^{(s-1)r(2-d)}$. On the other hand, we claim that 
\begin{equation}\label{claimGc}\mathbb{E}(\mathbf{1}_{E} \mathcal{U}^s_r(z,w))\lesssim e^{-r(2-d+q)}
\end{equation} for some $q>0$. Indeed, we can bound $\mathbb{E}(\mathbf{1}_{E} \mathcal{U}^s_r(z,w))$ above by
\rev{\[\mathbb{P}(\mathcal{U}_r^s(z,w)\mathbf{1}_{E}\mathbb{P}( \CR(\rev{z},\HH\setminus \eta)\le y(r)\CR(z,\HH\setminus \eta_\sigma)\mid \eta_\sigma))+\mathbb{P}(\CR(\rev{z},\HH\setminus \eta)\in [y(r)e^{-r},e^{-r}/y(r))).\]}
The second term decays as required by \eqref{E:greens}. Moreover, the event in the first term is contained in the event $\{\rho<\infty\}$, and  the same argument as in \cite[Lemma 2.6]{LawlerRezaei} shows that $\mathbb{P}(\rho<\infty \mid \eta_\sigma)\le e^{-br}$ for some $b>0$. Using again that $\mathbb{P}(\mathcal{U}_r^s(z,w))\lesssim e^{-r(2-d)}$, this proves \eqref{claimGc}. We thus conclude that 
\begin{align}
\label{final1conc1}\left|\mathbb{E}\left((Q_{r,\delta}^q(z)-\tilde{Q}_{r,\delta}^q(z))Q_{r,\delta}^q(w)\I{E,\tau_{r}(w)<\infty} \mathcal{U}^s_r(z,w)\right)\right| 
& \lesssim e^{2r(2-d)} e^{(s-1)r(2-d)}e^{-r(2-d)}e^{-qr} \\ & \lesssim e^{-\beta r} \nonumber
\end{align}
for $s>0$ and $\beta>0$ small enough.
The same argument works replacing $E$ with the event $E'$, and we get an analogous bound to \eqref{final1conc1}. Putting this together with \eqref{final1conc1} and \eqref{final1conc2}, and using the observation \eqref{QneQ}, yields \eqref{final1}.}

\end{proof}

\subsection{Proof of Theorem \ref{T:confminkone}}

The proof of Theorem \ref{T:confminkone} given Proposition \ref{P:confminkone} follows standard arguments, almost identical to those in \cite{LawlerRezaei}, so we keep this section brief and mostly make use of the work already done in \cite{LawlerRezaei}.

\begin{lemma}\label{lem:weakconv}
For $q=L,R$, $m_\eta^{q,u}$ converges a.s.\,with respect to the vague topology of measures on $\overline{\HH}$. The limiting measure $m_\eta^q$ satisfies $m_\eta^q(\Gamma)=m_\eta^q(\mathrm{int}(\Gamma))=J_i(\Gamma)$ a.s.\,for all $\Gamma\in \mathcal{Q}_+$ where $J_i(\Gamma)$ is the almost sure limit from Proposition \ref{P:confminkone}. Moreover, for any $K>0$, $m_\eta^q([-K,K]+i[0,2^{-n}])\to 0$ a.s.\,\rev{as $n\to \infty$.} In other words, the limit measure $m_\eta^q$ gives no mass to the real line.
\end{lemma}

\begin{proof}
Let us fix $q=L$ or $R$ for the remainder of the proof.
First, we note that by the K{o}ebe quarter theorem, for any $A\subsetneq B\subset \HH$ we have
\begin{equation}\label{eq:comparecont}
\limsup_{u\to 0} m_\eta^{q,u}(A) \le \limsup_{u\to 0} u^{d-2} \mathrm{Area}(\{z: d(z,\eta \cap B) \le u\})=:\mathrm{Cont}_d^+(\eta \cap B).
\end{equation}
This immediately gives us, by \cite[Lemma 3.7]{LawlerRezaei}, that for any $K>0$, on an event of probability one, $ \limsup_{u\to 0} m_\eta^{q,u}([-K,K]\times i[0,2^{-n}])\le 2^{-n}$ for all $n$ large enough. 
It thus suffices to fix $m\in \mathbb{N}$ and prove a.s.\,weak convergence of $m_\eta^{q,u}$ to a limit $m_\eta^q$ on $D_m=[-2^m,2^m]\times i[2^{-m},2^m]$ such that  $m_\eta^q(\Gamma)=m_\eta^q(\mathrm{int}(\Gamma))=J_i(\Gamma)$ for each $\Gamma \in \mathcal{Q}_+$ with $\Gamma\subset D_m$. If we take an event $\Omega_0$ of probability one where $m_\eta^{q,u}(\Gamma)\to J_i(\Gamma)$ for every $\Gamma\in \mathcal{Q}_+$, it is clear that on $\Omega_0$, $m_\eta^{q,u}(D_m)$ is uniformly bounded and thus $m_\eta^{q,u}|_{D_m}$ is tight in $u$. It is standard and easy to prove using the Portmanteau theorem and first moment estimates (as in, e.g. \cite[Section 6]{Ber}) that any limit $m$ must satisfy $m(\Gamma)=m(\mathrm{int}(\Gamma))=J_i(\Gamma)$ for all $\Gamma$, which identifies the limit uniquely, and hence proves almost sure weak convergence to a measure with the desired property.
\end{proof}

\begin{lemma}\label{lem:measurability}
For $q=L,R$, let $m_\eta^q$ be as constructed in Lemma \ref{lem:weakconv}. Then for bounded $D\subset \HH$ with $d(D,\R)>0$, $m_\eta^q(\eta([0,t])\cap D)$ is $\sigma(\eta(s);s\le t)$-measurable for each $t>0$.
\end{lemma}

\begin{proof} Fix $D$ and $t>0$, and let $q=L$ without loss of generality. It suffices to show that $m_\eta^L(\eta([0,t])\cap D)$ is $\eta([0,t+\delta])$-measurable for arbitrary $\delta>0$. Let $V_n$ denote the intersection of $D$ with the union of all $\Gamma\in \mathcal{Q}_+$ of the form $[j2^{-n},(j+1)2^{-n}]\times [k2^{-n},(k+1)2^{-n}]$ with $k\ge 1$ that intersect $\eta([0,t])$. Then
\begin{align*} m_\eta^L(\eta([0,t]))& =\lim_{n\to \infty} m_\eta^L(V_n)+m_\eta^q(\eta([0,t])\cap\{z\in \HH: \Im(z)\le 2^{-n}\} \\
& =\lim_{n\to \infty} \lim_{u\to 0} u^{d-2}\int_{V_n} \1_{\{\CR(z,\HH\setminus \eta)\le u\} \cap A^L(z)} dz 
\end{align*}
where the second equality follows by Lemma \ref{lem:weakconv}. Now, let $a_n:=\sup_{z\in V_n} \frac{\CR(z,\HH\setminus \eta[0,t+\delta])}{\CR(z,\HH\setminus \eta)}$, so by definition of $V_n$ and since $\eta$ is a.s.\,non self-intersecting, $a_n\to 1$ a.s.\,as $n\to \infty$. Also write $A^L_{t+\delta}(z)$ for the event that $\Re(f_{t+\delta}(z))<0$, where $(f_t)_t$ is the centered Loewner flow for $\eta$. Then we also have that with probability one, there exists an $N>0$ such that $A_{t+\delta}^L(z)=A^L(z)$ for all $z\in V_n$ and $n\ge N$. We rewrite
\begin{align*} m_\eta^L(\eta([0,t]))& =\lim_{n\to \infty}   \lim_{u\to 0} \left(u^{d-2}\int_{V_n} \1_{\{\CR(z,\HH\setminus \eta[0,t+\delta])\le u\} \cap A^L_{t+\delta}(z)} dz + u^{d-2}\int_{V_n} \1_{B_u(z)}  dz\right)
\end{align*}
where $B_u(z)$ is the symmetric difference of $\{\CR(z,\HH\setminus \eta[0,t+\delta])\le u\} \cap A^L_{t+\delta}(z)$ and $\{\CR(z,\HH\setminus \eta)\le u\} \cap A^L(z)$. It suffices to prove that $\lim_{n\to \infty} \limsup_{u\to 0} u^{d-2}\int_{V_n}\1_{B_u(z)} dz=0$, since  $V_n$ is $\sigma(\eta(s);s\le t)$-measurable and $\{\CR(z,\HH\setminus \eta[0,t+\delta])\le u\}$ and $A^L_{t+\delta}(z)$ are $\sigma(\eta(s);s\le t+\delta)$-measurable for every $u,n,z$. To show this write
\begin{align*}
\lim_{n\to \infty}\limsup_{u\to 0} u^{d-2}\int_{V_n}\1_{B_u(z)} \le & \lim_{n\to \infty}\limsup_{u\to 0} u^{d-2}\int_{V_n} \1_{\CR(z,\HH\setminus \eta)\in [u/a_n, u]} dz \\ 
& + \lim_{n\to \infty}\limsup_{u\to 0} u^{d-2}\int_{V_n} \1_{\CR(z,\HH\setminus \eta)\le u} \1_{A_{t+\delta}^L(z)\ne A^L(z)} dz. 
\end{align*}
As already observed, with probability one, the $\limsup$ over $u$ in the second term on the right is equal to $0$ for all sufficiently large $n$. By Lemma \ref{lem:weakconv} the $\limsup$ over $u$ in the first term is less than or equal to $(1-a_n^{d-2})(m_\eta^L(V_n)+m_\eta^R(V_n))$ for each $n$. Since $a_n\to 1$ a.s.\,as $n\to \infty$, this completes the proof.
\end{proof}

\begin{lemma}\label{lem:cts}
In the set-up of Lemma \ref{lem:measurability}, $m_\eta^q(\eta([0,t])\cap D)$ is a.s.\,an increasing continuous process in $t$.
\end{lemma}

\begin{proof}
The fact that it is increasing in $t$ follows from the fact that $\eta$ is increasing, and $m_\eta^q$ is a positive measure. The continuity follows by applying \eqref{eq:comparecont} and using the fact that, as proved in \cite[Proposition 3.8]{LawlerRezaei}, there exists $\alpha>0$ such that with probability one for every $t<\infty$, all $s\le t$ and all $n$ sufficiently large$, \mathrm{Cont}_d^+(\eta[s,s+2^{-n}])\le 2^{-n\alpha}$.
\end{proof}

\begin{lemma}\label{lem:minkchar2} For $t>0$ and $D\subset \HH$ with $\rev{\mathrm{dist}}(D,\R)>0$, 
$\mathbb{E}\left(m_\eta^q(D)\, | \, \sigma(\eta(s);s\le t)\right)=m_\eta^q(D\cap\eta([0,t]))+\frac{1}{2} \int_{D\setminus \eta([0,t])} |g_t'(z)|^{2-d} G(g_t(z)-W_t) \, dz$ a.s., where $W$ is the driving function of $\eta$.
\end{lemma}

\begin{proof}
Due to Lemma \ref{lem:measurability}, it suffices to prove that 
\[\mathbb{E}\left(m_\eta^q(D\setminus \eta([0,t])\mid \sigma(\eta(s);s\le t)\right)=\frac{1}{2} \int_{D\setminus \eta([0,t])} |g_t'(z)|^{2-d} G(g_t(z)-W_t) \, dz.\]
Recall the definition of $V_n$ from the proof of Lemma \ref{lem:measurability}. We have $m_\eta^q(D\setminus \eta([0,t]))= \lim_{n\to \infty} m_\eta^q(D\setminus V_n)=\lim_{n\to \infty}\lim_{u\to 0} u^{d-2} \int_{D\setminus V_n} \1_{\CR(z,\HH\setminus \eta)<u}\1_{A^q(z)} dz$ where the limits hold almost surely and in $L^1(\mathbb{P})$ (by \ref{P:confminkone} for the limit in $u$ and monotone convergence for the limit in $n$). Since conditioning is a contraction in $L^1(\mathbb{P})$ we have 
\begin{align*} \mathbb{E}\left(m_\eta^q\left(D\setminus \eta([0,t])\right)\mid \sigma(\eta(s);s\le t)\right) & = \lim_{n\to \infty}\lim_{u\to 0} u^{d-2} \int_{D\setminus V_n} \mathbb{P}\left(\{\CR(z,\HH\setminus \eta)<u \}\cap A^q(z) \, | \, \sigma(\eta(s);s\le t)\right) \, dz \\
& = \lim_{n\to \infty}\int_{D\setminus V_n} |g_t'(z)|^{2-d} \frac{1}{2} G(g_t(z)-W_t) \\
& = \frac{1}{2} \int_{D\setminus \eta([0,t])} |g_t'(z)|^{2-d} G(g_t(z)-W_t) 
\end{align*}
where the second line follows from a change of variables, the Markov property of SLE and \eqref{E:greens}.
\end{proof}

\begin{proof}[Proof of Theorem \ref{T:confminkone}]
This follows by combining Lemmas \ref{lem:weakconv}, \ref{lem:measurability}, \ref{lem:cts} and \ref{lem:minkchar2}.
\end{proof}

\section{Continuity results for conditional expectations and measures}

In this section, we discuss some basic properties regarding the convergence of conditional expectations and measures. It is unlikely that these results are new but we provide a self-contained proof of each one.

This first lemma discusses the convergence of conditional expectations. 
\begin{lemma}\label{l.convergence_conditional_expectation}
Assume that $X_n\in \R$ and $Y_n \in \mathcal{Y}$ where $\mathcal{Y}$ is a Polish space are such that
\begin{align*}
X_n \stackrel{L^1}{\to} X, && Y_n \stackrel{\P}{\to} Y && \E{X_n\mid Y_n}\stackrel{\P}{\to} Z,
\end{align*}
then 
$\E{X_n\mid Y_n}$ converges in $L^1$ to $Z$ and 
\begin{align*}
\E{Z \mid Y } = \E{X\mid Y}.
\end{align*}
In particular, if $Z$ is $\sigma(Y)$-measurable, $Z=\E{X\mid Y}$.
\end{lemma}
\begin{proof}
    As $X_n$ converges in $L^1$, the sequence $(X_n)_{n\in \N}$ is uniformly integrable, thus the sequence $(\E{X_n\mid Y_n})_{n\in \N}$ is also uniformly integrable, and thus $\E{X_n\mid Y_n}$ converges in $L^1$ to $Z$. Furthermore, for any continuous and bounded function $f: \mathcal{Y}\mapsto \R$, we have that
    \begin{align*}
    \E{Xf(Y) } = \lim_{n\to \infty}\E{X_n f(Y_n) } = \lim_{n\to \infty}\E{\E{X_n\mid Y_n} f(Y_n) } = \E{Z f(Y) },
    \end{align*}
    thus $\E{X\mid Y} = \E{Z\mid Y}$. The case where $Z$ is $\sigma(Y)$-measurable follows directly.
\end{proof}

This second lemma gives sufficient condition for when the conditional independence is kept in the limit.
\begin{lemma}\label{l.conditional_independence_limit}
    Assume that you have a sequence of random variables $(X_n,Y_n,Z_n) \in \mathcal{X}\times \mathcal{Y} \times \mathcal{Z}$, where $\mathcal{X}$, $\mathcal{Y}$ and $\mathcal{Z}$ are \rev{P}olish spaces, such that conditionally on $X_n$, $Y_n$ is independent of $Z_n$. Furthermore, assume that for any continuous and bounded function $f:\mathcal{Y}\mapsto \R$ 
    \begin{align*}
    X_n \stackrel{\P}{\to} X, && Y_n \stackrel{\P}{\to} Y,  && Z_n\stackrel{\P}{\to} Z, && \E{f(Y_n)\mid X_n} \stackrel{\P}{\to}  w^Y, 
    \end{align*}
    and  $w^Y$ is $\sigma(X)$-measurable. Then conditionally on $X$, $Z$ and $Y$ are independent.
\end{lemma}
\begin{proof}
    We start by noting that for any continuous and bounded function $g:\mathcal{Z}\mapsto \R$, there is $w^Z$ such that $ \E{g(Z_n)\mid X_n}$ converges to $ w^Z$ and  $\E{f(Y_n)g(Z_n)\mid X_n}$  to $w^Yw^Z$, both in $L^1$. Using Lemma \ref{l.convergence_conditional_expectation} we have that $w^Y=\E{\rev{f}(Y)\mid Z}$ and 
    \begin{align*}
    \E{f(Y)g(Z)\mid X } = \E{ w^Y W^Z \mid X} = w^Y \E{ w^Z \mid X}= \E{f(Y)\mid X} \E{w^Z \mid X }.
    \end{align*}
    We conclude by noting that $
    \E{w^Z\mid X}=\E{g(Z)\mid X}$ as one can take $f=1$.
\end{proof}

This third lemma gives conditions for having convergence of the measure on a set when both the measures and the sets are converging.
\begin{lemma}\label{l:convergingmeasuresandsets}
Suppose that $\mathfrak{m}_n$ is a sequence of random measures that converge almost surely to a measure $\mathfrak{m}$ and that $A_n$ is a sequence of random closed sets that converge to $A$ almost surely in the Hausdorff topology. Suppose further that  for any $a>0$ 
\[\sup_n\P(\mathfrak{m}_n((A_n)_\eps)-\mathfrak{m}_n(A_n)\ge a)\to 0\] as $\eps\to 0$
where for a closed set $K$, $K_\eps$ denotes the set $\{z: d(z,K)<\eps\}$. \rev{Then $\mathfrak{m}_n(A_n)\to \mathfrak{m}(A)$ almost surely as $n\to \infty$.}
\end{lemma}

\begin{proof}
First observe that for any $\epsilon$, we have $A_n\subset A_\eps$ eventually. As a consequence, since $A$ is closed, we have that
\begin{align*}
\limsup_{n\to \infty} \mathfrak{m}_n(A_n)\leq \mathfrak{m}(A).
\end{align*}
The analogous argument, using that there exists a random sequence $\epsilon_n\searrow 0$ almost surely such that $A\subseteq (A_n)_{\eps_n}$ for all $n$, implies that
\begin{align*}
\mathfrak{m}(A) \leq \liminf_{n\to \infty} \mathfrak{m}_n((A_n)_{\eps_n}).
\end{align*}
Thus, we can conclude if we prove that for any random $\epsilon_n\searrow 0$ we have that 
\begin{align}\label{e.measure_not_growing_too_much}
\mathfrak{m}_n((A_n)_{\eps_n})-\mathfrak{m}_n(A_n)
\to 0
\end{align}
in probability.
This follows by noting that for any $\epsilon>0$
\begin{align*}
&\P(\mathfrak{m}_n((A_n)_{\eps_n})-\mathfrak{m}_n(A_n) \geq a) \\ 
&\leq \P(\epsilon_n>\epsilon) + \P(\mathfrak{m}_n((A_n)_{\eps})-\mathfrak{m}_n(A_n) \geq a) 
\end{align*}
which converges to $0$ by our assumption.
\end{proof}

\bibliographystyle{alpha}	
\bibliography{biblio}

\begin{thebibliography}{AHPS23}

\bibitem[AG25]{CLE4AG}
Morris Ang and Ewain Gwynne.
\newblock Critical {L}iouville quantum gravity and {CLE}$_4$.
\newblock {\em Annales de l'Institut Henri Poincar\'{e} Probabilit\'{e}s et
  Statistiques}, to appear, 2025.

\bibitem[AHPS23]{criticalMOT}
Juhan Aru, Nina Holden, Ellen Powell, and Xin Sun.
\newblock Brownian half-plane excursion and critical {L}iouville quantum
  gravity.
\newblock {\em J. Lond. Math. Soc. (2)}, 107(1):441--509, 2023.

\bibitem[AHS23]{AHSintegrability}
Morris Ang, Nina Holden, and Xin Sun.
\newblock {Integrability of SLE via conformal welding of random surfaces}.
\newblock {\em Communications on Pure and Applied Mathematics}, 2023.

\bibitem[ALS20]{ALS1}
Juhan Aru, Titus Lupu, and Avelio Sep{\'u}lveda.
\newblock {First passage sets of the 2D continuum Gaussian free field}.
\newblock {\em Probability Theory and Related Fields}, 176(3-4):1303--1355,
  2020.

\bibitem[APS19]{APS2}
Juhan Aru, Ellen Powell, and Avelio Sep\'{u}lveda.
\newblock Critical {L}iouville measure as a limit of subcritical measures.
\newblock {\em Electronic Communications in Probability}, 24, 2019.

\bibitem[APS20]{APS0}
Juhan Aru, Ellen Powell, and Avelio Sep{\'u}lveda.
\newblock {Liouville measure as a multiplicative cascade via level sets of the
  Gaussian free field}.
\newblock {\em Annales de l'Institut Fourier}, 70(1):205--245, 2020.

\bibitem[ARS23]{ARSintegrability}
Morris Ang, Guillaume Remy, and Xin Sun.
\newblock {FZZ formula of boundary Liouville CFT via conformal welding}.
\newblock {\em Journal of the European Mathematical Society}, 2023.

\bibitem[Ben18]{Benoist}
St\'{e}phane Benoist.
\newblock Natural parametrization of {SLE}: the {G}aussian free field point of
  view.
\newblock {\em Electronic Journal of Probability}, 23, 2018.

\bibitem[Ber17]{Ber}
Nathana{\"e}l Berestycki.
\newblock {An elementary approach to Gaussian multiplicative chaos}.
\newblock {\em Electronic Communications in Probability}, 22, 2017.

\bibitem[BG22]{BG}
Nathana\"{e}l Berestycki and Ewain Gwynne.
\newblock Random walks on mated-{CRT} planar maps and {L}iouville {B}rownian
  motion.
\newblock {\em Communications in Mathematical Physics}, 395(2):773--857, 2022.

\bibitem[BN11]{SLEnotes}
Nathana\"el Berestycki and James~R. Norris.
\newblock Lectures on {S}chramm--{L}oewner {E}volution.
\newblock Available on the websites of the authors, 2011.

\bibitem[BP25]{BP}
Nathana{\"e}l Berestycki and Ellen Powell.
\newblock {\em {G}aussian free field and {L}iouville quantum gravity}.
\newblock Cambridge Studies in Advanced Mathematics. Cambridge University
  Press, 2025.

\bibitem[DMS21]{MOT}
Bertrand Duplantier, Jason Miller, and Scott Sheffield.
\newblock Liouville quantum gravity as a mating of trees.
\newblock {\em Ast\'{e}risque}, (427), 2021.

\bibitem[DS11]{DS}
Bertrand Duplantier and Scott Sheffield.
\newblock Liouville quantum gravity and {KPZ}.
\newblock {\em Inventiones Mathematicae}, 185(2):333--393, 2011.

\bibitem[GHS20]{GHS}
Ewain Gwynne, Nina Holden, and Xin Sun.
\newblock A mating-of-trees approach for graph distances in random planar maps.
\newblock {\em Probability Theory and Related Fields}, 177(3-4):1043--1102,
  2020.

\bibitem[GHSS21]{GHSS}
Christophe Garban, Nina Holden, Avelio Sep{\'u}lveda, and Xin Sun.
\newblock Liouville dynamical percolation.
\newblock {\em Probability Theory and Related Fields}, 180(3-4):621--678, 2021.

\bibitem[GMS21]{Tutte}
Ewain Gwynne, Jason Miller, and Scott Sheffield.
\newblock The {T}utte embedding of the mated-{CRT} map converges to {L}iouville
  quantum gravity.
\newblock {\em The Annals of Probability}, 49(4):1677--1717, 2021.

\bibitem[HMP10]{HMP}
Xiaoyu Hu, Jason Miller, and Yuval Peres.
\newblock Thick points of the {G}aussian free field.
\newblock {\em The Annals of Probability}, 38(2):896--926, 2010.

\bibitem[HP21]{HP}
Nina Holden and Ellen Powell.
\newblock Conformal welding for critical {L}iouville quantum gravity.
\newblock {\em Annales de l'Institut Henri Poincar\'{e} Probabilit\'{e}s et
  Statistiques}, 57(3):1229--1254, 2021.

\bibitem[HS23]{HScardy}
Nina Holden and Xin Sun.
\newblock Convergence of uniform triangulations under the {C}ardy embedding.
\newblock {\em Acta Mathematica}, 230(1):93--203, 2023.

\bibitem[Kah85]{Kahane}
Jean-Pierre Kahane.
\newblock Sur le chaos multiplicatif.
\newblock {\em Annales des Sciences Math\'{e}matiques du Qu\'{e}bec},
  9(2):105--150, 1985.

\bibitem[Kem17]{Kem17}
Antti Kemppainen.
\newblock {\em Schramm-{L}oewner evolution}, volume~24 of {\em SpringerBriefs
  in Mathematical Physics}.
\newblock Springer, Cham, 2017.

\bibitem[KS17]{KS}
Antti Kemppainen and Stanislav Smirnov.
\newblock Random curves, scaling limits and {L}oewner evolutions.
\newblock {\em The Annals of Probability}, 45(2):698--779, 2017.

\bibitem[KW16]{KW}
Antti Kemppainen and Wendelin Werner.
\newblock The nested simple conformal loop ensembles in the riemann sphere.
\newblock {\em Probability Theory and Related Fields}, 165:835--866, 2016.

\bibitem[Law08]{LawC}
Gregory Lawler.
\newblock {\em Conformally invariant processes in the plane}.
\newblock Number 114. American Mathematical Society, 2008.

\bibitem[LR15]{LawlerRezaei}
Gregory~F. Lawler and Mohammad~A. Rezaei.
\newblock Minkowski content and natural parameterization for the
  {S}chramm-{L}oewner evolution.
\newblock {\em The Annals of Probability}, 43(3):1082--1120, 2015.

\bibitem[LS11]{LawlerSheffield}
Gregory~F. Lawler and Scott Sheffield.
\newblock A natural parametrization for the {S}chramm-{L}oewner evolution.
\newblock {\em The Annals of Probability}, 39(5):1896--1937, 2011.

\bibitem[LW13]{LawlerWerness}
Gregory~F. Lawler and Brent~M. Werness.
\newblock Multi-point {G}reen's functions for {SLE} and an estimate of
  {B}effara.
\newblock {\em The Annals of Probability}, 41(3A):1513--1555, 2013.

\bibitem[LZ13]{LawlerZhou}
Gregory~F. Lawler and Wang Zhou.
\newblock {\it {SLE}} curves and natural parametrization.
\newblock {\em The Annals of Probability}, 41(3A):1556--1584, 2013.

\bibitem[MS20]{TBM1}
Jason Miller and Scott Sheffield.
\newblock Liouville quantum gravity and the {B}rownian map {I}: the {${\rm
  QLE}(8/3,0)$} metric.
\newblock {\em Inventiones Mathematicae}, 219(1):75--152, 2020.

\bibitem[MS23]{MS4}
Vlad Margarint and Lukas Schoug.
\newblock A {G}aussian free field approach to the natural parametrisation of
  {SLE}$_4$.
\newblock {\em Electronic Communications in Probability}, 28, 2023.

\bibitem[RS05]{RohdeSchramm}
Steffen Rohde and Oded Schramm.
\newblock Basic properties of {SLE}.
\newblock {\em Ann. of Math. (2)}, 161(2):883--924, 2005.

\bibitem[RV10]{RV}
Raoul Robert and Vincent Vargas.
\newblock Gaussian multiplicative chaos revisited.
\newblock {\em The Annals of Probability}, 38(2):605--631, 2010.

\bibitem[She16]{Sheffield}
Scott Sheffield.
\newblock Conformal weldings of random surfaces: {SLE} and the quantum gravity
  zipper.
\newblock {\em The Annals of Probability}, 44(5):3474--3545, 2016.

\bibitem[SSV22]{SSV}
Lukas Schoug, Avelio Sep{\'u}lveda, and Fredrik Viklund.
\newblock Dimensions of two-valued sets via imaginary chaos.
\newblock {\em International Mathematics Research Notices}, 2022(5):3219--3261,
  2022.

\bibitem[SW16]{SheffieldWang}
Scott Sheffield and Menglu Wang.
\newblock Field-measure correspondence in {L}iouville quantum gravity almost
  surely commutes with all conformal maps simultaneously.
\newblock {\em Preprint arXiv:1605.06171}, 2016.

\bibitem[Zha08]{reversibility}
Dapeng Zhan.
\newblock Reversibility of chordal {SLE}.
\newblock {\em Ann. Probab.}, 36(4):1472--1494, 2008.

\end{thebibliography}

\end{document}